\documentclass[12pt,twoside]{article}

\usepackage[margin=1.5in]{geometry}

\usepackage[displaymath, mathlines]{lineno}

\usepackage{amsmath}
\usepackage{amsthm}
\usepackage{amssymb}
\usepackage{graphicx}
\usepackage{hyperref}
\usepackage{caption}
\usepackage{enumitem}
\usepackage{fancyhdr}
\usepackage[shortcuts]{extdash}
\usepackage{url}

\numberwithin{equation}{section}

\newenvironment{cproof}[0]{\noindent\emph{Proof of Claim}:}{ $\hfill\clubsuit$\smallskip}
\newtheorem{thm}{Theorem}[section] \newtheorem{cor}[thm]{Corollary} \newtheorem{lem}[thm]{Lemma}
\newtheorem{prop}[thm]{Proposition}

\newtheorem{mainthm}{Theorem}
\newtheorem{mainconj}{Conjecture}

\theoremstyle{definition}

\theoremstyle{remark}

\newtheorem{claim}{Claim}

\newtheorem*{claim*}{Claim}

 \newcommand{\FF}{\mathbb F}  
   
\newcommand{\Gg}{\mathcal G}   \newcommand{\Tt}{\mathcal T}
  \newcommand{\Cc}{\mathcal C} 
 \newcommand{\Rr}{\mathcal R}  
 \newcommand{\Bb}{\mathcal B}  
 \newcommand{\Ss}{\mathcal S}  
 
\newcommand{\comout}[1]{}

\renewcommand{\Pr}{\mathsf{P}}
\newcommand{\Om}{\Omega}
\newcommand{\GB}{(G,\Bb)}
\newcommand{\GBu}{(\widehat{G},\widehat{\Bb})}

\newcommand{\fI}{\varphi}
\newcommand{\inv}{^{-1}}
\newcommand{\ie}{\textit{i.e.}\ }

\newcommand{\iiff}{if and only if }
\newcommand{\wolog}{without loss of generality}

\DeclareMathOperator{\cl}{cl}

\newcommand{\longcomment}[1]{}
\newcommand{\br}[1]{\left( #1 \right)}
\renewcommand{\:}{\colon}

\newcommand{\mc}[0]{\mathcal}

\newcommand{\bb}[0]{\mathbb}
\newcommand{\bs}[0]{\backslash}
\newcommand{\vp}[0]{\varphi}
\newcommand{\tail}[0]{{\mathbf{t}}}
\newcommand{\head}[0]{\mathbf{h}}
\newcommand{\ab}{\allowbreak}
\DeclareMathOperator{\Star}{star}

\title{Matrix representations of frame and lifted-graphic matroids correspond to gain functions}
\author{Daryl Funk\footnote{Department of Mathematics, Douglas College, New Westminster, British Columbia, Canada.
\underline{Email}: funkd@douglascollege.ca}
 \and
Irene Pivotto\footnote{Department of Mathematics and Statistics, University of Western Australia, Perth, Australia}
 \and
Daniel Slilaty\footnote{Department of
Mathematics and Statistics, Wright State University, Dayton OH, USA. \underline{Email}: daniel.slilaty@wright.edu. Work
partially supported by a grant from the Simons Foundation \#246380.}}

\date{} 

\begin{document}


\maketitle

\begin{abstract}
Let $M$ be a 3-connected matroid and let $\mathbb F$ be a field.
Let $A$ be a matrix over $\mathbb F$ representing $M$ and let $(G,\mathcal B)$ be a biased graph representing $M$.
We characterize the relationship between $A$ and $(G,\mathcal B)$, settling four conjectures of Zaslavsky.
We show that for each matrix representation $A$ and each biased graph representation $\GB$ of $M$, $A$ is projectively equivalent to a canonical matrix representation arising from $G$ as a gain graph over $\mathbb F^+$ or $\mathbb F^\times$ realizing $\Bb$.
Further, we show that the projective equivalence classes of matrix representations of $M$ are in one-to-one correspondence with the switching equivalence classes of gain graphs arising from $(G,\mathcal B)$, except in one degenerate case.
\end{abstract}

\section{Introduction} \label{sec:intro}

\renewcommand{\thefootnote}{\fnsymbol{footnote}} 
\footnotetext{\emph{Key words:} frame matroids, lifted-graphic matroids, representable matroids, gain graphs, group-labelled graphs}     
\renewcommand{\thefootnote}{\arabic{footnote}} 

An amazing theorem of Kahn \& Kung \cite{MR654846} shows that there are only two matroid varieties containing 3-connected matroids: those given by (1) projective geometries over finite fields, and (2) Dowling geometries over finite groups.
The matroids belonging to both of these varieties are precisely the simple frame matroids linearly representable over finite fields.
Thus it is perhaps not surprising that these matroids have a central role in matroid structure theory  \cite{MR2459453, GeelenGerardsWhittle:HighlyConnected}.
In this paper we characterize the relationship between the matrix and biased-graphic representations of the 3-connected matroids that are members of both of these important classes.
While on the surface there is no obvious reason that there should be any such relationship, these representations are in fact inherently linked.

 For a vertex $v$ in a graph $G$, denote by $\Star(v)$
the set of edges incident to $v$ whose other end is in $V (G)-\{v\}$ (thus $\Star(v)$ contains no loops).
Let $M$ be a matroid.
Suppose there is a graph $G$ with $E(G)=E(M)$ such that
(i) the edge set of each component of $G$ has rank in $M$ no larger than the number of its vertices,
(ii) for each vertex $v \in V(G)$, $\cl_M(E(G-v)) \subseteq E(G) - \Star(v)$, and
(iii) no circuit of $M$ induces a subgraph of $G$ with more than two components.
Such a graph $G$ is a \emph{framework} for $M$; a  matroid that has a framework is \emph{quasi-graphic} (\cite{JGT22177} is the introductory paper).
Setting $\Bb = \{ C : C$ is a cycle of $G$ and a circuit of $M\}$ yields a \emph{biased graph} $(G,\Bb)$: that is, a graph together with a distinguished collection $\Bb$ of its cycles obeying a certain condition (precise definitions are given in Section \ref{sec:preliminaries}).
Cycles in $\Bb$ are said to be \emph{balanced}, and are otherwise \emph{unbalanced}.

Suppose $M$ also has a representation as a matrix $A$ over a field $\FF$: the columns of $A$ are indexed by $E(M)$ such that a subset of elements of $M$ is independent precisely when its columns are linearly independent. That is, $M$ is \emph{linearly representable over $\FF$}; we write $M=M(A)$.
It is shown in \cite{MR4037634} (and in \cite{JGT22177} in the case that $M$ is 3-connected) that when $M$ is both linearly representable over a field and quasi-graphic, then $M$ is either a \emph{frame matroid} or a \emph{lifted-graphic matroid}.
When $M$ is frame or lifted-graphic, the biased graph $(G,\Bb)$
completely determines $M$.
We write $M = F\GB$ when $M$ is a frame matroid, and $M = L\GB$ when $M$ is lifted-graphic, and say the biased graph $(G,\Bb)$ \emph{represents} $M$.

We ask the question: given a 3-connected matroid $M$, a matrix $A$ over a field, and a biased graph $\GB$, such that $M = M(A) = F\GB$ (respectively, $M = M(A) = L\GB$), what relationship is there between $A$ and $\GB$?
In the remainder of this introduction we provide a brief high-level overview of our answer; precise definitions are given in Section \ref{sec:preliminaries}.

\subsection{Projective equivalence, gain graphs, and canonical representations} \label{subsec:gaingraphscanonicalreps}

A pair of matrices representing a matroid over a field $\FF$ are \emph{projectively equivalent} if one may be obtained from the other via a sequence of the following elementary operations:
interchange two rows, multiply a row by a nonzero element of $\FF$, replace a row by its sum with another row, add or remove a zero row, interchange two columns (along with their labels), or multiply a column by a nonzero element of $\FF$.
(A pair of matrices are generally said to be \emph{equivalent} if one may be obtained from the other via such a sequence along with possibly applying an automorphism of $\FF$ to the entries of one of the matrices; our results do not require field automorphisms.  Projective equivalence has sometimes been called \emph{strong equivalence} in the literature.)
Our main tool for showing projective equivalence of a pair of matrix representations is the
following well-known result of Brylawski and Lucas.

\begin{prop}[\cite{oxley:mt2011}, Proposition 6.3.12]
Let $A$ and $B$ be $r \times n$ matrices over a field with the columns of each labelled, in order, by $e_1$, $e_2$, \ldots, $e_n$, with $r \geq 1$.
Then $A$ and $B$ are projectively equivalent representations of a matroid on $\{e_1, e_2, \ldots, e_n\}$ if and only if there is a non-singular matrix $T$ and a non-singular diagonal matrix $S$ such that $TAS=B$.
\end{prop}

The connection between matrix and biased graph representations is provided by \emph{gain graphs} (also known as \emph{group-labelled graphs}).
We show that if a matroid $M$ has a representation as a matrix over a field $\FF$ and has a representation as a biased graph $\GB$, then an element of $\FF$ may be assigned to each edge of the graph $G$ such that $G$ becomes a gain graph, over either the additive or multiplicative group of $\FF$, realizing $\Bb$.
Conversely, Zaslavsky \cite{Zaslavsky:BG4} showed that every such gain graph yields a \emph{canonical matrix representation} $A$ for $M$ over $\FF$, as follows.

When the gain graph is over the additive group of $\FF$, $A$ is the matrix with rows indexed by $V(G) \cup v_0$, where $v_0 \notin V(G)$, and columns indexed by $E(G)$ in which, if $e$ is assigned element $g \in \FF$, has distinct endpoints $u, v$, and is directed from $u$ to $v$, then entry $a_{v_0 e} = g$, $a_{ue} = 1$, $a_{ve} = -1$, and all remaining entries in column $e$ are zero; if $e$ is a loop then $a_{v_0 e} = g$ and all remaining entries in column $e$ are zero.
Thus $A$ consists of the oriented incidence matrix of $G$ with one additional row containing the elements of $\FF$ that are assigned to each edge $e \in E(G)$;
we say $A$ is a \emph{canonical lift matrix}.
\label{descriptioncanonicalframematrix}
When the gain graph is over the multiplicative group of $\FF$, $A$ is the matrix with rows indexed by $V(G)$ and columns indexed by $E(G)$ in which, if $e$ is assigned element $g \in \FF$, has distinct endpoints $u, v$, and is directed from $u$ to $v$, then $a_{ue} = 1$, $a_{ve} = -g$, and the remaining entries in column $e$ are zero; if $e$ is a loop incident to vertex $v$ then $a_{ve} = 1-g$ and all other entries in column $e$ are zero.
We say $A$ is a \emph{canonical frame matrix}.

An assignment of elements of a group $\Gamma$ to the oriented edges of a graph is a \emph{$\Gamma$-gain function}.
A graph equipped with a $\Gamma$-gain function is a \emph{$\Gamma$-gain graph}.
\emph{Switching}, and when $\Gamma$ is the additive group of a field, \emph{scaling}, are two operations that may be applied to a $\Gamma$-gain function to obtain a new $\Gamma$-gain function on the same graph.
These operations re-assign elements of $\Gamma$ to $E(G)$ in such a way that the collection of cycles in $\Bb$ determined by the new gain function remains unchanged.
Thus the biased graph determined by a pair of gain graphs, one resulting from a switching or scaling operation applied to the other, remains the same.
So it makes sense to consider switching and scaling as biased graphic analogs to the elementary matrix operations for projective equivalence.
In fact, as we shall see, the connection between switching and scaling on gain graphs and projectively equivalent matrix representations is deeper than merely analogous.

As will be clear by the citations in Section \ref{sec:preliminaries}, the notions of frame matroids, lifted-graphic matroids, their representations by biased graphs, gain graphs, their associated canonical matrix representations, and the operations of switching and scaling, are all due to Thomas Zaslavsky, presented in his seminal series of papers \cite{Zaslavsky:BG1, Zaslavsky:BG2, MR1273951, Zaslavsky:BG4}.

\subsection{Main results}

Switching and scaling naturally partition the set of $\Gamma$-gain functions on a graph into equivalence classes in which two gain functions are equivalent if one may be obtained from the other by switching,
or when $\Gamma$ is the additive group of a field, by switching and scaling.
When $\GB$ is a biased graph realized by a $\Gamma$-gain graph, these are the \emph{switching classes} of $\Gamma$-realizations of $\GB$.
When the group is the additive group $\FF^+$ or multiplicative group $\FF^\times$ of a field $\FF$, it is straightforward to show that gain functions belonging to the same switching class yield projectively equivalent canonical matrix representations \cite{Zaslavsky:BG4}.
In \cite{Zaslavsky:BG4} Zaslavsky conjectured that the converse also holds.
Leaving aside some technicalities,
Conjectures 2.8 and 4.8 of \cite{Zaslavsky:BG4} are essentially as follows.

\begin{mainconj}[Zaslavsky \cite{Zaslavsky:BG4}] \label{Zconj28}
Let $G$ be a graph and let $\FF$ be a field.
The canonical frame matrices given by two $\FF^\times$-gain functions $\vp$ and $\psi$ on $G$ are projectively equivalent if and only if $\vp$ and $\psi$ are switching equivalent.
\end{mainconj}

\begin{mainconj}[Zaslavsky \cite{Zaslavsky:BG4}] \label{Zconj48}
Let $G$ be a graph and let $\FF$ be a field.
The canonical lift matrices given by two $\FF^+$-gain functions $\vp$ and $\psi$ on $G$ are projectively equivalent if and only if $\vp$ and $\psi$ are switching-and-scaling equivalent.
\end{mainconj}

A \emph{joint} is an unbalanced loop in a biased graph.
It is a property of the switching operation that the gain on a joint remains unchanged by every switching operation (provided $\Gamma$ is abelian; more precisely, the gain on a joint is conjugated by switching).
Thus one of the technicalities to be dealt with in
Conjectures \ref{Zconj28} and \ref{Zconj48} is the following.
Let $\vp$ and $\psi$ be $\FF^\times$-gain functions on a graph $G$.
Suppose $\vp(e)=\psi(e)$ for every edge $e$ that is not a joint, but that there is a joint $e'$ for which $\vp(e') \neq \psi(e')$ and neither $\vp(e')$ nor $\psi(e')$ is zero.
Since a switching operation can alter neither $\vp(e')$ nor $\psi(e')$, $\vp$ and $\psi$ are not switching equivalent.
Yet clearly the canonical matrices defined by $\vp$ and $\psi$
are projectively equivalent, since each of their columns representing the element $e'$ may be scaled so that the single nonzero entry in each is equal to any element of $\FF^\times$.
Our first main result is that for gain graphs representing 3-connected matroids,
this issue with gains on loops provides the only counterexamples to Conjectures \ref{Zconj28} and \ref{Zconj48}.

\begin{mainthm} \label{mainthm1a}
Let $M$ be a 3-connected matroid of rank greater than two. Let $\GB$ be a loopless biased graph representing $M$ and let $\FF$ be a field.

\textup{(i)} The canonical frame matrices given by two $\FF^\times$-gain functions $\vp$ and $\psi$ realizing $\GB$ are projectively equivalent if and only if $\vp$ and $\psi$ are switching equivalent.

\textup{(ii)} The canonical lift matrices given by two $\FF^+$-gain functions $\vp$ and $\psi$ realizing $\GB$ are projectively equivalent if and only if $\vp$ and $\psi$ are switching-and-scaling equivalent.

\textup{(iii)} Provided $\GB$ has no vertex whose deletion leaves a biased graph with no unbalanced cycles, no canonical frame matrix representation is projectively equivalent to any canonical lift matrix representation of $M$.
\end{mainthm}

In fact, we can say more.
Theorem \ref{mainthm1a} follows from the stronger statements of our Theorems \ref{T:ProjectiveIsSwitching} and \ref{thm:proj_equiv_iff_switching_equiv_bal_vertex}.
Together these also give necessary and sufficient conditions for projective equivalence of a canonical lift matrix and a canonical frame matrix representation.
Examples that are not 3-connected for which the conclusions of Theorem \ref{mainthm1a} do not hold are not difficult to construct, so the hypothesis that $M$ be 3-connected is necessary. 

Let $M$ be a linearly representable frame or lifted-graphic matroid.
Not only may $M$ have projectively inequivalent matrix representations arising as canonical matrix representations from the same biased graph, but there may also be different biased graphs representing $M$.
We say a matrix representation is \emph{particular to} the biased graph $\GB$ when it is a canonical matrix arising from a gain function realizing $\GB$.
Zaslavsky has conjectured the following.

\begin{mainconj}[\cite{Zaslavsky:BG2}]
\label{con:Z3}
Let $\GB$ be a biased graph, where $G$ is sufficiently connected, and let $\FF$ be a field.
If $F\GB$ (resp., $L\GB$) is linearly representable over $\FF$ then $F\GB$ (resp., $L\GB$) has a canonical representation particular to $\GB$.
\end{mainconj}
Zaslavsky subsequently further conjectured the following.

\begin{mainconj}[Zaslavsky, personal communication] \label{con:Z4}
Let $\GB$ be a biased graph, where $G$ is sufficiently connected, and let $\FF$ be a field.
Every $\FF$\-/representation of $F\GB$ (respectively $L\GB$) is projectively equivalent to a canonical representation.
\end{mainconj}

Geelen, Gerards, and Whittle prove Conjecture \ref{con:Z4} for 3-connected matroids in \cite{JGT22177}, though the result is not explicitly stated (it appears in the proof of their Theorem 1.4).
Surprisingly, 
a result even stronger than Conjectures \ref{con:Z3} and \ref{con:Z4} holds:

\begin{mainthm} \label{mainthm2}
Let $M$ be a 3-connected matroid, and let $\FF$ be a field.
Let $A$ be a matrix over $\FF$ representing $M$ and let $\GB$ be a biased graph representing $M$.
Then $A$ is projectively equivalent to a canonical representation particular to $\GB$.
\end{mainthm}

Theorem \ref{mainthm2} follows from the stronger statement of Theorem \ref{thm:all_reps_eqiv_to_cannonical}.
Together Theorems \ref{T:ProjectiveIsSwitching},
\ref{thm:proj_equiv_iff_switching_equiv_bal_vertex}, and \ref{thm:all_reps_eqiv_to_cannonical} imply our next main result.
We need just a few more definitions before we can state it precisely.
A biased graph is \emph{balanced} when all of its cycles are balanced, \emph{almost-balanced} when it is not balanced but there is a vertex that is contained in every unbalanced cycle of length at least two, and \emph{properly unbalanced} otherwise.
A properly unbalanced biased graph with no pair of vertex-disjoint unbalanced cycles is \emph{tangled}.
If $\GB$ is balanced, then both $F\GB$ and $L\GB$ are equal to the cycle matroid $M(G)$ of $G$.
Since $M(G)$ has a projectively unique matrix representation over every field, and every gain function realizing a balanced biased graph is switching equivalent to the gain function assigning the group identity to every element (this follows from Proposition \ref{P:Normalize1} below), Conjectures \ref{Zconj28}-\ref{con:Z4} hold in this case.
Thus we just need consider almost-balanced and properly unbalanced biased graphs.
The collection of unbalanced cycles in almost-balanced biased graphs is highly structured and well understood (see \cite{AlmostBalancedReps}).
For each almost-balanced biased graph $\GB$ there is
a family of almost-balanced biased graphs $\Rr_{\GB}$, each of which represents the frame matroid $F\GB$.
We denote the unique biased graph in $\Rr_{\GB}$ with the least number of loops by $(\widehat{G},\widehat{\Bb})$; this is also the unique biased graph in the collection for which $F\GBu = L\GBu$.

Let $M$ be a 3-connected matroid with rank greater than two.
Let $\FF$ be a field and let $\GB$ be a biased graph representing $M$.
Let $\Ss_{\FF^\times}\GB$ denote the collection of switching classes of $\FF^\times$-gain functions realizing $\GB$,
and let $\Ss_{\FF^+}\GB$ denote the collection of switching-and-scaling classes of $\FF^+$-gain functions realizing $\GB$.
For each biased graph representing $M$ and each field $\FF$, define
\begin{linenomath}
\[
\Ss_{\FF}\GB =
\begin{cases}
\Ss_{\FF^\times}\GB &\text{if } M = F\GB \neq L\GB, \\
\Ss_{\FF^+}\GB &\text{if } M = L\GB \neq F\GB, \\
\Ss_{\FF^\times}\GB \cup \Ss_{\FF^+}\GB &\text{if $\GB$ is tangled,} \\
\Ss_{\FF^+}\GBu &\text{if $\GB$ is almost-balanced.}
\end{cases}
\]
\end{linenomath}

\begin{mainthm} \label{mainthm1to1correspondence}
Let $M$ be a 3-connected matroid of rank greater than two.
Let $\FF$ be a field and let $\GB$ be a loopless biased graph representing $M$.
The projective equivalence classes of matrices over $\FF$ representing $M$ are in one-to-one correspondence with the switching classes of gain functions in $\Ss_{\FF}\GB$.
\end{mainthm}

The proofs of our main results use an inductive argument.
For this purpose we
determine a small collection of biased graphs, at least one of which must occur as a biased topological subgraph in every 2-connected biased graph.
This investigation yields three results on unavoidable minors and topological minors of biased graphs that are of independent interest.

The graph obtained from a 3-cycle by replacing each edge with a pair of parallel edges is denoted by $2C_3$.
The graph obtained from a 4-cycle by replacing each edge in a pair of non-adjacent edges with a pair of parallel edges is the \emph{tube graph}, denoted by $2C_4''$.
We show that the collection $\mathcal{G}_0$ consisting of the six biased $2C_3$'s with no balanced 2-cycle, the three biased tubes with no balanced 2-cycle, and the four biased $K_4$'s with no balanced triangle, is the set of biased graphs that are minor-minimal amongst all 2-connected, properly unbalanced biased graphs.
Contraction of a loop is never required to obtain one of these minors; such a minor is a \emph{link minor}.

\begin{mainthm} \label{MT:UnavoidableMinors}
Every 2-connected properly unbalanced biased graph contains a biased graph in $\Gg_0$ as a link minor.
\end{mainthm}

We prove an analogue of Theorem \ref{MT:UnavoidableMinors} for biased topological subgraphs.
Let $\Pr$ denote the biased graph consisting of the triangular prism with just its two triangles balanced (Figure \ref{fig:T2PrimeSplit} on page \pageref{fig:T2PrimeSplit}).
Let $\Tt_0$ be the set of biased graphs consisting of those in $\Gg_0$ together with $\Pr$ and the two biased graphs obtained from $\Pr$ by contracting 1 and 2 edges, respectively, of the matching in $\Pr$ linking its two triangles.

\begin{mainthm} \label{MT:Unavoidsubdivisions}
Every 2-connected properly unbalanced biased graph contains as a biased subgraph a subdivision of a biased graph in $\Tt_0$.
\end{mainthm}

Theorems \ref{MT:UnavoidableMinors} and \ref{MT:Unavoidsubdivisions} are useful because switching inequivalence of gain functions over abelian groups can always be found
(in the interesting case that there are no joints) on a small minor.
This is the content of our final main result.
Two biased graphs representing the 4-point line $U_{2,4}$ as a frame or lifted-graphic matroid are shown in Figure \ref{fig:U24unlabelled} on page \pageref{fig:U24unlabelled}; we denote these biased graphs by $\mathsf U_2$ and $\mathsf U_3$.

\begin{mainthm} \label{MT:InequivalenceLocalized}
Let $(G,\mathcal B)$ be a 2-connected, loopless, and properly unbalanced biased graph.
Let $\Gamma$ be an abelian group and let $\vp$ and $\psi$ be $\Gamma$\-/realizations of $\GB$.
Assume that $\vp$ and $\psi$ are not switching equivalent.
In the case that $\Gamma$ is the additive group of a field, assume that $\vp$ and $\psi$ are not switching-and-scaling equivalent.
Then either

\begin{enumerate}[label=\textup{(\roman*)}]
\item
$\GB$ has a link minor $(H,\Ss) \in \Gg_0$ such that the gain functions induced by $\vp$ and $\psi$ on $E(H)$ are not switching equivalent (resp., not switching-and-scaling equivalent), or

\item
$\GB$ has $\mathsf U_2$ as a minor and $\mathsf U_3$ as a link minor such that the gain functions induced by $\vp$ and $\psi$ are not switching equivalent (resp., not switching-and-scaling equivalent) on the 2-cycle of $\mathsf U_2$ nor on the theta subgraph of $\mathsf U_3$.
\end{enumerate}
\end{mainthm}

The remainder of the paper is structured as follows.
In Section \ref{sec:preliminaries} we provide necessary preliminary notions.
In Section \ref{sec:unavoidableminorsandsubdivisions} we prove Theorems \ref{MT:UnavoidableMinors},  \ref{MT:Unavoidsubdivisions}, and \ref{MT:InequivalenceLocalized} on unavoidable minors and unavoidable biased topological subgraphs.
In Section \ref{sec:unavoidableminors} we show that Theorems \ref{mainthm1a} and \ref{mainthm2}
hold for the set of unavoidable minors, and finally in Section \ref{sec:MainThmProofs} we prove Theorems \ref{mainthm1a},  \ref{mainthm2}, and \ref{mainthm1to1correspondence}.

\section{Preliminaries} \label{sec:preliminaries}

We assume that the reader is familiar with matroid theory as in Oxley's standard text \cite{oxley:mt2011}.
In this section we summarize those notions that are central to the results of this paper or are not standard in the literature, and introduce some required notation.

\subsection{Graphs and biased graphs} \label{sec:graphsandbiaedgraphs}

Let $G$ be a graph.
We denote the subgraph of $G$ induced by a subset $X\subseteq E(G)$ by $G[X]$.
The set of vertices of $G[X]$ is denoted by $V(X)$.
A $k$\-/\emph{separation} of $G$ is a partition $(A,B)$ of $E(G)$ with $|A|\geq k$, $|B|\geq k$, and
$|V(A)\cap V(B)|=k$.
A \emph{vertical $k$\-/separation} of $G$ is a $k$\-/separation $(A,B)$ of $G$ with both $V(A)-V(B)$ and $V(B)-V(A)$ non-empty.
A graph on at least $k+2$ vertices is \emph{$k$\-/connected} if it has no vertical $l$\-/separation for any $l<k$.
A graph on $k+1$ vertices is said to be $k$\-/connected if it has a spanning complete subgraph.
Thus a highly connected graph may contain loops or parallel edges.
We often need to distinguish between edges with distinct endpoints and loops; an edge with distinct endpoints is a \emph{link}.

A \emph{biased graph} is a pair $(G, \mathcal{B})$ where $G$ is a graph and $\mathcal{B}$ is a collection of cycles of $G$ with the property that no theta subgraph of $G$ contains
exactly two cycles in $\Bb$;
a \emph{theta} graph is the union of three internally disjoint paths linking a pair of vertices.
Such a collection $\Bb$ is said to satisfy the \emph{theta property}.
Cycles in $\Bb$ are \emph{balanced}; cycles not in $\Bb$ are \emph{unbalanced}.
A biased graph is
\emph{balanced} if all cycles are balanced,
\emph{unbalanced} if it contains an unbalanced cycle and  \emph{contrabalanced} if no cycle is balanced.
Similarly, a subset of edges or a subgraph is \emph{balanced}, \emph{unbalanced}, or \emph{contrabalanced}, according to whether all, not all, or none, respectively, of the cycles it induces or contains are balanced.
A vertex $v$ is a \emph{balancing vertex} if every unbalanced cycle contains $v$.
A biased graph $(G,\mc B)$ is \emph{$k$\-/connected} if $G$ is $k$\-/connected.
For two biased graphs $(G,\mc B)$ and $(H,\mc S)$ an \emph{isomorphism} $\iota\colon(G,\mc B)\to(H,\mc S)$ consists of an underlying graph isomorphism $\iota\colon G\to H$ that takes $\mc B$ to $\mc S$.
We sometimes write $\Om = \GB$ and speak of the biased graph $\Om$ when there is no need to be explicit about the underlying graph $G$ and its collection of balanced cycles $\Bb$.

We denote the set of all cycles in $G$ by $\mc C(G)$.
Let $(G,\mc B)$ be a biased graph and $e$ an edge in $G$. Define $(G,\mc B)\bs e=(G\bs e,\mc B|_{G\bs e})$ where $\mc
B|_{G\bs e}=\mc B\cap\mc C(G\bs e)$. If $e$ is a link, then define $(G,\mc B)/e=(G/e,\mc B|_{G/e})$ where $\mc
B|_{G/e}=\{C\in\mc C(G/e):C\in\mc B\mbox{ or }C\cup e\in\mc B\}$. If $e$ is a balanced loop, then $(G,\mc B)/e=(G,\mc B)\bs e$.
In order that contraction of a joint $e$ of $\GB$ remain consistent with the operation of contraction of $e$ in the lifted-graphic or frame matroid represented by $\GB$, two different contraction operations in $\GB$ are required, depending upon which matroid $\GB/e$ is to represent.
Let $e$ be a joint of $\GB$, incident to vertex $v$.
To obtain a biased graph  representing $L\GB/e$, define $\GB/e = (G \bs e, \Cc(G \bs e))$.
To obtain a biased graph representing $F\GB/e$, define
$(G,\mc B)/e=(G',\mc B')$ where $G'$ is obtained from $G$ by adding each loop $e'\neq e$ incident to $v$ to $\Bb$ and replacing each link $f$ incident to $v$ with a joint incident to its other endpoint.
The collection $\mc B'$ is $\mc B$ restricted to the subgraph $G-v$ along with any new balanced loops incident to $v$.
Whenever contracting a joint, we will be explicit about which contraction operation is used, \emph{lift-type} or \emph{frame-type}, respectively.
In fact, we will only require contraction of a joint once; this will be a frame-type contraction.
All other minors we consider may be obtained
without contracting joints.
Such a minor is a \emph{link minor}.
We permit deletion of isolated vertices and do so without comment.

\subsection{Matroids arising from biased graphs} \label{sec:matorids_of_biased_graphs}

A \emph{frame matroid} is a matroid $M$ to which a basis $B$ may be added such that each $e\in E(M)$ is contained in the closure of some 2-element subset of $B$.
A subset $C\subseteq E(G)$ is a circuit of the frame matroid $F(G,\mc B)$ precisely when $C \in \Bb$ or $C$ induces a subdivision of one of the graphs shown in Figure \ref{fig:FrameCircuits} containing no balanced cycle \cite{Zaslavsky:BG2}.
\begin{figure}[tbp] \begin{center} \includegraphics[page=1,height=30pt]{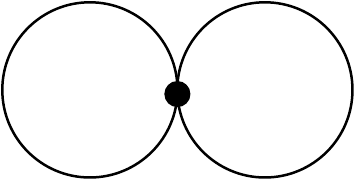}\hspace{1.25cm}
\includegraphics[page=2,height=30pt]{Circuits}\hspace{1.25cm}\includegraphics[page=4,height=35pt]{Circuits} \end{center}
\caption{Circuits of the frame matroid.} \label{fig:FrameCircuits} \end{figure}
A matroid $M$ is \emph{lifted-graphic} if there is a single-element extension $M_0$ of $M$ by an element $e_0$ such that $M_0/e_0$ is graphic.
Dualizing Crapo's characterization \cite{Crapo:Single-ElementExtensions} of single-element extensions of matroids, Zaslavsky observed \cite{Zaslavsky:BG2} that if $G$ is a graph such that $M_0/e_0=M(G)$
then $M$ has a natural description in terms of a biased graph $\GB$.
A subset $C\subseteq E(G)$ is a
circuit of the lifted-graphic matroid $L(G,\mc B)$ precisely when $C \in \Bb$ or $C$ induces a subdivision of one of the subgraphs shown in Figure \ref{fig:LiftCircuits} containing no balanced cycle \cite{Zaslavsky:BG2}.
Minors of biased graphs and their matroids agree:
for any edge $e$ of a biased graph $\GB$,
$F(G,\mc B)\bs e = F((G,\mc B)\bs e)$,
$F(G,\mc B)/e = F((G,\mc B)/e)$,
$L(G,\mc B)\bs e = L((G,\mc B) \bs e)$,
for any link $e$, $L(G,\mc B)/e = L((G,\mc B)/e)$, and if $e$ is a joint then $L(G,\mc B)/e = M(G \bs e)$
\cite{Zaslavsky:BG2}.

\begin{figure}[tbp] \begin{center}
\includegraphics[page=1,height=30pt]{Circuits}\hspace{1.25cm}
\includegraphics[page=3,height=30pt]{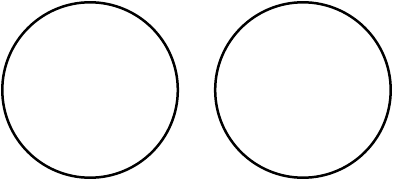}\hspace{1.25cm}\includegraphics[page=4,height=35pt]{Circuits}
\end{center}
\caption{Circuits of the lift matroid.} \label{fig:LiftCircuits}
\end{figure}

Evidently whenever $\GB$ is a biased graph with no two vertex-disjoint unbalanced cycles, $F\GB = L\GB$.
For notational convenience and to avoid potential confusion, in the case that $\GB$ does not have a pair of vertex-disjoint unbalanced cycles we denote the matroid on ground set $E(G)$ that is equal to both $F\GB$ and $L\GB$ by $M\GB$.

\subsection{Gain graphs}

A \emph{gain graph} is obtained from a graph by assigning a direction and an element of a group (a \emph{gain}) to each edge of $G$ \cite{Zaslavsky:BG1}.
Let $G$ be a graph and let $\Gamma$ be a group.
Orient the edges of $G$ by arbitrarily assigning a direction to each edge: for each edge $e \in E(G)$ choose one of its ends for its \emph{tail} $\tail(e)$.
The other end of $e$ is its \emph{head} $\head(e)$; if $e$ is a loop then $\head(e)=\tail(e)$.
We specify an oriented edge $e$ by the ordered triple $(e;u,v)$ where $u=\tail(e)$ and $v=\head(e)$; then $e\inv$ is the ordered triple $(e;v,u)$.
We think of $e\inv$ as the oriented edge $e$ traversed in the reverse direction.
The collection of ordered triples $\{(e; u, v) : e \in E(G)$ has ends $u, v \}$, consisting of all oriented edges of $G$ and their inverse orientations is denoted by $\vec E(G)$.
As long as no confusion may arise, we write $e$ for both the edge $e \in E(G)$ and for an ordered triple $(e; u,v) \in \vec E(G)$ specifying a direction of traversal of $e$.
Let $\Gamma$ be a group, and let $\gamma : \vec E(G) \to \Gamma$ be a function satisfying the following condition:
for each link $e$, if $\Gamma$ is multiplicative then $\gamma(e\inv) = \gamma(e)\inv$ while if $\Gamma$ is additive then $\gamma(e\inv) = -\gamma(e)$.
Such a function is a \emph{$\Gamma$\-/gain function on} $G$; the pair $(G,\gamma)$ is a \emph{$\Gamma$\-/gain graph}.

A gain graph $(G,\gamma)$ naturally gives rise to a biased graph $(G,\Bb_\gamma)$, where the membership of each cycle in the collection of balanced cycles $\Bb_\gamma$ is determined by the gains assigned to its edges, as follows.
Let $W = (v_1,e_1,v_2, \ldots, v_n, e_n, \ab v_{n+1})$, where for each $i$, $v_i, v_{i+1}$ are the ends of edge $e_i$, be a walk in $G$.
Define $\gamma(W) = \gamma(e_1; v_1, v_2) \gamma(e_2; v_2, v_3) \cdots \gamma(e_n; v_n, v_{n+1})$ when $\Gamma$ is multiplicative and $\gamma(W) = \gamma(e_1; v_1, v_2) + \gamma(e_2; v_2, v_3) + \cdots + \gamma(e_n; v_n, v_{n+1})$ when $\Gamma$ is additive.
For each cycle $C$ of $G$ choose a closed Eulerian walk $W_C$ in $C$, and define $C$ to be \emph{balanced} with respect to $\gamma$ if $\gamma(W_C)$ is the identity element of $\Gamma$.
Let $\mc B_\gamma$ be the collection of cycles in $G$ that are
balanced with respect to $\gamma$.
Observe that if $W'_C$ is another closed Eulerian walk in $C$ then $\gamma(W_C)$ and $\gamma(W'_C)$ are conjugate.
Thus $\Bb_\gamma$ is well-defined.
If $\GB$ is a biased graph and $\gamma$ is a $\Gamma$-gain function such that $\Bb_\gamma = \Bb$, then we say $\gamma$ \emph{realizes} $\GB$ and that $\gamma$ is a \emph{$\Gamma$\-/realization} of $\GB$.

\paragraph{Switching and scaling.}
Given a $\Gamma$\-/gain function $\gamma$ on $G$ and a function $\eta\colon V(G)\to\Gamma$, define the gain function
$\gamma^\eta$ by $\gamma^\eta(e) = \eta(\tail(e))\inv \cdot \gamma(e)\cdot \eta(\head(e))$ if $\Gamma$ is multiplicative and by
$\gamma^\eta(e)=-\eta(\tail(e))+\gamma(e)+\eta(\head(e))$ if $\Gamma$ is additive.
The function $\eta$ is a \emph{switching function}.
It is straightforward to check that a cycle $C$ is balanced with respect to $\gamma$ if and only if $C$ is balanced with respect to $\gamma^\eta$, so $\mc B_\gamma=\mc B_{\gamma^\eta}$.
Observe also that for switching functions $\eta_1$ and $\eta_2$, $(\gamma^{\eta_1})^{\eta_2}=\gamma^{\eta_1\eta_2}$ or $\gamma^{\eta_1+\eta_2}$.
Two $\Gamma$\-/gain functions $\vp$ and $\psi$ are \emph{switching equivalent} if there is a switching function $\eta$ such that $\vp^\eta=\psi$.
In the case that $\Gamma$ is the additive group of a field $\FF$, we may choose any nonzero element $a \in \FF$ and obtain a new gain function $a\vp$ defined by $a\cdot\vp(e)$ for each edge $e$.
Clearly $\Bb_{a\vp} = \Bb_{\vp}$.
We say $a\vp$ is obtained from $\vp$ by \emph{scaling}.
When $\Gamma$ is the additive group of a field, we say two gain functions $\vp$ and $\psi$ are \emph{switching-and-scaling equivalent} if there is a switching function $\eta$ and scalar $a\in\bb F^\times$ such that $a\vp^\eta=\psi$.
Evidently, for a multiplicative group $\Gamma$ the relation of being switching equivalent partitions the collection of $\Gamma$\-/gain functions on a graph into equivalence classes, its \emph{switching classes}.
Similarly, when $\Gamma$ is the additive group of a field the relation the being switching-and-scaling equivalent partitions the collection of $\Gamma$\-/gain functions on a graph into its \emph{switching-and-scaling classes}.
Propositions \ref{P:Normalize1} and \ref{P:NormalizationUnique} are immediate.

\begin{prop} \label{P:Normalize1}
Let $F$ be a forest in a graph $G$ and let $\gamma$ be a $\Gamma$\-/gain function on $G$.
There is a switching function $\eta$ such that $\gamma^\eta(e)$ is the identity for all edges $e$ in $F$.
\end{prop}

Let $F$ be a forest of a graph.
A gain function $\gamma$ is said to be $F$\-/\emph{normalized} when $\gamma(e)$ is
the identity for all edges $e$ in $F$.

\begin{prop} \label{P:NormalizationUnique}
Let $\Gamma$ be an abelian group, $G$ be a graph, and $F$ a maximal forest in $G$. If $\vp$ and $\psi$ are two $F$\-/normalized $\Gamma$\-/gain functions on $G$, then $\vp$ and $\psi$ are switching equivalent if and only if $\vp = \psi$.
In the case that $\Gamma$ is the additive group of a field $\FF$, $\vp$ and $\psi$ are switching-and-scaling equivalent if and only if $\vp = a \psi$ for some scalar $a \in \FF^\times$.
\end{prop}

\paragraph{Minors and induced gain functions.}
Minors of gain graphs are defined so that they are consistent with those of their corresponding biased graphs.
Let $\Gamma$ be a group, and let $(G,\gamma)$ be a $\Gamma$\-/gain graph.
Every minor $H$ of $G$ has an \emph{induced} $\Gamma$\-/gain function $\gamma|_H$ inherited from $\gamma$.
Moreover, whenever $(H,\Ss)$ is a biased graph that is a minor of $(G,\Bb_\gamma)$ then $(H,\Ss)$ is realized by the induced gain function $\gamma|_H$ on $E(H)$ inherited from $\gamma$.
We now define these notions and justify this claim.

Let $e$ be an edge of $G$.
We denote by $(G,\gamma)\bs e$ and $(G,\gamma)/e$ the gain graphs obtained by deletion and contraction of $e$ with their induced gain functions defined as follows.
The induced gain function $\gamma|_{G\bs e}$ on $G\bs e$ is the restriction of $\gamma$ to $E(G)-e$.
If $e$ is a loop assigned the identity element of $\Gamma$ by $\gamma$, then $G/e = G \bs e$ so again the induced gain function is just the restriction of $\gamma$ to $E(G)-e$.
If $e$ is a link, then there is switching function $\eta$ such that $\gamma^\eta(e)$ is the identity element of $\Gamma$.
Define the induced gain function $\gamma|_{G/e}$ on $G/e$ to be the restriction of $\gamma^\eta$ to $E(G) - e$.
Finally, suppose $e$ is a loop with $\gamma(e)$ not equal to the identify element of $\Gamma$.
Suppose $e$ is incident to vertex $v \in V(G)$.
For a lift-type contraction of $e$, $(G,\gamma)/e$ is the gain graph $(G \bs e, \iota)$ where $\iota$ is the gain function assigning the identity element of $\Gamma$ to every edge; declare $\iota$ to be the induced gain function on $(G,\gamma)/e$.
For a frame-type contraction of $e$, $(G,\gamma)/e$ is the gain graph
$(G',\gamma|_{G'})$ in which $G'$ is the underlying graph of the biased graph $(G,\Bb_\gamma)/e$ obtained by the frame-type contraction of the joint $e$ defined in Section \ref{sec:graphsandbiaedgraphs} above,
and $\gamma|_{G'}$ is the gain function
whose restriction to $E(G-v)$ is equal to the restriction of $\gamma$ to $E(G-v)$,
that assigns the identity element of $\Gamma$ to each balanced loop of $(G,\Bb_\gamma)/e$,
and assigns $\gamma(e)$ to each new joint of $(G,\Bb_\gamma)/e$.

For link minors, induced gain functions can be defined globally (up to switching) as follows.
Consider a biased graph $\GB$ and two disjoint subsets $K, D \subseteq E(G)$ where $K$ does not contain any loops.
Let $K' \subseteq K$ be a set of edges that induce a maximal forest in $G[K]$, and let $D' = D \cup (K \bs K')$.
Then $(G,\mc B)/K\bs D=(G,\mc B)/K'\bs D'$.
Thus a link minor may always be obtained by contraction of an acyclic set.
Consider a gain graph $(G,\gamma)$ with disjoint subsets $K, D \subseteq E(G)$ where $K$ does not contain a loop.
We obtain an induced gain function $\gamma|_{G/K\bs D}$ for $G/K\bs D$ as follows.
Choose a subset $K' \subseteq K$ such that $G[K']$ is a maximal forest contained in $G[K]$, and choose a maximal forest $F$ of $G$ containing $K'$.
Let $\gamma^\eta$ be the $F$\-/normalization of $\gamma$.
Define $\gamma|_{G/K\bs D}$ to be the restriction of $\gamma^\eta$ to $E(G) - (K \cup D)$.

\begin{prop} \label{P:InequivalencePreservedInContraction}
Let $G$ be a graph and let $F$ be the edge set of a forest in $G$.
Let $\Gamma$ be an abelian group (resp., the additive group of a field).
Then $\vp$ and $\psi$ are switching equivalent (resp., switching-and-scaling equivalent) $\Gamma$\-/gain functions on $G$
if and only if
$\vp|_{G/F}$ and $\psi|_{G/F}$ are switching equivalent (resp., switching-and-scaling equivalent).
\end{prop}

\begin{proof}
Extend $F$ to a maximal forest $F_m$ in $G$ and assume that $\vp$ and $\psi$ are normalized on $F_m$.
If $\vp$ and $\psi$ are switching inequivalent (resp., switching-and-scaling inequivalent), then certainly $\vp\neq\psi$ (resp., $\vp \neq a\psi$ for any scalar $a$), and since both are normalized on $F_m$, their restrictions
to $E(G) \bs F_m$ are not equal (resp., neither is obtained by scaling the other).
Now in $G/F$, the induced gain functions $\vp|_{G/F}$ and $\psi|_{G/F}$ are normalized
on the maximal forest $F_m\bs F$ of $G/F$ and $\vp|_{G/F}\neq\psi|_{G/F}$.
Thus by Proposition \ref{P:NormalizationUnique} $\vp|_{G/F}$ and $\psi|_{G/F}$ are switching inequivalent (resp., switching-and-scaling inequivalent) on $G/F$.

Conversely, if $\vp$ and $\psi$ are switching equivalent (resp., switching-and-scaling equivalent), then by Proposition \ref{P:NormalizationUnique}, $\vp=\psi$ (resp., $\vp = a\psi$ for some scalar $a$), so $\vp|_{G/F} = \psi|_{G/F}$ (resp., $\vp|_{G/F} = a\psi|_{G/F}$).
\end{proof}

We say that biased graph $(G,\mc B)$ has an $(H,\mc S)$\-/minor (respectively, $(H,\mc S)$\-/link minor) when there is a minor (resp., link minor) $(G',\mc B')$ of $(G,\mc B)$ that is isomorphic to $(H,\mc S)$.
The following proposition now follows immediately from the definitions.

\begin{prop} \label{P:InducedRealization}
If $\gamma$ is a $\Gamma$\-/realization of $(G,\mc B)$ and $(G',\mc B')$ is a minor of $(G,\mc B)$, then the induced
gain function $\gamma|_{G'}$ is a $\Gamma$\-/realization of $(G',\mc B')$.
\end{prop}

When $\gamma$ is a $\Gamma$\-/realization of $\GB$ and $(G',\Bb')$ is a minor of $\GB$, we say the $\Gamma$\-/realization $\gamma|_{G'}$ of $(G',\mc B|_{G'})$ is the
\emph{induced $\Gamma$\-/realization} of $(G',\mc B|_{G'})$.

\paragraph{Minors and canonical representations.}
Given a matroid $M$ represented over a field by a matrix $A$, the operation of removing column $e$ from $A$ yields a matrix representation of $M \bs e$.
The operation of applying row operations so that column $e$ contains a unique nonzero element equal to 1 and then removing column $e$ along with the row in which column $e$ is nonzero yields a matrix representation of $M/e$.
With just a little more care, we may apply these usual operations to a canonical matrix representation to obtain a canonical matrix representation given by the corresponding induced gain function on the gain graph minor.
For a canonical matrix $A$, and column $e$ of $A$, denote by $A \bs e$ the matrix obtained by removing column $e$ from $A$.
Denote by $A/e$ a matrix obtained by applying row operations so that column $e$ contains a unique nonzero element equal to 1, say in row $i$, and then removing column $e$ and row $i$ from $A$, subject to the following.
If $A$ is a canonical lift matrix and $e$ is a link, then row $i$ is not the ``gains row" $v_0$.
If $A$ is a canonical frame matrix and if $e$ is a joint, then also scale columns so that for each column $e' \neq e$ with a nonzero entry in row $i$ and with a second nonzero entry in a row $j \neq i$, entry $A_{j e'}$ is equal to $1-\gamma(e)$.

Given an $\FF^\times$-gain function $\vp$ on a graph $G$, denote by $A_F(G,\vp)$ the canonical frame matrix defined by $(G,\vp)$ as described in Section \ref{subsec:gaingraphscanonicalreps}.
Similarly, for an $\FF^+$\-/gain function $\psi$ on $G$, denote by $A_L(G,\psi)$ the canonical lift matrix defined by $(G,\psi)$, as described in Section \ref{subsec:gaingraphscanonicalreps}.
The following lemma is a straightforward consequence of the definitions.

\begin{lem} \label{lem:gain_graph_and_canonical_matrix_minors}
Let $G$ be a graph, let $\FF$ be a field, and let $\vp \: \vec E(G) \to \FF^\times$ and $\psi \: \vec E(G) \to \FF^+$ be gain functions.
Let $e \in E(G)$.

\textup{(i)} $A_F((G,\vp)\bs e) = A_F(G,\vp) \bs e$ and $A_L((G,\vp) \bs e) = A_L(G,\vp) \bs e$.

\textup{(ii)} $A_F((G,\vp)/e) = A_F(G,\vp)/e$, where if $e$ is a joint the contraction operation is of frame type.

\textup{(iii)} $A_L((G,\psi)/e) = A_L(G,\psi)/e$, where if $e$ is a joint the contraction operation is of lift type.
\end{lem}

Evidently, if $A$ and $B$ are projectively equivalent matrices over $\FF$, then so too are $A \bs e$ and $B \bs e$ projectively equivalent, as are $A/e$ and $B/e$.

\subsection{$\Delta$-$Y$ and $Y$-$\Delta$ exchanges}
The operations of $\Delta$-$Y$ and $Y$-$\Delta$ exchange in graphs and in matroids are well understood.
Here we generalize these operations from graphs to biased graphs, and show that an exchange in a biased graph representation agrees with that in the matroid.
We apply $\Delta$-$Y$ and $Y$-$\Delta$ exchanges in a biased graph $\GB$ by performing the usual operation in $G$, then appropriately defining a collection of balanced cycles to define the resulting biased graph.
We use these tools in Sections \ref{sec:unavoidableminors} and \ref{sec:MainThmProofs}.

Let $X=\{a,b,c\}$ be the edge set of a triangle $T$ in a graph $G$, with $V(T) = \{x,y,z\}$, where $a=yz$, $b=xz$, and $c=xy$.
Delete edges $a,b,c$ from $G$, add a new vertex $v$ to $G$, along with three new edges $vx$, $vy$, and $vz$, and label the new edges $a,b,c$ so that $a = vx$, $b=vy$, and $c=vz$.
The resulting graph, in which the triangle $X$ is replaced with a $K_{1,3}$-subgraph with the same set of edges, is said to have been obtained from $G$ by a \emph{$\Delta$-$Y$ exchange}, and is denoted by $\Delta_X G$.

The reverse operation is a \emph{$Y$-$\Delta$ exchange}:
Let $v$ be vertex of degree three in a graph $G$, whose incident edges $Y = \{a,b,c\}$ induce a $K_{1,3}$-subgraph of $G$.
Let $x,y,z$ be the neighbours of $v$ in $G$, where $a=vx$, $b=vy$, and $c=vz$.
Delete $v$ along with its incident edges $a,b,c$, and add three new edges $xy$, $yz$, and $xz$, and label the new edges $a,b,c$ so that $a=yz$, $b=xz$, and $c=xy$.
The resulting graph, in which the $K_{1,3}$-subgraph $Y$ is replaced with a triangle with the same set of edges, is said to have been obtained from $G$ by a \emph{$Y$-$\Delta$ exchange}, and is denoted by $\nabla_Y G$.

We extend these operations to biased graphs as follows.
Let $\GB$ be a biased graph and let $X$ be a balanced triangle of $G$.
Define $\Delta_X \Bb$ to be the collection of cycles of $\Delta_X G$ given by
\begin{linenomath}
\[
\{C \in \Bb:|C\cap X|=0\mbox{ or }2\}\cup\{C \triangle X:C \in \Bb\mbox{ and }|C\cap X|=1\}
\]
\end{linenomath}
where $\triangle$ denotes symmetric difference.
Let $Y$ be a $K_{1,3}$\-/subgraph of $\GB$.
Define $\nabla_Y \Bb$ to be the collection of cycles of $\nabla_Y G$ consisting of $Y$ together with the minimal nonempty members of the set
\begin{linenomath}\[ \{ C : C \in \Bb \text{ and } |C \cap Y| = 0 \text{ or } 2\} \cup \{ C \triangle Y : C \in \Bb \text{ and } |C \cap Y| = 2\}.\]\end{linenomath}

The proofs of the following two propositions are straightforward checks.

\begin{prop}
Let $X$ be a balanced 3-cycle in a biased graph $\GB$.
Then $(\Delta_X G, \Delta_X \Bb)$ is a biased graph.
\end{prop}


\begin{prop}
Let $Y$ be a $K_{1,3}$\-/subgraph of a biased graph $\GB$.
Then $(\nabla_Y G, \nabla_Y \Bb)$ is a biased graph. \end{prop}


We denote the biased graph $(\Delta_X G, \Delta_X \Bb)$ by $\Delta_X\GB$ and the biased graph $(\nabla_Y G, \nabla_Y \Bb)$ by $\nabla_Y \GB$.

Observe that a $\Delta$-$Y$ exchange in a graph $G$ is given by a 3-sum of a labelled $K_4$ and $G$.
To perform a $\Delta$-$Y$ exchange in a matroid $M$, take a copy of $M(K_4)$ on ground set $\{a,b,c,a',b',c'\}$ labelled so that $\{a,b,c\}$ is a triangle, $\{a',b',c'\}$ is a triad, and
each of $\{a,b',c'\}$, $\{a',b,c'\}$, and $\{a',b',c\}$ are triangles.
Let $X=\{a,b,c\}$ be a triangle of $M$, and take the generalized parallel connection of $M$ and $M(K_4)$ across $X$ (see \cite{oxley:mt2011}, Sections 11.4 \& 11.5).
Finally, delete $X$ and relabel each of $a',b',c'$ by $a,b,c$, respectively.
The resulting matroid is said to have been obtained via a \emph{$\Delta$-$Y$ exchange on $X$}, and is denoted $\Delta_X M$.

A $Y$-$\Delta$ exchange in a matroid $M$ is defined via duality.
Let $Y$ be a triad of $M$.
Then $Y$ is a triangle of the dual $M^*$.
Define $\nabla_Y M = (\Delta_Y M^*)^*$.
We say $\nabla_Y M$ is obtained from $M$ via a \emph{$Y$-$\Delta$ exchange on $Y$}.
%

The following proposition is a straightforward consequence of the definitions.
It can be proved by comparing flats.


\begin{prop} \label{prop:DeltaYoperation}
Let $\GB$ be a biased graph.

\textup{(i)} If $X$ is a balanced triangle of $\GB$, then
$F\br{\Delta_X \GB} = \Delta_X F \GB$ and $L\br{\Delta_X \GB} = \Delta_X L \GB$.

\textup{(ii)} Let $Y$ be a $K_{1,3}$\-/subgraph of $G$ for which, if $\GB$ is unbalanced, $\GB-E(Y)$ remains unbalanced.
Then
$F\br{\nabla_Y \GB} = \nabla_Y F\GB$ and
$L\br{\nabla_Y \GB} = \nabla_Y L\GB$.
\end{prop}

The projective equivalence classes of matrix representations of a matroid are well-behaved under $\Delta$-$Y$ exchanges:

\begin{prop}[Whittle {\cite[Lemma 5.7]{MR1710531}}] \label{P:Whittle}
Let $M'$ be a matroid obtained from the matroid $M$ by a single $\Delta$-$Y$ exchange, and let $\FF$ be a field.

\textup{(i)} $M$ is $\FF$\-/representable if and only if $M'$ is $\FF$\-/representable.

\textup{(ii)} The projective equivalence classes of $\bb F$\-/representations of $M$ are in one-to-one correspondence with the projective equivalence classes of $\FF$\-/representations of $M'$.
\end{prop}

Gain functions realizing a biased graph are similarly well-behaved under $\Delta$-$Y$ exchanges.
Proposition \ref{prop:gain_functions_DeltaYs} is an analogue of Proposition \ref{P:Whittle}.

\begin{prop}\label{prop:gain_functions_DeltaYs}
Let $X$ be a balanced triangle of a biased graph $\GB$.
Let $\Gamma$ be the multiplicative (resp., additive) group of a field.

\textup{(i)} $\vp$ is a $\Gamma$\-/realization of $\GB$ if and only if $\vp$ is a $\Gamma$\-/realization of $\Delta_X \GB$.

\textup{(ii)} The switching (resp., switching-and-scaling) classes of $\Gamma$\-/realizations of $\GB$ are in one-to-one correspondence with the switching (resp., switching\-/and\-/scaling) classes of $\Gamma$\-/realizations of $\Delta_X \GB$.
\end{prop}

\begin{proof}
(i) Let $F$ be a maximal forest of $\Delta_X G$ containing $X$.
For each edge $e\in X$, $F-e$ is a maximal forest of $G$ that contains two edges of $X$.
By Proposition \ref{P:NormalizationUnique}, every $\Gamma$\-/realization of $(G,\mathcal B)$ is switching equivalent (resp., switching-and-scaling equivalent) to a
unique (resp., unique up to scaling) $(F-e)$\-/normalized $\Gamma$\-/realization and every $\Gamma$\-/realization of $\Delta_X(G,\mathcal B)$ is switching (resp., switching-and-scaling) equivalent to a unique (resp., unique up to scaling) $F$\-/normalized $\Gamma$\-/realization.
Since $X$ is a balanced triangle of $\GB$, every $(F-e)$\-/normalized $\Gamma$\-/realization of $(G,\mathcal B)$ also has identity gain value on $e$.
Thus a $\Gamma$\-/realization assigning identity gains to each edge in $X$ is a $\Gamma$\-/realization of $(G,\mathcal B)$ if and only if it is a $\Gamma$\-/realization of $\Delta_X(G,\mathcal B)$.
Thus we obtain a canonical bijection between between $\Gamma$\-/gain realizations of $\GB$ and $\Delta_X \GB$, so (ii) holds.
\end{proof}

We now show that a $\Delta$-$Y$ exchange applied to a canonical representation agrees with the exchange applied to its gain graph.
Over any field, a matrix representing the matroid $M(K_4)$ is projectively equivalent to the matrix $I(K_4)$ shown below.
\begin{linenomath}\[I(K_4)= \left(\begin{array}{cccccc} 1 & 0 & 0 & 1 & 1 & 0
\\ 0 & 1 & 0 & -1 & 0 & 1 \\ 0 & 0 & 1 & 0 & -1 & -1 \\ -1 &-1 &-1 & 0 & 0 & 0\\ \end{array} \right)\]\end{linenomath} The first three
columns of this matrix represent a triad and the last three columns a triangle of $K_4$.
Given a matrix $A'$ such that $M(A')$ contains a triangle $X$, $A'$ is projectively equivalent to a matrix $A$ in which the columns corresponding to $X$ are as the last three columns of $I(K_4)$, perhaps with one row omitted or with zero rows added.
Let $\Delta_XA$ denote the matrix obtained from $A$ by replacing the columns for $X$ with first three columns of $I(K_4)$, where possibly the fourth row is omitted or additional zero rows are added.
Then $M(\Delta_X A)=\Delta_X M(A)$ \cite{Brylawski:Sums}.
Similarly if $M(A')$ contains a triad $Y$, then $A'$ is projectively equivalent to a matrix $A$ in which the columns corresponding to $Y$ are as the first three columns of $I(K_4)$, possibly with a row omitted or additional zero rows.
Let $\nabla_Y A$ denote the matrix $A$ obtained by replacing the columns of $Y$ with those of the last three columns of $I(K_4)$.
Then $M(\nabla_Y A)=\nabla_Y M(A)$ \cite{Brylawski:Sums}.
Thus the next fact follows immediately
from Proposition \ref{prop:gain_functions_DeltaYs}.

\begin{prop} \label{P:DeltasAndCanonical}
Let $\FF$ be a field and let $(G,\vp)$ be a $\Gamma$\-/gain graph, where $\Gamma \in \{\FF^\times, \FF^+\}$.
Let $X$ be a balanced triangle of $(G,\Bb_\vp)$ and let $Y$ be a $K_{1,3}$\-/subgraph of $G$.
\begin{enumerate}[label=\textup{(\roman*)}]
\item $\Delta_X A_F(G,\vp)$ is a canonical frame matrix particular to $(\Delta_X G, \vp)$,
\item $\nabla_Y A_F(G,\vp)$ is a canonical frame matrix particular to $(\nabla_Y G, \vp)$,
\item $\Delta_X A_L(G,\vp)$ is a canonical lift matrix particular to $(\Delta_X G, \vp)$,
\item $\nabla_Y A_L(G,\vp)$ is a canonical lift matrix particular to $(\nabla_Y G, \vp)$.
\end{enumerate}
\end{prop}

\subsection{Almost-balanced biased graphs: roll-ups}
\label{sec:rollups}

Let $\GB$ be an almost-balanced biased graph with the property that after deleting its joints it has a unique balancing vertex $u$.
Let $J$ be the set of joints of $\GB$ that are not incident to $u$, and denote by $\delta(u)$ the set of links incident to $u$.
It is not difficult to check that the theta property implies that for each pair $e, e' \in \delta(u)$, either all cycles containing both $e$ and $e'$ are balanced or all cycles containing both $e$ and $e'$ are unbalanced (for details, see \cite[Section 1]{AlmostBalancedReps}).
Observe that, just as for a pair $e, e' \in \delta(u)$ for which every cycle containing both $e$ and $e'$ is balanced, for every pair $f,f' \in J$, every path linking the endpoint of $f$ with the endpoint of $f'$ together with $f$ and $f'$ is a circuit of $F\GB$.
Define $\Sigma_G(u) = \Sigma(u) = \delta(u) \cup J$.
Then for each pair of edges $e_1, e_2 \in \Sigma_G(u)$, either
every minimal path linking the endpoints of $e_1$ and $e_2$ in $G-u$ together with $\{e_1, e_2\}$ forms a circuit in $F\GB$, or each such path together with $\{e_1, e_2\}$ is independent in $F\GB$.
Define a relation $\sim$ on $\Sigma_G(u)$
in which $e_i \sim e_j$ if and only if $i = j$, or there is a balanced cycle containing $e_i$ and $e_j$, or $e_i$ and $e_j$ are both in $J$.
It is straightforward to check that $\sim$ is an equivalence relation \cite[Lemma 1.4]{AlmostBalancedReps}.
We call the $\sim$-equivalence classes partitioning $\Sigma(u)$ the \emph{unbalancing classes} of $\Sigma(u)$, and denote the set of unbalancing classes of $\Sigma(u)$ by $\Sigma(u)/\!\!\sim$.

Now consider the biased graph $\GBu$ obtained from $\GB$ by replacing each joint $e\in J$ incident to a vertex $v \neq u$ with a $uv$-link.
Define $\widehat{\Bb}$ to be the set of those cycles having intersection of size 0 or 2 with each unbalancing class of $\Sigma_G(u)$.
It is straightforward to check by comparing circuits that $F\GBu = F\GB$.
Call $\GBu$ the \emph{unrolling of $(G,\mc B)$ to} $u$.
If $J$ is empty, then set $\GBu = \GB$.
Observe that $u$ is a balancing vertex of $\GBu$, that $\Sigma_{\widehat{G}}(u) = \Sigma_G(u)$, and that
$\Sigma_{\widehat{G}}(u)/\!\!\sim \ = \Sigma_G(u)/\!\!\sim$;
that is, the unbalancing classes of $\Sigma_G(u)$ and of $\Sigma_{\widehat{G}}(u)$ are the same.

For each unbalancing class $U$ of $\Sigma_{\widehat{G}}(u)$ there is a biased graph $(G_U,\Bb_U)$ for which $F(G_U,\Bb_U) = F\GB = F\GBu$, obtained from $\GBu$ by replacing each link $e=uv \in U$ with a joint incident to $v$ (this fact is straightforward to check by comparing circuits; it appears in \cite[Proposition 2.2]{AlmostBalancedReps}).
Call each such biased graph a \emph{roll-up of $\GBu$ from $u$}.
It is a straightforward check that for each unbalancing class $U \in \Sigma_{\widehat{G}}(u)/\!\!\sim$, $\Sigma_{G_U}(u)/\!\!\sim \ = \Sigma_{\widehat{G}}(u)/\!\!\sim$.
Note that $J$ is an unbalancing class of $\Sigma_{\widehat{G}}(u)$, and that $(G_J, \Bb_J) = \GB$.
Define $\Rr_{\GB}$ to be the set of biased graphs consisting of $\GBu$ together with all of its roll-ups.
Since each of these biased graphs shares precisely the same set of unbalancing classes, we may write simply $\Sigma(u)/\!\!\sim$ for this set, when it is clear that we are considering a biased graph in the collection $\Rr_{\GB}$.
Thus $|\Rr_{\GB}| = |\Sigma(u)/\!\!\sim\!\!|+1$.
Finally, observe that since $\GBu$ has no pair of vertex-disjoint unbalanced cycles, $L\GBu = F\GBu$.

There is a special case to consider if $\GB$ is balanced after removing its set of joints $J$.
Let $x$ be a new isolated vertex added to $V(G)$.
Then $J$ may be unrolled to any vertex of $G$, including $x$.
In the case that $J$ is unrolled to $x$, we obtain a balanced biased graph $(H,\Cc(H))$, so $F\GB$ is equal to the cycle matroid $M(H)$ of $H$.
The reverse operation may be applied to any graph.
Given a graph $H$ and a vertex $x \in V(H)$, let $(G,\Bb)$ be the rollup of the set of edges incident to $x$; that is, since the set of edges incident to $x$ is a single unbalancing class, replace each edge $xv$ with a joint incident to its endpoint $v$.
Then $M(H) = F\GB$.

If a biased graph has two distinct balancing vertices, then it has the very restricted form described in Proposition
\ref{prop:Zas_fattheta}.

\begin{prop}[Zaslavsky \cite{MR897095}] \label{prop:Zas_fattheta}
Let $\GB$ be a 2-connected unbalanced biased graph with two distinct balancing vertices $x$ and $y$.
Then $G$ is a union of subgraphs $G_1\cup \dots \cup G_m$ where for each pair $i \not= j$, $G_i \cap G_j = \{x,y\}$, and a cycle is in $\Bb$ if and only if it is contained in a single subgraph $G_i$.
If $m\geq3$ then $x$ and $y$ are the only balancing vertices of $\GB$.
\end{prop}

Observe that if a biased graph $\GB$ of the form described in Proposition \ref{prop:Zas_fattheta} has a subgraph $G_i$ with at least two edges, then $(E(G_i), E(G)-E(G_i))$ is a 2-separation of both $L\GB$ and $F\GB$.

\subsection{Full-rank canonical lift-matrix representations}
\label{sec:on_canonical_lift_representations}

Let $G$ be a graph and let $\gamma$ be an $\FF^+$-gain function on $G$, for some field $\FF$.
The canonical lift matrix $A_L(G,\gamma)$ consists of the oriented incidence matrix of the subgraph of $G$ induced by its links together with a row $v_0$ of gains.
It is sometimes inconvenient that this matrix is not of full rank.
Choosing one vertex in each component of $G$, and deleting the rows of $A_L(G,\gamma)$ indexed by these vertices
yields a matrix representation of $L(G,\Bb_\gamma)$ that is of full rank.
In the case $G$ is connected, just one row is removed, and the oriented incidence matrix of $G$ is recovered from the resulting matrix by appending a row
equal to the negation of the sum of all rows but $v_0$.
When $G$ is connected, let us denote by $A_L^{-v}(G,\gamma)$ the matrix obtained from $A_L(G,\gamma)$ by deleting the row indexed by the vertex $v \in V(G)$.
Then $A_L^{-v}(G,\Bb_\gamma)$ is a full-rank matrix representing $L(G,\Bb_\gamma)$.
Clearly, for any vertex $v \in V(G)$ the matrices $A_L^{-v}(G,\gamma)$ and $A_L(G,\gamma)$ are projectively equivalent, and for any pair of vertices $u,v \in V(G)$, the matrices $A_L^{-u}(G,\gamma)$ and $A_L^{-v}(G,\gamma)$ are projectively equivalent.
We say $A_L^{-v}(G,\gamma)$ is a \emph{full-rank canonical lift matrix representation} of $L(G,\Bb_\gamma)$.
In the case that $(G,\Bb_\gamma)$ has a balancing vertex $u$ after deleting its joints, it will be convenient to use $A_L^{-u}(G,\gamma)$ as the full-rank canonical lift matrix representation of $L(G,\Bb_\gamma)$.

\section{Unavoidable minors}
\label{sec:unavoidableminorsandsubdivisions}

In this section we prove Theorems \ref{MT:UnavoidableMinors},  \ref{MT:Unavoidsubdivisions}, and \ref{MT:InequivalenceLocalized}.
We show that there is a small collection of biased graphs, at least one of which must appear as a minor in every 2-connected biased graph.
From this collection we obtain a slightly larger collection of biased graphs, and show that every 2-connected properly unbalanced biased graph contains a subdivision of at least one of these biased graphs.
We prove an analogous result for 2-connected almost-balanced biased graphs, which we require for the proof of Theorem \ref{thm:all_reps_eqiv_to_cannonical}.
Finally, we show that inequivalence of gain functions may always be found on a one of small number of unavoidable minors.

\subsection{The minor-minimal, 2-connected, properly unbalanced biased graphs} \label{sec:base_biased_graphs}

Let $\Gg_0$ denote the set of minor-minimal biased graphs that are 2-connected and properly unbalanced.
We first describe 13 biased graphs in $\Gg_0$, then show that these 13 biased graphs form the complete set.
Recall that we denote the graph obtained from a 3-cycle by replacing each edge with a pair of parallel edges by $2C_3$, and that we call
the graph obtained from a 4-cycle by
replacing each edge in a pair of non-adjacent edges with a pair of parallel edges the tube, and denote it by $2C_4''$.
Six of the biased graphs in $\mathcal{G}_0$ have underlying graph $2C_3$, three have underlying graph $2C_4''$, and four have underlying graph $K_4$.

The set of cycles of the graph $K_4$ consists of four triangles and three quadrilaterals.
We denote by $\mathsf D_{t,q}=(K_4,\mc B_{t,q})$ the biased $K_4$ with exactly $t$ balanced triangles and $q$ balanced 4-cycles.
There are seven biased $K_4$'s: $\mathsf D_{0,0}$, $\mathsf
D_{0,1}$, $\mathsf D_{0,2}$, $\mathsf D_{0,3}$, $\mathsf D_{1,0}$, $\mathsf D_{2,1}$, and $\mathsf D_{4,2}$  \cite{Zaslavsky:BG1}.
A biased $K_4$ is properly unbalanced if and only if it does not contain a balanced triangle.
Thus the properly unbalanced $K_4$'s are $\mathsf D_{0,0}$,
$\mathsf D_{0,1}$, $\mathsf D_{0,2}$, and $\mathsf D_{0,3}$.

\begin{figure}[tbp] \begin{center} \includegraphics[height=45pt,trim= 0pt 20pt 0pt -10pt,page=1]{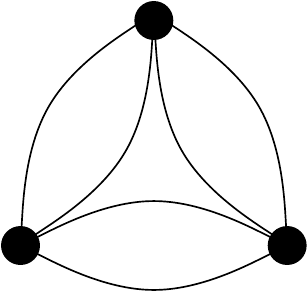}\hspace{2cm}
\begin{tabular}{cc} {Name} & {Balanced cycles} \\
\hline $\mathsf T_0$ & none\\
\hline $\mathsf T_1$ & \includegraphics[height=15pt,trim= 0pt 20pt 0pt -10pt,page=2]{DoubleTriangle}\\
\hline $\mathsf T_2$ & \includegraphics[height=15pt,trim= 0pt 20pt 0pt -10pt,page=4]{DoubleTriangle}\\
\hline $\mathsf T_{2}'$ & \includegraphics[height=15pt,trim= 0pt 20pt 0pt -10pt,page=3]{DoubleTriangle}\\
\hline $\mathsf T_3$ & \includegraphics[height=15pt,trim= 0pt 20pt 0pt -10pt,page=5]{DoubleTriangle}\\
\hline $\mathsf T_4$ & \includegraphics[height=15pt,trim= 0pt 20pt 0pt -10pt,page=6]{DoubleTriangle}\\
\hline
\end{tabular}
\end{center}
\caption{The graph $2C_3$ and its six possible classes of balanced cycles not containing a cycle of length two.} \label{F:Biased2C3}
\end{figure}

\begin{prop} \label{P:Biased2C3s} There are six unlabelled properly unbalanced biased $2C_3$'s.
\end{prop}

\begin{proof}
A biased graph $(2C_3,\Bb)$ is properly unbalanced if and only if $\Bb$ does not contain a 2-cycle.
Thus by theta property, $\Bb$ is a collection of triangles pairwise intersecting in at most one edge.
There are eight triangles in $2C_3$; any set of five contains a pair that intersect in more than one edge.
Hence $\Bb$ contains at most 4 triangles.
The possibilities are shown in Figure \ref{F:Biased2C3}.
\end{proof}

A biased tube is properly unbalanced if and only if it has no balanced 2-cycle. There are
three such tubes, described in Figure \ref{F:BiasedTube}.

\begin{figure}[tbp]
\begin{center}
\includegraphics[height=35pt,trim= 0pt 20pt 0pt -10pt,page=3]{BiasedTubes}
\hspace{2cm}
\begin{tabular}{cc} {{Name}} & {{Balanced cycles}} \\\hline $\mathsf
B_0$ & none\\\hline $\mathsf B_1$ & \includegraphics[height=15pt,trim= 0pt 20pt 0pt -10pt,page=1]{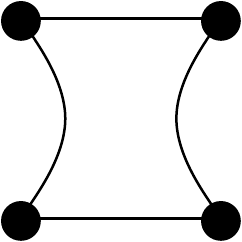}\\\hline
$\mathsf B_2$ & \includegraphics[height=15pt,trim= 0pt 20pt 0pt -10pt,page=2]{BiasedTubes}\\\hline \end{tabular}
\caption{The graph $2C_4''$ and its three possible classes of balanced cycles not containing a cycle of length two.}
\label{F:BiasedTube}
\end{center}
\end{figure}

We now show that the set
of biased graphs consisting of the four biased $K_4$'s with no balanced triangle, the six biased $2C_3$'s with no balanced 2-cycle, and the three biased tubes with no balanced 2-cycle,
\begin{linenomath}
\[
\{\mathsf{D_{0,0}, D_{0,1}, D_{0,2}, D_{0,3}, T_0, T_1, T_2, T_2', T_3, T_4, B_0, B_1, B_2}\} ,
\]
\end{linenomath}
forms the complete collection $\mathcal{G}_0$ of minor-minimal, 2-connected, properly unbalanced biased graphs.

A \emph{subdivision} of a biased graph $(G,\mc B)$ is a biased graph $(H,\mc S)$ in which $H$ is a subdivision of $G$ and a cycle $C$ of $H$ is in $\mc S$ if and only if its corresponding cycle $C'$ of $G$ is in $\mc B$.
Proposition \ref{P:TubeMinor} follows immediately from Menger's Theorem.

\begin{prop} \label{P:TubeMinor}
Let $\GB$ be a 2-connected biased graph.
If $\GB$ contains a vertex-disjoint pair of unbalanced cycles, neither of which is a loop, then $\GB$ contains a subdivision of $\mathsf B_0$, $\mathsf B_1$, or $\mathsf B_2$.
\end{prop}

Thus to prove Theorem \ref{MT:UnavoidableMinors} it remains just to show that a properly unbalanced biased graph without two vertex-disjoint unbalanced cycles has a
link minor from $\mc G_0$.
A properly unbalanced biased graph with no two vertex-disjoint unbalanced cycles is \emph{tangled}.
The structure of tangled signed graphs was characterized by Slilaty \cite{MR2344133} and the structure of tangled biased graphs in general was characterized by Chen and Pivotto \cite{2014arXiv1403.1919C}.
The following theorem could be proven
as a consequence of Chen and Pivotto's work in \cite{2014arXiv1403.1919C}, but the direct proof we present here seems no more difficult.

\begin{thm} \label{T:TangledMinor}
Every tangled biased graph contains  as a link minor either a biased $2C_3$ with no balanced 2-cycle or a biased $K_4$ with no balanced triangle.
\end{thm}

\begin{lem} \label{L:UntanglingThree}
Let $\Omega$ be a tangled biased graph.
Assume $\Om$ contains an unbalanced cycle $C$ and a pair of unbalanced cycles $C_x$ and $C_y$ such that $V(C) \cap V(C_x) = \{x\}$, $V(C) \cap V(C_y) = \{y\}$, and $x \neq y$.
Then $C \cup C_x \cup C_y$ contains a biased $2C_3$ with no balanced 2-cycle as a link minor.
\end{lem}

\begin{proof}
Since $\Omega$ is tangled, $(C_x\cup C_y)-\{x,y\}$ is connected; furthermore, it is vertex-disjoint from $C$ and
so balanced.
Let $K$ be the edge set of $(C_x\cup C_y)-\{x,y\}$.
Then $(C\cup C_x\cup C_y)/K$ is a link minor of $C\cup
C_x\cup C_y$ and is a subdivision of a biased $2C_3$ with no balanced 2-cycle. The result follows.
\end{proof}

\begin{proof}[Proof of Theorem \ref{T:TangledMinor}]
Let $\Omega$ be a link-minor-minimal counterexample.
Then $|V(\Om)| > 2$ and $\Om$ has no joint, else $\Om$ would not be tangled.
If $\Om$ has more than one unbalanced block but no two disjoint unbalanced cycles, then $\Om$ must have a balancing vertex, a contradiction. Hence $\Omega$ has only one unbalanced block. Evidently our desired minor exists in $\Omega$ if and only if it exists in the unbalanced block of $\Omega$.
Hence by minimality $\Omega$
is 2-connected.
By minimality we may also assume that $\Omega$ has no balanced 2-cycles.

\begin{claim} \label{claim:NoUnbalanced2Cycle} The underlying graph of $\Omega$ is simple. \end{claim} \begin{cproof} By
way of contradiction assume that $C$ is an unbalanced 2-cycle in $\Omega$ with vertices $x$ and $y$. Thus $\Omega-x$
contains an unbalanced cycle $C_y$ passing through $y$ and $\Omega-y$ contains an unbalanced cycle $C_x$ passing
through $x$. By Lemma \ref{L:UntanglingThree}, $C_x\cup C_y\cup C$ contains a biased $2C_3$ without a balanced 2-cycle,
a contradiction.
\end{cproof}

If $\Omega$ has just three vertices, then the underlying graph of $\Omega$ is a triangle.
This is not the case, as a biased triangle has a balancing vertex.
If $\Om$ has exactly four vertices, then $\Omega-v$ is an unbalanced triangle for each vertex $v$, else $\Om$ has a balancing vertex.
But then $\Omega$ is a biased $K_4$ without a balanced triangle, a contradiction.
Thus $|V(\Omega)| \geq 5$.

\begin{claim} \label{clm:Omega-v} For each vertex $v$, $\Omega - v$ is unbalanced and has a balancing vertex.
\end{claim} \begin{cproof} Minimality implies that for any vertex $v$ in $\Omega$, $\Omega - v$ is not tangled. Since
$\Om-v$ is unbalanced and has no two disjoint unbalanced cycles, it must have a balancing vertex. \end{cproof}

Given an edge $e$ with endpoints $x$ and $y$, we denote the vertex in $\Om/e$ resulting from the identification of $x$
and $y$ by $v_e$ or $v_{xy}$.

\begin{claim} \label{claim:v_euniquebalvertex} For each edge $e$, $\Om/e$ has $v_e$ as its unique balancing vertex.
\end{claim}
\begin{cproof} By minimality, $\Omega/e$ is not tangled and has no two vertex-disjoint unbalanced cycles and
so must therefore have a balancing vertex.
If there is a balancing vertex $u \neq v_e$, then every unbalanced cycle of $\Om/e$ passes through $u$, so every unbalanced cycle of $\Om$ passes through $u$.
But this implies that $u$ is a balancing vertex of $\Om$, a contradiction.
\end{cproof}

\begin{claim} \label{claim:NoBalanced2SepTangled} $\Omega$ does not have a vertical 2-separation $(A,B)$ in which $B$ is
balanced. \end{claim}
\begin{cproof} Suppose, for a contradiction, that $(A,B)$ is a vertical 2-separation in which $B$
is balanced. Let $\{x,y\} = V(A) \cap V(B)$, and let $e$ be an edge in $B$ not incident to at least one of $x$ and $y$.
By Claim \ref{claim:v_euniquebalvertex}, $\Omega/e$ has balancing vertex $v_e$. By our choice of $e$, $(A,B \setminus
e)$ is a 2-separation of $\Om/e$, and $V(A) \cap V(B \setminus e)$ is either $\{x,y\}$, $\{x,v_e\}$ or $\{v_e,y\}$. In
any case, since $B \setminus e$ is balanced, every unbalanced cycle of $\Omega/e$ either does not intersect $B
\setminus e$ or intersects $B \setminus e$ in the edges of a path linking the two vertices of $V(A) \cap V(B \setminus e)$.
This implies that there is a vertex $v\in\{x,y\}$ such that every unbalanced cycle in $\Omega$ contains $v$.
This means $v$ a balancing vertex of $\Omega$, a contradiction.
\end{cproof}

\begin{claim} \label{claim:No2SepTangled} $\Omega$ is 3-connected. \end{claim} \begin{cproof} Suppose that
$\Omega$ has a vertical 2-separation $(A,B)$ and let $\{x,y\}=V(A)\cap V(B)$. By Claim
\ref{claim:NoBalanced2SepTangled}, neither $A$ nor $B$ is balanced. Since $\Omega$ does not have a balancing vertex
both $\Omega-x$ and $\Omega-y$ are unbalanced. Let $C_x$ be an unbalanced cycle in $\Omega-x$ and $C_y$ be an unbalanced cycle in $\Omega-y$.
Without loss of generality, assume $E(C_x) \subseteq A$.
Since $\Om$ has no pair of vertex-disjoint unbalanced cycles and $C_y$ does not contain $y$, this implies that also $E(C_y) \subseteq A$.
Hence $E(C_x)\cup E(C_y)\subseteq A$.
Since $B$ is unbalanced, there is an unbalanced cycle $C'$ in $\Omega[B]$; since $C'$ is vertex-disjoint from neither $C_x$ nor $C_y$, $C'$ meets both vertices $x$ and $y$.
Thus by Lemma \ref{L:UntanglingThree}, $C_x\cup C_y\cup C'$ contains a biased $2C_3$ having no balanced 2-cycle, a contradiction. \end{cproof}

Now let $e=xy$ be an edge of $\Om$ and let $E_1,\ldots,E_m$ be the unbalancing classes of edges incident to balancing vertex $v_{xy}$ in $\Om/e$.
Since $\Om/e$ is unbalanced, $m \geq 2$.
Let $E_{x,i}$ be the set of edges of $E_i$ that are incident to $x$
in $\Omega$ and let $E_{y,i}$ be the set of edges of $E_i$ that are incident to $y$ in $\Omega$. Since $\Omega-y$ is unbalanced, at
least two of $E_{x,1},\ldots,E_{x,m}$ are nonempty; similarly, at least two of $E_{y,1},\ldots, E_{y,m}$ are
nonempty. Let $X$ be the set of vertices in $\Omega-y$ adjacent to $x$, and let $Y$ be the set of vertices in $\Omega-x$
that are adjacent to $y$. Since the underlying graph of $\Omega$ is simple, $|X|\geq2$ and $|Y|\geq2$. Now take
$x_1,x_2\in X$ so that edges $xx_1$ and $xx_2$ are in different sets $E_{x,1},\ldots,E_{x,m}$;
similarly, take $y_1, y_2 \in Y$ so that edges $yy_1$ and $yy_2$ are in different sets $E_{y,1},\ldots,E_{y,m}$.
Since the underlying graph of $\Omega$ is simple, $x_1\neq x_2$ and $y_1\neq y_2$.

\begin{claim} \label{claim:NotFour} Vertices $x_1, x_2, y_1, y_2$ cannot be chosen so that
$\{x_1,x_2\}\cap\{y_1,y_2\}=\emptyset$. \end{claim} \begin{cproof}
Suppose to the contrary that $\{x_1,x_2\} \cap \{y_1,y_2\}$ is empty.
Since $\Omega$ is 3-connected, there is an $x_1$-$x_2$ path $P$ in $\Omega-\{x,y\}$.
Because edges $xx_1$ and $xx_2$ are in different sets $E_{x,1}, \ldots, E_{x,m}$, the cycle $xx_1Px_2x$ is unbalanced.
For $i\in\{1,2\}$ let $e_i$ denote the $yy_i$-edge in $\Omega$. Since $v_{yy_i}$
is a balancing vertex in $\Omega/e_i$ (by Claim \ref{claim:v_euniquebalvertex}), the path $P$ must contain $y_1$ and $y_2$.
Hence there is a $y_1$-$y_2$ path $P'$ properly contained in $P$ that avoids both $x_1$ and $x_2$.
Because edges $yy_1$ and $yy_2$ are in different sets $E_{y,1},\ldots,E_{y,m}$, the cycle $C = y y_1 P' y_2 y$ is unbalanced.
But $C$ avoids $x_1$, $x_2$, and $x$, and so avoids the balancing vertex $v_{xx_1}$ in $\Om/xx_1$, so this is a contradiction.
\end{cproof}

\begin{claim} \label{claim:NotThree} Vertices $x_1, x_2, y_1, y_2$ cannot be chosen so that
$|\{x_1,x_2\}\cap\{y_1,y_2\}|=1$. \end{claim} \begin{cproof} By way of contradiction assume that
$|\{x_1,x_2\}\cap\{y_1,y_2\}|=1$ where, without loss of generality, $x_2=y_1$. As in the proof of Claim
\ref{claim:NotFour}, any $x_1$-$x_2$ path $P$ in $\Omega-\{x,y\}$ must contain $y_2$.
Thus there is a $y_1$-$y_2$ path $P'$
properly contained in $P$ and avoiding $x_1$, leading to the contradiction that there is an unbalanced cycle in $\Om/xx_1$ avoiding the balancing vertex $v_{xx_1}$. \end{cproof}

By Claims \ref{claim:NotFour} and \ref{claim:NotThree}, every choice of $x_1$, $x_2$, $y_1$, $y_2$ has (without loss of
generality) $x_1=y_1$ and $x_2=y_2$. This implies that $X=\{x_1,x_2\}$ = $Y=\{y_1,y_2\}$, since otherwise we may choose
a third vertex in $X$ or $Y$ so that $|\{x_1,x_2\}\cap\{y_1,y_2\}|\in\{0,1\}$. Let $e'$ and $e''$ be respectively the
$xx_1$- and $xx_2$-edges in $\Omega$ and let $f'$ and $f''$ be the $yx_1$- and $yx_2$-edges in $\Omega$. It cannot be
that $e'$ and $f'$ are in the same unbalancing equivalence class $E_j\in\{E_1,\ldots, E_m\}$ because then $\Omega-x_2$
would be balanced. Similarly $e''$ and $f''$ are not in the same equivalence class. Because $\Omega$ is 3-connected, there is an $x_1x_2$-path $P$ in $\Omega-\{x,y\}$. The subgraph of $\Omega$ on edges
$E(P)\cup\{e,e',e'',f',f''\}$ is a subdivision of $K_4$ without a balanced triangle, a contradiction. \end{proof}

\subsection{Unavoidable topological subgraphs}

As is often the case with graphs, minors are harder to work with than subgraphs.
In this section we prove Theorem \ref{MT:Unavoidsubdivisions}, an analogue of Theorem \ref{MT:UnavoidableMinors} for biased topological subgraphs.
We also prove a result on unavoidable biased topological subgraphs for almost-balanced biased graphs.

Let $\Pr$ denote the biased graph whose underlying graph is the triangular prism, with just its two triangles balanced
(Figure \ref{fig:T2PrimeSplit}).
Let $L$ denote the matching of three edges linking the two triangles of $\Pr$.
Observe that $\Pr/L\cong\mathsf T_2$.
Let
$\Pr_2$
and
$\Pr_1$,
respectively, be the biased graphs obtained from $\Pr$ by contracting 1 and 2 edges of $L$, respectively (so $\Pr_2$ has two edges of $L$ remaining, and $\Pr_1$ has just one edge of $L$ remaining from $\Pr$).
Set $\Tt_0 = \Gg_0 \cup \{\Pr, \Pr_1, \Pr_2\}$.
Theorem \ref{MT:Unavoidsubdivisions} guarentees that at least one of the biased graphs in $\Tt_0$ is unavoidable as a biased topological subgraph in 2-connected properly unbalanced biased graphs.

\begin{figure}[tbp]
\begin{center}
\includegraphics[scale=0.9]{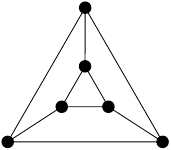}
\end{center}
\caption{$\Pr$ has exactly two triangles, which comprise its set of balanced cycles.}
\label{fig:T2PrimeSplit}
\end{figure}

\begin{proof}[Proof of Theorem \ref{MT:Unavoidsubdivisions}]
If $\Omega$ is not tangled, then the result follows from Proposition \ref{P:TubeMinor}.
So assume that $\Omega$ is tangled.
By Theorem \ref{T:TangledMinor}, $\Omega$ contains a link minor $(G,\mc B)$ that is either a biased $K_4$ with no balanced triangle or a biased $2C_3$ with no balanced 2-cycle.
In the first case, since $K_4$ is 3-regular $\Om$ contains a subdivision of $\GB$.
In the second case, either $\Omega$ contains as a subgraph a subdivision of $(G,\mc B)$ or $\Omega$ contains a link minor
$(G',\mc B')$ that is 2-connected, has minimum degree 3, and contains an edge $e'$ for which $(G',\mc B')/e' = (G,\mc B)$.
Since $\Omega$ is tangled,
$G'$ is as shown in Figure \ref{F:DoubleTriangleSplitToK4-0}.
\begin{figure}[tbp]
\begin{center}
\includegraphics[scale=0.9]{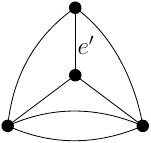}
\end{center}
\caption{The graph $G'$.}\label{F:DoubleTriangleSplitToK4-0} \end{figure}
It is straightforward to check that if $(G,\mc B)\ncong\mathsf T_2$, then $(G',\mc B')$
contains a $K_4$ with no balanced triangle, and so our desired subdivision.
So assume that $(G,\mc B)\cong\mathsf T_2$.
Then either $\Bb'$ consists of a pair of edge-disjoint triangles both avoiding $e'$ or $\Bb'$ consists of a pair of 4-cycles each of which contain $e'$ but are otherwise edge-disjoint.
In the latter case, we again have a biased $K_4$ with no balanced triangle as a subgraph, and so are done.
In the former case, $(G',\Bb') \cong \Pr_1$.
Thus either $\Omega$ contains a subdivision of $\Pr_1$
or
$\Omega$ contains as a link minor $(G'',\mc B'')$ which has minimum degree 3 and an edge $e''$ for which $(G'',\mc B'')/e'' \cong \Pr_1$.
Either $\Bb''$ consists of a pair of 4-cycles sharing just $e''$ or $\Bb''$ consists of a pair of edge-disjoint triangles avoiding $\{e',e''\}$.
Thus either $(G'',\Bb'')$ contains as a subdivision a biased $K_4$ with no balanced triangle, in which case we are done,
or
$(G'',\mc B'')$ is isomorphic to $\Pr_2$.
In the latter case, either $\Omega$ contains a subdivision of $\Pr_2$ or $\Om$ contains a link minor $(G''', \Bb''')$ with minimum degree 3 and an edge $e'''$ for which $(G''',\Bb''')/e''' \cong \Pr_2$.
Thus $\Bb'''$ either consists of a pair of disjoint triangles both avoiding $\{e',e'',e'''\}$ or a pair of 4-cycles sharing just $e'''$.
But if $\Bb'''$ consists of a pair of 4-cycles, then $(G''',\Bb''')$ contains a pair of vertex-disjoint unbalanced triangles, contradicting the fact that $\Om$ is tangled.
Hence it must be the case the $\Bb'''$ consists of a pair of disjoint triangles, so $(G''', \Bb''') \cong \Pr$.
\end{proof}

We now show that every almost-balanced biased graph containing a contrabalanced theta contains as a biased topological subgraph one of the biased graphs in following collection.
Denote the graph obtained from $2C_3$ by deleting an edge by $2C_3 \bs e$.
The graph $2C_3\bs e$ is obtained from
the tube $2C_4''$ by contracting one of its non-doubled links, so the cycles of $2C_3\bs e$ are in bijective correspondence with the cycles of $2C_4''$.
Thus there are exactly three biased graphs $(2C_3\bs e,\mc B)$ without a balanced 2-cycle, each obtained as a single-edge contraction of $\mathsf B_0$, $\mathsf B_1$, or $\mathsf B_2$ (Figure \ref{F:BiasedTube}).
We denote these biased graphs by
$\mathsf B_0'$, $\mathsf B_1'$, and $\mathsf B_2'$, respectively.
Recall that $\mathsf D_{1,0} =(K_4,\Bb)$ where $\Bb$ consists of exactly one balanced triangle; $D_{1,0}$ has a unique balancing vertex.

\begin{prop} \label{P:UniqueBalancingVertex}
Let $(G,\mc B)$ be a 2-connected biased graph that contains a contrabalanced theta subgraph and a unique balancing vertex after removing joints.
Then $(G,\mc B)$ contains a subdivision of $\mathsf
D_{1,0}$, $\mathsf B_0'$, $\mathsf B_1'$, or $\mathsf B_2'$. \end{prop}

\begin{proof}
Denote by $nK_2$ the graph consisting of two vertices with $n$ links between them.
Since $(G,\mc B)$ contains a contrabalanced theta
subgraph, it contains a subdivision of $(nK_2,\emptyset)$ for some $n\geq 3$.
Let $K$ be such a subdivision in $(G,\mc B)$ with $n$ as large as possible.
Let $u$ be the balancing vertex of $\GB$.
One of the two degree-$n$ vertices of $K$ is $u$;
let $v$ be the other degree-$n$ vertex of $K$.
Then $K$ is the union of $n$ internally disjoint $u$-$v$-paths $P_1,\ldots,P_n$.
By assumption $\GB$ does not have the structure described in Proposition \ref{prop:Zas_fattheta}.
Thus there is a path $P$ in $G$ internally disjoint from $K$
with both its ends in $K$ such that either
\begin{itemize}
\item both ends of $P$ are internal vertices of two distinct paths $P_i$ and $P_j$, or
\item $P$ has $u$ as one end, an internal vertex of one of the paths $P_i$ as its other end, and the cycle contained in $P \cup P_i$ is unbalanced.
\end{itemize}
In the first case, $P_i \cup P_j \cup P$ is a theta subgraph with its cycle $P_i \cup P_j$ unbalanced and its cycle avoiding $u$ balanced.
By the theta property the cycle in $P_i \cup P_j \cup P$ avoiding $v$ is unbalanced.
Thus $K \cup P$ contains a subdivision of $D_{1,0}$.
In the second case, $K \cup P$ contains a subdivision of $\mathsf B_0'$, $\mathsf B_1'$, or $\mathsf B_2'$.
\end{proof}

\subsection{Confining inequivalence to a small minor}

We can now prove Theorem \ref{MT:InequivalenceLocalized}.

\begin{figure}[tbp]
\begin{center}
\includegraphics[scale=0.95]{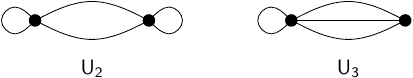}
\end{center}
\caption{Two contrabalanced biased graphs
representing $U_{2,4}$.}
\label{fig:U24unlabelled}
\end{figure}

Two biased graphic representations of $U_{2,4}$ are $\mathsf U_2$ and $\mathsf{U_3}$, shown in Figure \ref{fig:U24unlabelled}; all cycles in each are unbalanced.
Theorem \ref{MT:InequivalenceLocalized} localizes switching inequivalence of gain functions on a small minor: if not a biased graph in $\Gg_0$ then on both $\mathsf U_2$ and $\mathsf U_3$.

\begin{proof}[Proof of Theorem \ref{MT:InequivalenceLocalized}]
By Theorem \ref{MT:Unavoidsubdivisions}, $(G,\mc B)$ has a biased subgraph $(G_0,\mc B_0)$ that is a subdivision of a member of
$\mc G_0\cup\{\Pr_1, \Pr_2, \Pr\}$. Since $G$ is 2-connected and
loopless, there is a sequence of 2-connected subgraphs $(G_0,\mc B_0), \ldots, (G_n,\mc B_n)$ such that
$(G_n,\mc B_n)=(G,\mc B)$ and $(G_{i+1}, \ab \mc B_{i+1}) \ab = (G_i,\mc B_i)\cup P_i$ for some path $P_i$ that is internally
disjoint from $G_i$. Let $\vp_i$ and $\psi_i$ be the $\Gamma$\-/gain functions induced by $\vp$ and $\psi$ on $G_i$. If
$\vp_0$ and $\psi_0$ are switching inequivalent (resp., switching-and-scaling inequivalent in the case that $\Gamma$ is the additive group of a field), then the result
follows by Proposition \ref{P:InequivalencePreservedInContraction}. Otherwise, there is an integer $t\in\{0,\ldots,n-1\}$ such
that $\vp_i$ and $\psi_i$ are switching (resp., switching-and-scaling) equivalent for $i\leq t$ and
$\vp_{t+1}$ and $\psi_{t+1}$ are switching (resp., switching-and-scaling) inequivalent.
Let $e$ be an edge in path $P_t$. Since $G_{t+1}$ is 2-connected, there is a spanning tree $T_{t+1}$ of $G_{t+1}$ that does not contain $e$.
Normalize $\vp_{t+1}$ and $\psi_{t+1}$ on $T_{t+1}$.
Let $T_t$ be $T_{t+1}$ restricted to $G_t$.
Then $T_t$ is a spanning tree of $G_t$.
Since $\vp_{t+1}$ and $\psi_{t+1}$ are normalized on $T_{t+1}$, $\vp_t$ and $\psi_t$ are normalized on $T_t$.
Since $\vp_t$ and $\psi_t$ are switching (-and-scaling) equivalent while $\vp_{t+1}$ and $\psi_{t+1}$ are switching (-and-scaling) inequivalent, by Proposition \ref{P:NormalizationUnique} $\vp_t=\psi_t$ (resp., $\vp_t = a\psi_t$ for some scalar $a$) while $\vp_{t+1}$ and $\psi_{t+1}$ are equal (resp., equal up to scaling) everywhere except at $e$.
Because $(G_0,\Bb_0)$ is unbalanced, $(G_t,\Bb_t)$ is unbalanced, so neither $\vp_t$ nor $\psi_t$ is trivial (that is, neither assigns the identity element of $\Gamma$ to every edge).
Hence for every cycle $C$ of $G_{t+1}$ containing path $P_t$,
$\vp_{t+1}(C)\neq\psi_{t+1}(C)$
(resp.\ if $a$ is a scalar such that $\vp_t = a\psi_t$ then $\vp_{t+1}(C) \neq a\psi_{t+1}(C)$).
Since $\vp_{t+1}$ and $\psi_{t+1}$ are $\Gamma$\-/realizations of $(G_{t+1},\mc
B_{t+1})$, it must be that every such cycle $C$ is unbalanced. Extend $P_t$ in $G_{t+1}$ to a path $P$ that is
internally disjoint from $G_0$ but whose endpoints are both on $G_0$. Let $\vp'$ and $\psi'$ be $\vp_{t+1}$ and
$\psi_{t+1}$ restricted to the biased graph $(G_0,\mc B_0)\cup P$. Again, $\vp'$ and $\psi'$ are equal (resp., equal up to scaling) on every edge of
$(G_0,\mc B_0)\cup P$ save for the edge $e$
and every cycle $C$ in $(G_0,\mc B_0)\cup P$
containing $e$ is therefore unbalanced. Now in $(G_0,\mc B_0)\cup P$ there is a
link minor $\Omega=((G_0,\mc B_0)\cup P)/K\bs D$ for which $\Omega\bs e$ is in $\Gg_0$ or $\Omega\bs e/f$ is in $\Gg_0$ for some link $f$.
The possibilities for $\Omega$ are as shown in Figure \ref{F:InequivalenceLocalized}.
\begin{figure}[tbp] \begin{center}
\includegraphics[scale=0.9]{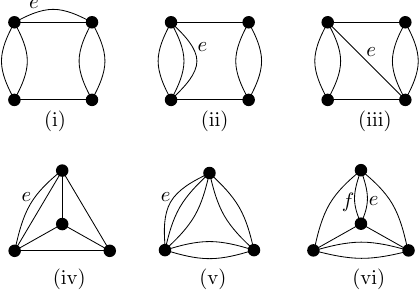}
\end{center} \caption{The possibilities for $\Omega$.} \label{F:InequivalenceLocalized}
\end{figure}

Re-normalize $\vp'$ and $\psi'$ on a spanning tree of $(G_0,\mc B_0)\cup(P-e)$ that contains the contraction set $K$
and let $\vp|_\Omega$ and $\psi_\Omega$ be the induced gain functions on $\Omega$. Again, $\vp|_\Omega$
and $\psi|_\Omega$ are equal (equal up to scaling) on each edge of $(G_0,\mc B_0)\cup(P-e)$ but differ (differ even after scaling) on edge $e$, and every cycle
containing $e$ is unbalanced.
The first outcome of the theorem holds in cases (i), (ii), (iv), and (v) of Figure \ref{F:InequivalenceLocalized} while the second outcome holds in cases (iii) and (vi).
A single frame-type contraction of a joint is necessary in order to obtain $\mathsf U_2$ in case (vi), but no other contraction of a joint is required.
Thus all minors but this one are link minors.
\end{proof}

The following observation will be used in the proof of Theorem \ref{T:ProjectiveIsSwitching}.

\begin{cor} \label{cor:ofInequivalenceLocalized}
If $\GB$ is tangled, then the first outcome of Theorem  \ref{MT:InequivalenceLocalized} holds.
\end{cor}

\begin{proof}
If the first outcome of Theorem \ref{MT:InequivalenceLocalized} does not hold, then the biased graph $\Omega$ in the proof of Theorem \ref{MT:InequivalenceLocalized} is that of either case (iii) or (vi) of Figure \ref{F:InequivalenceLocalized}.
Each of these contain a pair of vertex-disjoint unbalanced cycles, and so cannot occur if $(G,\mc B)$ is tangled.
\end{proof}

\section{Representations of unavoidable minors} \label{sec:unavoidableminors}

In this section we examine the relationship between gain functions on the biased graphs in $\Gg_0$ (along with a few other small biased graphs) and matrix representations of their associated frame and lift matroids.
In Section \ref{S:SwitchingAndProjective} we show that for each biased graph $\Om \in \Gg_0$, a pair of canonical
representations of $F(\Om)$ (resp., $L(\Om)$) are projectively equivalent \iiff their associated gain functions are switching equivalent (resp., switching-and-scaling equivalent).
In Section \ref{sec:Base_case_graphs_all_reps_are_canonical} we show that for each biased graph $\Om \in \Gg_0$ every $\FF$\-/representation of each of $F(\Om)$ and $L(\Om)$ is projectively equivalent to a canonical $\FF$\-/representation particular to $\Om$.

\subsection{Switching and projective equivalence} \label{S:SwitchingAndProjective}

Let $G$ be a graph and let $\FF$ be a field.
Recall that for an $\FF^\times$-gain function $\vp$ and an $\FF^+$-gain function $\psi$, we denote by $A_F(G,\vp)$ and $A_L(G,\psi)$, respectively, the canonical frame and lift matrices defined by $(G,\vp)$ and $(G,\psi)$, resp., as described in Section \ref{subsec:gaingraphscanonicalreps}.
Our starting point is the following result of Zaslavsky.

\begin{prop}[Zaslavsky {\cite{Zaslavsky:BG4}}]
\label{TomProjConRep}
\label{P:switch_equiv_implies_proj_equiv}
Let $G$ be a graph and let $\FF$ be a field.
Let $\fI$ and $\psi$ be $\bb F^\times$- (resp., $\bb F^+$-) gain functions on $G$.
If $\vp$ and $\psi$ are switching equivalent (resp., switching-and-scaling equivalent) then their canonical matrix representations are projectively equivalent.
\end{prop}

\begin{proof}
Suppose $\fI$ and $\psi$ are $\FF^\times$\-/gain functions on $G$ and $\eta$ is a switching function with $\fI^\eta=\psi$.
Let $V(G)=\{v_1$, $v_2$, \ldots, $v_{|V(G)|}\}$.
Let $T$ be
the diagonal matrix with rows and columns indexed by $V(G)$ in which diagonal entry $T_{ii}$ is $\eta(v_i)$, and let
$S$ be the $|E(G)| \times |E(G)|$ diagonal matrix with diagonal entries $S_{jj} = \eta(v_i)\inv$ if vertex $v_i$ is the
tail of edge $e_j$. Then $T A_F(G,\fI) S = A_F(G,\psi)$.

Now suppose $\fI$ and $\psi$ are $\FF^+$\-/gain functions, $\eta$ is a switching function, and that there is a scalar $s \in \FF^\times$ so that
$s\fI^\eta=\psi$. Let $T$ be the $(n+1) \times (n+1)$ matrix whose first row is
$$(s \ s\eta(v_1) \ s\eta(v_2) \ \cdots \ s\eta(v_n) ),$$
first column is
$(s \ 0 \ 0 \ \cdots 0)^T$,
and with the $n \times n$ identity matrix as the submatrix consisting of its remaining rows and columns.
Let $S$ be the diagonal matrix with $s_{11}=1/s$ and all other $s_{ii}=1$.
Then $T A_{L}(G,\fI)S = A_{L}(G,\psi)$.
\end{proof}

The proof of Proposition \ref{P:switch_equiv_implies_proj_equiv} shows that if $\vp$ and $\psi$ are switching equivalent $\FF^\times$-gain functions, then a diagonal matrix $T$ provides witness to the projective equivalence of $A_F(G,\vp)$ and $A_F(G,\psi)$.
The converse is also true.
A similar statement holds for canonical lift matrices.


\begin{lem} \label{P:switch_equiv_iff_diagonal}
Let $G$ be a loopless graph and let $\FF$ be a field.

\textup{(i)} Let $\vp$ and $\psi$ be $\FF^\times$-gain functions on $G$ and assume $A_F(G,\fI)$ and $A_F(G,\psi)$ are projectively equivalent. Then $\vp$ and $\psi$ are switching equivalent \iiff there exists a nonsingular diagonal matrix $T$ and a diagonal column-scaling matrix $S$ such that $TA_F(G,\vp)S=A_F(G,\psi)$.

\textup{(ii)} Let $\vp$ and $\psi$ be $\FF^+$-gain functions on $G$ and assume $A_L(G,\fI)$ and $A_L(G,\psi)$ are projectively equivalent. Then $\vp$ and $\psi$ are switching-and-scaling equivalent \iiff there exists a nonsingular matrix $T$ and a diagonal column-scaling matrix $S$ such that $TA_L(G,\vp)S=A_L(G,\psi)$, where removing the row and column of $T$ indexed by $v_0$ leaves an identity matrix, and all but the first entry of column $v_0$ consists of zeros.
\end{lem}

\begin{proof}
(i)
Put $A = A_F(G,\fI)$ and $B = A_F(G,\psi)$.
If $\vp$ and $\psi$ are switching equivalent, then the proof of Proposition \ref{P:switch_equiv_implies_proj_equiv}
shows that $A$ and $B$ are projectively equivalent via diagonal nonsingular matrices $T$ and $S$. Conversely, suppose that
$B=TAS$ where $T$ and $S$ are both diagonal and nonsingular. Since $T$ is diagonal, row $i$ of $TA$ is obtained by multiplying row $v_i$ of $A$ by $T_{ii}$. Since both $A$ and $B$ are canonical frame representations, both have 1 in
position $v_i$ of column $e_j$ whenever vertex $v_i$ is the tail of edge $e_j$. Hence the diagonal elements $S_{jj}$ of
$S$ satisfy $S_{jj} = T_{ii}\inv$, where $v_i$ is the tail of $e_j$. Thus $\vp$ and $\psi$ are switching equivalent via
$\eta(v_i) = T_{ii}$ for each $v_i \in V(G)$.

(ii)
Put $A = A_L(G,\fI)$ and $B = A_L(G,\psi)$.
If $\vp$ and $\psi$ are switching-and-scaling equivalent, then the proof of Proposition \ref{P:switch_equiv_implies_proj_equiv}
shows that $A$ and $B$ are projectively equivalent via matrices $T$ and $S$ of the required forms.
Conversely, suppose that $B=TAS$ where $T$ and $S$ are of the form in statement (ii).
Put $n=|V(G)|$.
Index the rows and columns of $T$ by $\{v_0,v_1, \ldots, v_n\}$, where as usual row $v_0$ of $A$ contains the gains assigned by $\vp$ and for $i \in \{1, \ldots, n\}$ row $v_i$ corresponds to vertex $v_i$.
Let $\eta : V(G) \to \FF^+$ be the switching function defined by $\eta(v_i) = T_{v_0 v_i}$ for $i \in \{1, \ldots, n\}$.
Then $\br{T_{v_0 v_0}} \vp^\eta=\psi$.
\end{proof}

We now proceed with our examination of each biased graph in $\Gg_0$.

\begin{lem} \label{lem:2C3_equil_implies_switching_equiv}
Let $\fI \: \vec E(2C_3) \to \FF^\times$ be a gain function
with $\Bb_\vp$ containing no 2-cycle, and let $\psi \: \vec E(2C_3) \to \FF^\times$ be another gain function. Then $\fI$
and $\psi$ are switching equivalent \iiff $A_F(2C_3,\fI)$ and $A_F(2C_3, \psi)$ are projectively equivalent.
\end{lem}
\begin{proof}
If $\fI$ and $\psi$ are switching equivalent, then $A_F(2C_3,\fI)$ and $A_F(2C_3, \psi)$ are projectively
equivalent by Proposition \ref{P:switch_equiv_implies_proj_equiv}. To prove the converse, let $A = A_F(2C_3,\fI)$ and $B = A_F(2C_3,\psi)$ be a pair of projectively equivalent canonical frame matrices.
By normalizing on the spanning tree with
edge set $\{e_1, e_3\}$, by Proposition \ref{P:switch_equiv_implies_proj_equiv} we may assume that $\fI$ assigns gains to $E(2C_3)$ as shown in Figure \ref{fig:2C3}.
\begin{figure}[tbp]
\begin{center}
\includegraphics[scale=0.9]{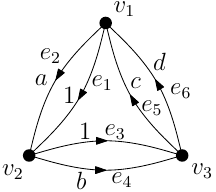}
\end{center}
\caption{Gain function $\fI$ on $2C3$}
\label{fig:2C3}
\end{figure}
Then
\begin{linenomath}
\begin{equation} \label{eqn:A_F(2C_3,phi)}
A = \bordermatrix{
& e_1 & e_2 & e_3 & e_4 & e_5 & e_6 \cr
v_1 & 1 & 1 & 0 & 0 & -c & -d \cr
v_2 & -1 & -a & 1 & 1 & 0 & 0 \cr
v_3 & 0 & 0 & -1 & -b & 1 & 1 \cr }
\end{equation}
\end{linenomath}
Since $\Bb_\fI$ contains no 2-cycle, $a \not=1 \not= b$ and $c \not= d$. Since $A$ and $B$ are projectively equivalent,
there is a nonsingular matrix $T$ and a diagonal matrix $S$ so that $TA=BS$.
Denote by $t_i$ row $i$ of $T$, and by $e_j$ column $j$ of $A$.
Entry $(BS)_{ij}=0$ \iiff entry $(TA)_{ij}=0$; consider the dot products $t_i \cdot e_j = 0$, where
$(i,j) \in \{(3,1),(3,2),(1,3),(1,4),(2,5),(2,6)\}$. The product $t_3 \cdot e_1 = 0$ implies $T_{31}=T_{32}$, and $t_3
\cdot e_2 =0$ implies $T_{31}=aT_{32}$. Together these imply (since $a \not= 1$) that $T_{32}=T_{31}=0$. Similarly,
$t_1 \cdot e_3 = t_1 \cdot e_4 = 0$ imply $T_{13}=T_{12}=0$, and $t_2 \cdot e_5 = t_2 \cdot e_6 =0$ imply
$T_{21}=T_{23}=0$. Hence $T$ is diagonal, and so by Lemma \ref{P:switch_equiv_iff_diagonal}, $\fI$ and $\psi$ are
switching equivalent.
\end{proof}

\begin{lem} \label{lem:2C3_equil_implies_switching_equiv_lift_version}
Let $\fI \: \vec E(2C_3) \to \FF^+$ be a gain function with $\Bb_\vp$ containing no 2-cycle and let $\psi\:\vec
E(2C_3) \to \FF^+$ be another gain function. Then $\fI$ and $\psi$ are switching-and-scaling equivalent \iiff $A_L(2C_3,\fI)$ and $A_L(2C_3, \psi)$ are projectively equivalent.
\end{lem}

\begin{proof}
If $\fI$ and $\psi$ are switching-and-scaling equivalent, then their
associated lift matrices are projectively equivalent by Proposition \ref{P:switch_equiv_implies_proj_equiv}.
Conversely, let $A = A_L(2C_3,\fI)$ and $B = A_L(2C_3,\psi)$ be a pair of projectively equivalent canonical lift matrices.
Normalizing on spanning tree $\{e_1, e_2\}$, and scaling if necessary, we may
assume that $\fI$ assigns gains to $2C_3$ as shown in Figure \ref{fig:2C3_add_gains}.
\begin{figure}[tbp]
\begin{center}
\includegraphics[scale=0.9]{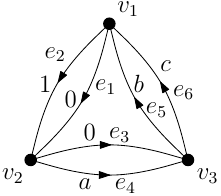}
\end{center}
\caption{Gain function $\fI \: \vec E(2C_3) \to \FF^+$.}
\label{fig:2C3_add_gains}
\end{figure}
Then
\begin{linenomath}
\begin{equation}
\label{eqn:A_L(2C_3,phi)}
A = \bordermatrix{
& e_1 & e_2 & e_3 & e_4 & e_5 & e_6 \cr
 v_0 & 0 & 1 & 0 & a & b & c \cr
 v_1 & 1 & 1 & 0 & 0 & -1 & -1 \cr
 v_2 & -1 & -1 & 1 & 1 & 0 & 0 \cr
 v_3 & 0 & 0 & -1 & -1 & 1 & 1 \cr
}
\end{equation}
\end{linenomath}
Since $\fI$
has no balanced cycles of length 2, $a \not= 0$ and $b \not= c$. There is a non-singular matrix $T$ and a diagonal matrix $S$ so that $TA=BS$. As with
$\fI$, by switching and scaling we may assume $\psi$ assigns gains to $\vec E(2C_3)$ as in Figure \ref{fig:2C3_add_gains},
replacing $a$, $b$, and $c$ with $x$, $y$, and $z$, respectively. Then, denoting elements $S_{ii}$ of $S$ by $s_i$, we
have
\begin{linenomath}
\[ BS =
\begin{pmatrix}
 0 & s_2 & 0 & s_4 x & s_5 y & s_6 z \\
 s_1 & s_2 & 0 & 0 & -s_5 & -s_6 \\
 -s_1 & -s_2 & s_3 & s_4 & 0 & 0 \\
 0 & 0 & -s_3 & -s_4 & s_5 & s_6
\end{pmatrix}
\]
\end{linenomath}
This gives 24 relations among the members
of $T$, one for each dot product $t_i \cdot e_j$, where $t_i$ is the $i$th column of $T$ and $e_j$ is the $j$th column
of $A$.
The eight relations $t_i \cdot e_j=0$ yield $T_{12}=T_{13}=T_{14}$, $T_{21}=T_{31}=T_{41}=0$, $T_{23}=T_{24}$,
and $T_{42}=T_{43}$.
Now after establishing these relations, $t_3\cdot e_5=0$ yields $T_{32}=T_{34}$ and so
\begin{linenomath}
\[
T = \begin{pmatrix}
T_{11} & T_{12} & T_{12} & T_{12} \\
0 & T_{22} & T_{23} & T_{23} \\
0 & T_{32} & T_{33} & T_{32} \\
0 & T_{42} & T_{42} & T_{44} \\
\end{pmatrix}
\]
\end{linenomath}
Now the relations $s_1=t_2\cdot e_1$, $s_2=t_2\cdot e_2$,
$s_3=t_3\cdot e_3$, $s_4=t_3\cdot e_4$, $s_5=t_4\cdot e_5$, $s_6=t_4\cdot e_6$, $-s_2=t_3\cdot e_2$, and $-s_3=t_4\cdot
e_3$ yield $s_1=s_2=s_3=s_4=s_5=s_6$.
Hence the relation $s_2=t_1\cdot e_2$ yields $T_{11}=s_1$. Now the relations
$t_1 \cdot e_4 = s_1 x$, $t_1 \cdot e_5 = s_1 y$, and $t_1 \cdot e_6 = s_1 z$ yield $a=x$, $b=y$ and $c=z$.
This implies that $\vp$ and $\psi$ are switching equivalent after scaling.
\end{proof}

\begin{lem} \label{lem:C_3_frame_lift_not_ech_equiv}
Let $\fI \: \vec E(2C_3) \to \FF^\times$ and $\psi \: \vec E(2C_3) \to \FF^+$ be gain functions, neither
of which yield a balanced 2-cycle. Then $A_F(2C_3,\fI)$ and $A_L(2C_3,\psi)$ are not projectively equivalent.
\end{lem}

\begin{proof}
As in previous cases, without loss of generality we may assume that $\fI$ assigns gains to $2C_3$ as in Figure \ref{fig:2C3}, and
$\psi$ as in Figure \ref{fig:2C3_add_gains}, replacing $a$ with $x$, $b$ with $y$, and $c$ with $z$.
Then $A_F(2C_3,\fI)$ is the matrix of (\ref{eqn:A_F(2C_3,phi)}), and $A_L(2C_3,\psi)$ is the matrix of (\ref{eqn:A_L(2C_3,phi)}) with $a$, $b$, and $c$ replaced by $x$, $y$, and $z$ respectively.
Recall (Section \ref{sec:on_canonical_lift_representations}) that $A_L(2C_3,\psi)$ is projectively equivalent to the matrix $A_L^{-v}(2C_3,\psi)$ obtained by deleting the row indexed by any vertex $v \in V(2C_3)$ from $A_L(2C_3,\psi)$.
Let $B = A_L^{-v_3}(2C_3,\psi)$.
Then $B$ is a full-rank canonical lift matrix projectively equivalent to $A_L(G,\psi)$; in particular, $A$ and $B$ both have three rows.
Now suppose for a contradiction that there exists a non-singular
matrix $T$ and a diagonal matrix $S$ so that $TA = BS$. Writing $S_{ii} = s_i$, and denoting row $i$ of $T$ by $t_i$
and column $j$ of $A$ by $e_j$, we have $t_2 \cdot e_1 = s_1$, $t_2 \cdot e_2 = s_2$, $t_2 \cdot e_3 = 0$, and $t_2
\cdot e_4 = 0$. Together these imply that $T_{22}=T_{23}=0$ and that $T_{21}=s_1=s_2$. Moreover, we have $t_3 \cdot e_1
= -s_1$, $t_3 \cdot e_2 = -s_2$, $t_3 \cdot e_5 = 0$, and $t_3 \cdot e_6 = 0$. Since $s_1=s_2$, $a \not=0,1$, and $c
\not= d$, these imply that $T_{31} = T_{32} = T_{33} = 0$, which implies that $T$ is singular, a contradiction.
\end{proof}

\begin{lem} \label{lem:nobaltriangle}
Let $\fI \: \vec E(K_4) \to \FF^\times$ be a gain function with $\Bb_\fI$ containing no 3-cycle, and let $\psi \: \vec
E(K_4) \to \FF^\times$ be another gain function. Then $\fI$ and $\psi$ are switching equivalent \iiff $A_F(K_4,\fI)$
and $A_F(K_4,\psi)$ are projectively equivalent.
\end{lem}
\begin{proof}
By switching we may assume that $\fI$ and $\psi$ are both equal to the identity on a $K_{1,3}$\-/subgraph $Y$. This
allows us to consider $\vp$ and $\psi$ as gain functions on $\nabla_Y K_4=2C_3$. Now $\nabla_Y(K_4,
\Bb_\fI)=(2C_3,\Bb_\vp)$ and $\nabla_Y(K_4, \Bb_\psi)=(2C_3,\Bb_\psi)$. Since $\fI$ has no balanced triangles in $K_4$,
neither has it any balanced 2-cycles in $2C_3$.
Thus by Lemma \ref{lem:2C3_equil_implies_switching_equiv}, $\fI$
and $\psi$ are switching equivalent \iiff $A_F(2C_3,\fI)$ and $A_F(2C_3,\psi)$ are projectively equivalent.
Thus Propositions \ref{P:Whittle} and  \ref{prop:gain_functions_DeltaYs} imply that $\fI$ and $\psi$ are switching equivalent
\iiff $A_F(K_4,\fI)$ and $A_F(K_4,\psi)$ are projectively equivalent. \end{proof}

Using $Y$-$\Delta$ exchanges as in the proof of Lemma \ref{lem:nobaltriangle} yields Lemmas
\ref{lem:add_gains_K4_switch_eq_iff_equal} and \ref{lem:K4_frame_lift_not_ech_equiv}.

\begin{lem} \label{lem:add_gains_K4_switch_eq_iff_equal}
Let $\fI \: \vec E(K_4) \to \FF^+$ be a gain function with $\Bb_\fI$ containing no 3-cycle and let $\psi \: \vec E(K_4)
\to \FF^+$ be another gain function. Then $\fI$ and $\psi$ are switching-and-scalaing equivalent \iiff $A_L(K_4,\fI)$
and $A_L(K_4,\psi)$ are projectively equivalent.
\end{lem}

\begin{lem} \label{lem:K4_frame_lift_not_ech_equiv}
Let $\fI \: \vec E(K_4) \to \FF^\times$ and $\psi \: \vec E(K_4) \to \FF^+$ be gain functions neither of which yield a balanced 3-cycle.
Then $A_F(K_4,\fI)$ and $A_L(K_4,\psi)$ are not projectively equivalent.
\end{lem}

\begin{lem} \label{lem:4tube_equil_implies_switching_equiv}
Let $\fI \: \vec E(2C_4'') \to \FF^\times$ be a gain function with $\mathcal B_\vp$ containing no 2-cycle, and let $\psi \: \vec E(2C_4'') \to \FF^\times$ be another gain function. Then $\fI$ and $\psi$ are switching equivalent \iiff $A_F(2C_4'',\fI)$ and $A_F(2C_4'', \psi)$ are projectively equivalent.
\end{lem}

\begin{proof}
Let $A= A_F(2C_4'',\fI)$ and $B=A_F(2C_4'',\psi)$. If $\fI$ and $\psi$
are switching equivalent, then $A$ and $B$ are projectively equivalent by Proposition
\ref{P:switch_equiv_implies_proj_equiv}. To prove the converse, let $T$ and $S$ be matrices with $TA=BS$, where $T$ is
nonsingular and $S$ is a diagonal matrix scaling the columns of $B$. We may assume without loss of generality that the
edge orientations chosen to define $B$ are the same as those chosen to define $A$; by normalizing on the spanning tree
with edges $e_3, e_4, e_5$, we may assume without loss of generality that $\fI$ assigns gains to $\vec E(K_4)$ as shown in Figure
\ref{fig:4-tube} \begin{figure}[tbp]
\begin{center} \includegraphics[scale=0.9]{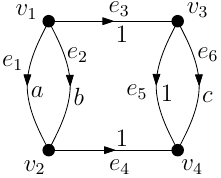} \end{center} \caption{Labelled $2C_4''$ with a normalized gain function.}
\label{fig:4-tube} \end{figure} where $a \not= b$ and $c \not= 1$. Thus
\begin{linenomath}
\[ A = \bordermatrix{ &e_1& e_2 & e_3 & e_4 &
e_5 & e_6\cr v_1 & 1 & 1 & 1 & 0 & 0 & 0 \cr v_2 & -a& -b & 0 & 1 & 0 & 0 \cr v_3 & 0 & 0 & -1 & 0 & 1 & 1 \cr v_4 & 0
& 0 & 0 & -1 & -1 & -c \cr }.\]
\end{linenomath}
Now, each entry $B_{ij}$ of $B$ is zero \iiff entry $(TA)_{ij}=0$. Now for a fixed row
$i$, there are three distinct $j$ such that $t_i \cdot e_j = 0$.
It is a straightforward check that for each $i$, that these three
relations yield $T_{ij}\neq0$ if and only if $i=j$.
For $i=1$, the relations are $0=t_1\cdot e_4$, which implies $T_{12}=T_{14}$; $0=t_1\cdot e_5$, which implies
$T_{13}=T_{14}$; and $0=t_1\cdot e_5$, which implies $T_{13}=cT_{14}$.
Since $c\neq 1$, this implies $T_{12}=T_{13}=T_{14}=0$.
Similarly, the entries of $T$ off its main diagonal in rows 2, 3, and 4 are all zero.
Thus $T$ is diagonal.
By Lemma \ref{P:switch_equiv_iff_diagonal}, $\fI$ and $\psi$ are switching equivalent.
\end{proof}

\begin{lem} \label{lem:4tube_equil_implies_switching_equiv_LIFT}
Let $\fI \: \vec E(2C_4'') \to \FF^+$ be a gain function with $\mathcal B_\fI$ containing no 2-cycle, and let $\psi \: \vec E(2C_4'') \to \FF^+$ be another gain function.
Then $\fI$ and $\psi$ are switching-and-scaling equivalent \iiff $A_L(2C_4'',\fI)$ and $A_L(2C_4'', \psi)$ are projectively equivalent.
\end{lem}

\begin{proof}
The ``only if'' statement again follows from Proposition \ref{P:switch_equiv_implies_proj_equiv}. For the converse, \wolog{}
assume that $\fI(e_1)=\psi(e_1)=1$, $\fI(e_2)=a$, $\psi(e_2)=x$, $\fI(e_3)=\psi(e_3)=0$, $\fI(e_4)=\psi(e_4)=0$,
$\fI(e_5)=\psi(e_4)=0$, $\fI(e_6)=b$, and $\psi(e_6)=y$ (where $2C_4''$ has edges and orientations as in Figure
\ref{fig:4-tube}) such that neither $a$ nor $x$ is 1 and neither $b$ nor $y$ is 0. Let $A=A_L(2C_4'',\fI)$ and
$B=A_L(2C_4'',\psi)$, and let $T$ and $S$ be matrices with $TA=BS$, where $S$ is diagonal (with $s_i=S_{ii}$) scaling
the columns of $B$. Denoting row $i$ of $T$ by $t_i$ and column $j$ of $A$ by $e_j$ we have $t_i \cdot e_j = 0$ for 15
pairs $(i,j)$, three pairs for each row.
It is straightforward to deduce that these relations imply
\begin{linenomath}
\[ T = \begin{pmatrix} T_{11} & T_{12} & T_{12} & T_{12} & T_{12}\\ 0 & T_{22} &T_{23} &
T_{23} & T_{23}\\ 0 & T_{32} & T_{33} & T_{32} & T_{32} \\ 0 & T_{42} & T_{42} & T_{44} & T_{42} \\ 0 & T_{52} & T_{52}
& T_{52} & T_{55} \end{pmatrix}.\]
\end{linenomath}
Now each column $e_j$ has $i,k\geq2$ such that $t_i\cdot e_j=s_j$ and
$t_k\cdot e_j=-s_j$. These 12 relations yield $s_1=s_2=s_3=s_4=s_5=s_6$. The relation $t_1\cdot e_1=s_1$ yields
$T_{11}=s_1$. Now the relation $t_1\cdot e_2=s_1x$ yields $a=x$ and the relation $t_1\cdot e_6=s_1y$ yields $b=y$.
Thus $A$ and $B$ are switching-and-scaling equivalent.
\end{proof}

Two biased graphic representations of $U_{2,4}$ are $\mathsf U_2$ and $\mathsf{U_3}$, shown in Figure \ref{fig:U24unlabelled}, where all cycles in each are unbalanced.
Denote the underlying graphs of $\mathsf{U_2}$ and $\mathsf{U_3}$ by $U_2$ and $U_3$, respectively.

\begin{figure}[tbp]
\begin{center}
\includegraphics[scale=0.9]{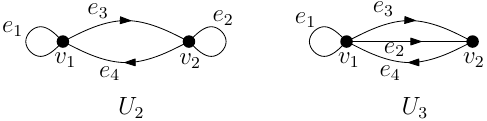}
\end{center}
\caption{Labels and edge orientations for $U_2$ and $U_3$.}
\label{fig:U24}
\end{figure}

\begin{lem} \label{lem:U2Inequivalence}
Let $\fI$ and $\psi$ be $\FF^\times$\-/realizations of $\mathsf{U_2}$. Then $A_F(U_2, \fI)$ and $A_F(U_2,\psi)$ are
projectively equivalent \iiff  $\fI(e_3 e_4) = \psi(e_3 e_4)$.
\end{lem}

\begin{proof}
We may assume ${U_2}$ is labelled with edge orientations as in Figure \ref{fig:U24}.
Matrices $A_F(U_2, \fI)$ and $A_F(U_2, \psi)$ are of the form
\begin{equation} \label{eqn:U24frame} \begin{pmatrix}
 1 & 0 & 1 & -g \\
  0 & 1 & -1 & 1 \end{pmatrix}
\end{equation}
up to scaling columns $e_1$ and $e_2$.
These are in standard form
relative to the basis $\{e_1,e_2\}$ and so are projectively equivalent if and only if entry $g$ is the same for both $A_F(U_2,
\fI)$ and $A_F(U_2, \psi)$. The result follows.
\end{proof}

\begin{lem} \label{lem:T3InequivalenceLift}
Let $\fI$ and $\psi$ be $\FF^+$\-/realizations of $\mathsf{U_3}$. Then $A_L(U_3,\fI)$ and $A_L(U_3,\psi)$ are
projectively equivalent \iiff $\fI|_{\{e_2,e_3,e_4\}}$ and $\psi|_{\{e_2,e_3,e_4\}}$ are switching-and-scaling
equivalent.
\end{lem}

\begin{proof}
We may assume that $\mathsf{U}_3$ is labelled as in Figure \ref{fig:U24}.
If $A=A_L(U_3,\fI)$ and $B=A_L(U_3,\psi)$ are projectively equivalent, then there is an invertible matrix $T$ and
diagonal matrix $S$ such that $TA=BS$.
By switching and scaling we may assume that $\fI(e_1)=\psi(e_1)=1$,
$\fI(e_2)=\psi(e_2)=0$, $\fI(e_3)=\psi(e_3)=1$, $\fI(e_4)=a$, and $\psi(e_4)=x$.
Again writing $s_i$ for entry $S_{ii}$,
\begin{linenomath}
\[
\begin{pmatrix}
T_{11} & T_{12} & T_{13}\\
T_{21} & T_{22} & T_{23}\\
T_{31} & T_{32} & T_{33}\\
\end{pmatrix}
\begin{pmatrix}
1 & 0 & 1 & a\\
0 & 1 & 1 & -1 \\
0 &-1 & -1 & 1
\end{pmatrix}=\begin{pmatrix}
s_1 & 0 & s_3 & s_4x\\
0 & s_2 & s_3 & -s_4 \\
0 &-s_2 & -s_3 & s_4
\end{pmatrix}.\]
\end{linenomath}
This yields $T_{11}=s_1$, $T_{21}=T_{31}=0$, and
$T_{12}=T_{13}$.
That is,
\begin{linenomath}
\[\begin{pmatrix}
s_1 & T_{12} & T_{12}\\
0 & T_{22} & T_{23}\\
0 & T_{32} & T_{33}\\
\end{pmatrix}\begin{pmatrix}
1 & 0 & 1 & a\\
0 & 1 & 1 & -1 \\
0 &-1 & -1 & 1
\end{pmatrix}=\begin{pmatrix}
s_1 & 0 & s_3 & s_4x\\
0 & s_2 & s_3 & -s_4 \\
0 &-s_2 & -s_3 & s_4
\end{pmatrix},\]
\end{linenomath}
which yields $s_1=s_2=s_3=s_4$.
Thus $a=x$, and so $\vp$ and $\psi$ are switching-and-scaling equivalent.
\end{proof}

\subsection{All $\FF$-representations are canonical} \label{sec:Base_case_graphs_all_reps_are_canonical}

Let $\FF$ be a field.
In this section we show that every $\bb F$\-/matrix representation of a frame or lift matroid arising from a biased graph in $\Tt_0$ is projectively equivalent to a canonical representation particular to that biased graph.
Recall that when $\GB$ is a biased graph with no two vertex-disjoint unbalanced cycles, $F\GB = L\GB$, and we denote this common matroid by $M\GB$.

\begin{lem} \label{thm:All_reps_2C3s_are_canonical}
Let $(2C_3,\Bb)$ be a biased graph with no balanced 2-cycle and let $A$ be an $\FF$-matrix representing $M(2C_3,\Bb)$.
Then $A$ is projectively equivalent to a canonical lift matrix particular to $(2C_3,\mc B)$ or to a canonical frame matrix particular to $(2C_3,\mc B)$, but not both.
\end{lem}

\begin{proof}
We may assume that $2C_3$ is labelled and has edge orientations as shown in Figure \ref{fig:2C3_add_gains_copy}.
\begin{figure}[tbp]
\begin{center}
\includegraphics[scale=0.9]{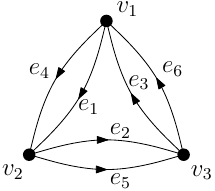}
\caption{Labels and edge orientations for $2C_3$.}
\label{fig:2C3_add_gains_copy}
\end{center}
\end{figure}
Let $A$ be a matrix over $\FF$ representing $M(2C_3,\Bb)$.
If $(2C_3,\Bb) \cong \mathsf T_4$ then $M(2C_3,\Bb)$ is
isomorphic to the cycle matroid of $K_4$, and so
has a projectively unique representation over every field.
Thus if the characteristic of $\bb F$ is two then $A$ is projectively equivalent to the canonical
lift matrix
\begin{linenomath}
\[B =
\left(
    \begin{array}{cccccc}
      1 & 1 & 1 & 0 & 0 & 0 \\
      1 & 0 & 1 & 1 & 0 & 1 \\
      1 & 1 & 0 & 1 & 1 & 0 \\
      0 & 1 & 1 & 0 & 1 & 1 \\
    \end{array}
  \right) \] \end{linenomath}
and if the characteristic of $\bb F$ is not two then $A$ is projectively equivalent to the canonical frame matrix
\begin{linenomath}
\[ C = \left(
    \begin{array}{cccccc}
      1 & 0 & 1 & 1 & 0 & -1 \\
      1 & 1 & 0 & -1 & 1 & 0 \\
      0 & 1 & 1 & 0 & -1 & 1 \\
    \end{array}
  \right).\]
\end{linenomath}

We now claim that there is no canonical frame matrix particular to $\mathsf T_4$ over any field of characteristic two, and that neither is there a canonical lift matrix particular to $\mathsf T_4$ in any field of characteristic different from two.
For, toward a contradiction, suppose $D$ is a canonical frame matrix particular to $\mathsf T_4$ over a field $\FF$ of characteristic two.
Assume the collection of balanced cycles of $\mathsf T_4$ is $\{e_1e_2e_6, e_1e_3e_5, e_2e_3e_4, e_4e_5e_6\}$.
We may assume
\begin{linenomath}
\[ D =
\bordermatrix{
                          & e_1 & e_2 & e_3 & e_4 & e_5 & e_6  \cr
                         &   1 &  0  &  c  &  1   &  0  &  1  \cr
                          &  a &  1  &  0  &  1   &  1 &  0 \cr
                           & 0 &  b &   1  &  0   &  1  &   1
} \]
\end{linenomath}
where $a, b, c \in \FF^\times$ (and we omit the customary negative signs as redundent).
Since $e_1e_2e_6$ is balanced, $ab=1$;
since $e_1e_3e_5$ is balanced, $ac=1$; and
because $e_2e_3e_4$ is balanced, $bc=1$.
These relations imply that $a=b=c$ and so that $a^2=1$.
But this implies $a=1$, and so $D$ does not represent $M(\mathsf T_4)$, a contradiction.

Similarly, suppose for a contradiction that $D$ is a canonical lift matrix particular to $\mathsf T_4$ over a field $\FF$ of characteristic different from two.
Then we may assume
\begin{linenomath}
\[ D =
\bordermatrix{
& e_1 & e_2 & e_3 & e_4 & e_5 & e_6  \cr
& 0 & 0 & 1 & a & b & c \cr
& 1 & 0 & -1 & 1 & 0 & -1 \cr
& -1 & 1 & 0 & -1 & 1 & 0
}
\]
\end{linenomath}
for some nonzero elements $a, b, c \in \FF$, where the second and third rows are the oriented incidence matrix of $2C_3$ with its row corresponding to $v_3$ removed.
Because $e_4e_5e_6$ is balanced, $a+b+c=0$.
Since $e_1e_2e_6$ is balanced, $c=0$;
since $e_1e_3e_5$ is balanced, $1+b=0$; and
because $e_2e_3e_4$ is balanced, $1+a=0$.
These relations imply that $a=b=-1$ and so that $a+b+c=-2 \neq 0$, a contradiction.
This completes the proof in the case that $(2C_3,\Bb) \cong \mathsf{T}_4$.

Now assume $(2C_3,\Bb) \not\cong \mathsf T_4$.
By Proposition \ref{P:Biased2C3s} we may assume that the triangles $\{e_1, e_2, e_3\}$ and $\{e_2, e_3, e_4\}$
are both unbalanced.
Since the only form a 3-circuit takes in $(2C_3,\mc B)$ is a balanced triangle and neither $e_5$ nor $e_6$ forms a triangle with $\{e_2, e_3\}$, $A$ is projectively equivalent to the matrix
\begin{linenomath}
\[ \bordermatrix{
                  & e_1 & e_2 & e_3 & e_4 & e_5 & e_6  \cr
                  & 1 &  0  &  0  &  1   &  1  &  1  \cr
                  & 0 &  1  &  0  &  1   &  a  &  c \cr
                  & 0 &  0 &   1  &  1   &  b  &   d \cr
} .\]
\end{linenomath}
Hence:
\begin{enumerate}[label=\textup{(\roman*)}]
\item
Neither $b$ nor $c$ is 0.
If $b=0$ then $\{e_1, e_2, e_5\}$ is a circuit; if $c=0$ then $\{e_1, e_3, e_6\}$ is a circuit: both contradictions.
\item  $a \not= b$:
If so then $a \not= 1$, as then $e_4$ and $e_5$ would form a parallel pair, a contradiction.
But then $\{e_1, e_4, e_5\}$ is a circuit, also a contradiction.
\item  $b \not= 1$:
If so, then $\{e_2, e_4, e_5\}$ is a circuit, a contradiction.
\item  $c\not=d$:
If so, then $c \not= 1$ as $e_4$ and $e_6$ are not a parallel pair.
But then $\{e_1, e_4, e_6\}$ is a circuit, a contradiction.
\item  $c \not= 1$:
If so, $\{e_3, e_4, e_6\}$ is a circuit, a contradiction.
\item  $a \not= c$:
If so, then $d \not=b$ since $e_5$ and $e_6$ are not a parallel pair.
But then $\{e_3, e_5, e_6\}$ is a circuit, a contradiction.
\item  $b \not= d$:
If so, then $a \not= c$ since $e_5$ and $e_6$ are not a parallel pair.
But then $\{e_2, e_5, e_6\}$ is a circuit, a contradiction.
\end{enumerate}

Suppose there are nonsingular matrices $T$ and $S$ such that $TAS$ is a canonical frame matrix particular to $(2C_3, \Bb)$, where $S$ is diagonal column-scaling matrix.
Then we may assume
\begin{linenomath}
\[ TAS =
\bordermatrix{
                          & e_1 & e_2 & e_3 & e_4 & e_5 & e_6  \cr
                         &   1 &  0  &  -g_1  &  1   &  0  &  -g_4  \cr
                          &  -1 &  1  &  0  &  -g_2   &  1 &  0 \cr
                           & 0 &  -1 &   1  &  0   &  -g_3  &   1 \cr
} \]
\end{linenomath}
for some elements $g_1, \ldots, g_4 \in \FF - \{0,1\}$.
Let us denote row $i$ of $T$ by $t_i$ and column $j$ of $A$ by $e_j$.
Consider the products $t_i \cdot e_j = (TA)_{ij}$ for $1 \leq i \leq 3$, $1 \leq j \leq 6$.
The products $t_3 \cdot e_1 = 0$, $t_1 \cdot e_2 = 0$, and $t_2 \cdot e_3 = 0$ imply $T_{31} = 0$, $T_{12}=0$, and $T_{23} = 0$, respectively.
Since $t_1 \cdot e_1 = -(t_2 \cdot e_1)$, $T_{21}=-T_{11}$.
Similarly, $t_2 \cdot e_2 = -(t_3 \cdot e_2)$ implies $T_{32}=-T_{22}$.
Now $t_3 \cdot e_4 = 0$ implies $T_{33} = -T_{32}$;
$t_1 \cdot e_5 = 0$ implies $T_{13} = -T_{11}/b$;
and finally, $t_2 \cdot e_6 = 0$ implies that $T_{22}=-T_{21}/c$.
Thus $T$ is the matrix
\begin{linenomath}
\[ \begin{pmatrix}
	t & 0 & -t/b \\
	-t & t/c & 0  \\
	0 & -t/c & t/c
\end{pmatrix}
= t
\begin{pmatrix}
	1 & 0 & -1/b \\
	-1 & 1/c & 0  \\
	0 & -1/c & 1/c
\end{pmatrix} \]
\end{linenomath}
for some nonzero $t \in \FF$.
Since $T$ has determinant
$t^3(1/c^2 - 1/bc)$, $T$ is non-singular if and only if $b \neq c$.
Assuming $b \neq c$, and taking $t=1$,
\begin{linenomath}
\[TA=\left(
\begin{array}{cccccc}
1 & 0 & -{1}/{b} & {(b-1)}/{b} & 0 & {(b-d)}/{b} \\
-1 & {1}/{c} & 0 & {(1-c)}/{c} & {(a-c)}/{c} & 0 \\
0 & -{1}/{c} & {1}/{c} & 0 & {(b-a)}/{c} & {(d-c)}/{c}
\end{array}
\right).\]\end{linenomath}
By claims 1-7 above $TA$ has exactly two nonzero entries in each column, so scaling the columns of $TA$ appropriately yields a canonical frame matrix.
Thus $A$ is projectively equivalent to a canonical frame matrix particular to $(2C_3, \Bb)$ if and only if $b \neq c$.

Now let $T$ be a nonsingular matrix such that $TAS$ is a canonical lift matrix particular to $(2C_3,\Bb)$, for some diagonal column-scaling matrix $S$.
We may assume that $TAS$ is of the form
\begin{linenomath}
\[ TAS =
\begin{pmatrix}
0 & 0 & 1 & g_1 & g_2 & g_3 \\
1 & 0 & -1 & 1 & 0 & -1 \\
-1 & 1 & 0 & -1 & 1 & 0
\end{pmatrix}
\]
\end{linenomath}
for some nonzero $g_1, g_2, g_3 \in \FF^+$,
where row 1 is indexed by $v_0$ and rows 2 and 3 are the oriented incidence matrix of $2C_3$ with its row corresponding to $v_3$ removed.
Consider the products $t_i \cdot e_j = (TA)_{ij}$.
The products $t_1 \cdot e_1 = 0$, $t_1 \cdot e_2=0$, $t_2 \cdot e_2=0$, and $t_3 \cdot e_3 = 0$ imply $T_{11}=0$, $T_{12}=0$, $T_{22}=0$, and $T_{33}=0$, respectively.
Thus
$t_2 \cdot e_1 = -(t_3 \cdot e_1)$ implies $T_{21}=-T_{31}$ and
$t_1 \cdot e_3 = -(t_2 \cdot e_3)$ implies $T_{13}=-T_{23}$;
$t_2 \cdot e_5 = 0$ implies $T_{21}=-bT_{23}$ and
$t_3 \cdot e_6 = 0$ implies $T_{31}=-cT_{32}$.
Now $t_2 \cdot e_4 = -(t_3 \cdot e_4)$ yields $T_{23}=-T_{32}$, which the preceding relations imply is equivalent to the statement $T_{31}/b=T_{31}/c$.
This holds if and only if either $T_{31}=0$ or $b=c$.
Hence if $b \neq c$, $T_{31}=0$.
Then the preceding relations imply that $T_{32}=0$.
Since $T_{33}=0$, this implies $T$ is singular, a contradiction.
Thus in the case that $b \neq c$, $A$ is not projectively equivalent to a canonical lift matrix particular to $(2C_3,\Bb)$.

So assume $T_{31}$ is nonzero and $b=c$.
Then the relations above imply $T$ is the matrix
\begin{linenomath}
\[ \begin{pmatrix}
	0 & 0 & t \\
	bt & 0 & -t  \\
	-bt & t & 0
\end{pmatrix}
= t
\begin{pmatrix}
	0 & 0 & 1 \\
	b & 0 & -1  \\
	-b & 1 & 0
\end{pmatrix}
\]
\end{linenomath}
for some nonzero $t \in \FF$.
Matrix $T$ is non-singular; taking $t=1$ we have
\begin{linenomath}
\[ TA =  \begin{pmatrix}
    0  & 0 & 1  & 1      & b & d & \\
    -b & 0 & -1 & b-1  & 0 & b-d \\
     b  & 1 & 0 & 1-b  & a-b & 0
\end{pmatrix} . \]
\end{linenomath}
By claims 1, 2, 3, and 7, none of $b$, $b-1$, $a-b$, nor $b-d$ is zero.
By scaling columns so that all nonzero entries in rows 2 and 3 are $\pm 1$ (and appending a fourth row obtained by negating the sum of rows 2 and 3 if desired), we obtain a canonical lift matrix particular to $(2C_3, \Bb)$.
Thus $A$ is projectively equivalent to a canonical lift matrix particular to $(2C_3, \Bb)$ if and only if $b = c$.
\end{proof}

Recall that $\Pr$ is the triangular prism with just its two triangles balanced, and that $\Pr_1$ and $\Pr_2$ are obtained from $\Pr$ by contracting 2 and 1 of the edges of the matching between the two triangles, respectively (Figure \ref{fig:T2PrimeSplit}).

\begin{lem} \label{L:SplitsOf2C3}
Let $\Omega \in \{\Pr, \Pr_1, \Pr_2\}$ and let $A$ be an $\FF$\-/matrix representing $M(\Omega)$.
Then $A$ is projectively equivalent to a canonical lift matrix particular to $\Omega$ or to a canonical frame matrix particular to $\Omega$, but not both.
\end{lem}

\begin{proof}
First, consider $\Pr_1$.
Let $A$ be an $\FF$\-/representation of $M(\Pr_1)$.
Let $Y$ be a $K_{1,3}$\-/subgraph of $\Pr_1$.
Then $\nabla_Y\Pr_1$ is obtained from $\mathsf T_2$ by the addition of an edge $e$ that creates a balanced 2-cycle.
Hence $e$ is parallel with an element of $M(\nabla_Y\Pr_1)$.
By Proposition \ref{prop:DeltaYoperation} $M(\nabla_Y\Pr_1) = \nabla_Y M(\Pr_1)$.
Thus by Lemma \ref{thm:All_reps_2C3s_are_canonical} every $\FF$\-/representation of
$\nabla_Y M(\Pr_1)$ is projectively equivalent to a canonical representation particular to $\nabla_Y \Pr_1$.
In particular, $\nabla_Y A$ is projectively equivalent to a canonical representation particular to $\nabla_Y \Pr_1$.
Thus by Proposition \ref{P:DeltasAndCanonical} $A$ is projectively equivalent to a canonical representation particular to $\Pr_1$.

Now consider $\Pr_2$.
Since $\nabla_Y\Pr_2$ is obtained from $\Pr_1$ by the addition of an edge that creates a
balanced 2-cycle, by the argument analogous to that of the previous paragraph every $\FF$\-/representation of $M(\Pr_2)$ is projectively equivalent to a canonical representation particular to $\Pr_2$.
Finally, the observation that $\nabla_Y\Pr$ is obtained from $\Pr_2$ by the addition of an edge that creates a balanced 2-cycle, along with the argument analogous to that of the previous paragraph, establishes the statement for $M(\Pr)$.

By Proposition \ref{P:DeltasAndCanonical} and Lemma \ref{thm:All_reps_2C3s_are_canonical}, in no case may $A$ be projectively equivalent to both a canonical lift and a canonical frame matrix.
\end{proof}

\begin{lem} \label{thm:All_reps_K4s_are_canonical}
Let $(K_4,\Bb)$ be a biased graph with no balanced 3-cycle and let $A$ be an $\FF$-matrix representing $M(K_4,\Bb)$.
Then $A$ is projectively equivalent to a canonical lift matrix particular to $(K_4,\mc B)$ or to a canonical frame matrix particular to $(K_4,\mc B)$, but not both.
\end{lem}

\begin{proof}
The biased graph $(K_4,\Bb)$ is isomorphic to $\mathsf D_{0,i}$ for some $i\in\{0,1,2,3\}$.
There is a $K_{1,3}$\-/subgraph $Y$ of $(K_4,\Bb)$ so that $\nabla_Y(K_4,\Bb) \cong \mathsf T_{i+1}$.
Since $\nabla_Y M(K_4,\Bb) \ab = \ab M(\nabla_Y(K_4,\Bb))$, the result follows by Lemma
\ref{thm:All_reps_2C3s_are_canonical} and Proposition \ref{P:DeltasAndCanonical}.
\end{proof}

\begin{lem} \label{thm:All_reps_Tube_are_canonical}
Let $(2C_4'',\Bb)$ be a biased graph with no balanced 2-cycle and let $A$ be an $\FF$-matrix representing $F(2C_4'',\Bb)$.
Then $A$ is projectively equivalent to a canonical frame matrix particular to $(2C_4'',\Bb)$.
\end{lem}

\begin{proof}
Without loss of generality we may assume that $2C_4''$ is labelled as shown in
Figure \ref{fig:TubeEdgeLabels}.
\begin{figure}[tbp] \begin{center}
\includegraphics[scale=0.9]{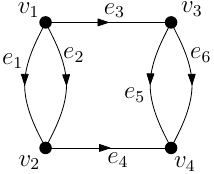}
 \caption{Labels and edge orientations for $2C_4''$.} \label{fig:TubeEdgeLabels}
\end{center} \end{figure}
There are three possibilities for $\Bb$: $|\Bb| \in \{0,1,2\}$.

Assume first that $|\Bb|$ is 0 or 1; \ie either $\Bb=\emptyset$ or, without loss of generality, $\Bb = \{e_1e_3e_4e_6\}$.
Then $A$ is projectively equivalent to the matrix
\begin{linenomath}
\[\label{eqn:A} A' = \bordermatrix{ & e_1 & e_2 & e_3 & e_4 & e_5 & e_6 \cr & 1 & 0 & 0 & 0 & 1 & 1 \cr & 0 & 1 & 0 & 0 & 1 & a \cr & 0 & 0 & 1 & 0 & 1 & b \cr & 0 & 0 & 0 & 1 & 1 & c \cr } \]\end{linenomath}
where $a,b,c \in \FF$ are distinct, neither of $b$ nor $c$ is 0, and none of $a$, $b$, or $c$ are 1;
in the case that $\Bb = \emptyset$, $a \neq 0$, while if $\Bb=\left\{e_1e_3e_4e_5\right\}$ then $a = 0$.
Let
\begin{linenomath}
\[T=\begin{pmatrix} b-a & 1-b & a-1 & 0
\\ a-c & c-1 & 0 & 1-a \\ 0 & 0 & 1-a & 0 \\ 0 & 0 & 0 & a-1 \end{pmatrix}.\]\end{linenomath}
Then $\det(T) = (a-1)^3(c-b)$ so $T$ is nonsingular, and
\begin{linenomath}
\[TA'=\left( \begin{array}{cccccc} b-a & 1-b & a-1 & 0 & 0 & 0 \\ a-c & c-1 & 0 & 1-a & 0 & 0 \\ 0 & 0 & 1-a & 0 & 1-a & (1-a) b \\ 0 & 0 & 0 & a-1 & a-1 & (a-1) c \\ \end{array} \right)\]\end{linenomath}
which has the desired canonical form after column scaling.

So assume $|\Bb|=2$. Without loss of generality, $\Bb \ab = \ab \left\{\ab e_1e_3e_4e_6, \ab e_2e_3e_4e_5\right\}$.
Then $A$ is projectively equivalent to
\begin{linenomath}
\[A' = \bordermatrix{ & e_1 & e_2 & e_3 & e_4 & e_5 & e_6 \cr & 1 & 0 &
0 & 0 & 0 & 1 \cr & 0 & 1 & 0 & 0 & 1 & 0 \cr & 0 & 0 & 1 & 0 & 1 & b \cr & 0 & 0 & 0 & 1 & 1 & c \cr }\]\end{linenomath}
where $b$ and $c$ are nonzero, distinct, and not equal to 1.
Let \begin{linenomath}\[ T =\begin{pmatrix} -b & -1 & 1 & 0 \\ c & 1 & 0 & -1 \\ 0 & 0
& -1 & 0 \\ 0 & 0 & 0 & 1 \end{pmatrix} \]\end{linenomath}
The determinant of $T$ is $b-c$, so $T$ is nonsingular.
Now \begin{linenomath}\[TA'=\left(
\begin{array}{cccccc} -b & -1 & 1 & 0 & 0 & 0 \\ c & 1 & 0 & -1 & 0 & 0 \\ 0 & 0 & -1 & 0 & -1 & -b \\ 0 & 0 & 0 & 1 & 1
& c \\ \end{array} \right)\]\end{linenomath} which has the desired canonical form after column scaling.
\end{proof}

\begin{lem} \label{thm:All_reps_Tube_Lift_are_canonical}
Let $(2C_4'',\Bb)$ be a biased graph with no balanced 2-cycle and let $A$ be an $\FF$-matrix representing $L(2C_4'',\Bb)$.
Then $A$ is projectively equivalent to a canonical lift matrix particular to $(2C_4'',\Bb)$.
\end{lem}

\begin{proof}
We may assume the edges of $2C_4''$ are labelled as in
Figure \ref{fig:TubeEdgeLabels}.
If $|\mc B|=2$, then there is a $\mathrm{GF}(2)^+$-gain function $\gamma$ realizing $(2C_4'',\Bb)$.
Thus $L(2C_4'',\Bb)$ is binary, represented by $A_L(2C_4'',\gamma)$, so $L(2C_4'',\Bb)$ has a projectively unique representation over every field, and the result follows.
So now assume that either $\Bb = \emptyset$ or, without loss of generality, $\Bb = \{e_1e_3e_4e_6\}$.
Since $\{e_1,e_2,e_5,e_6\}$ is a circuit, $A$ is projectively equivalent
to the matrix
\begin{linenomath}\[A' = \bordermatrix{ & e_1 & e_2 & e_3 & e_4 & e_5 & e_6 \cr & 1 & 0 & 0 & 0 & 1 & 1 \cr & 0 & 1 & 0 & 0 & 1 & a
\cr & 0 & 0 & 1 & 0 & 1 & b \cr & 0 & 0 & 0 & 1 & 1 & b \cr }\]\end{linenomath}
for some $a,b\in\FF$, where
$a$ and $b$ are distinct, neither $a$ nor $b$ is 1, $b \not= 0$, and $a = 0$ if and only if $|\Bb|=1$.
Let
\begin{linenomath}\[T=\left( \begin{array}{cccc} 0 & 1-b & 0 & 0 \\ b-a & 1-b & a-1 &
0 \\ a-b & b-1 & 0 & 1-a \\ 0 & 0 & 1-a & 0
\\ \end{array} \right).\]\end{linenomath}
Then $\det(T)=(a-1)^2(a - b)(b-1)\neq0$, so $T$ is nonsingular, and
\begin{linenomath}\[TA'=\left(
\begin{array}{cccccc} 0 & 1-b & 0 & 0 & 1-b & a-a b \\ b-a & 1-b & a-1 & 0 & 0 & 0 \\ a-b & b-1 & 0 & 1-a & 0 & 0 \\ 0 &
0 & 1-a & 0 & 1-a & b-a b \\ \end{array} \right).
\]
\end{linenomath}
After scaling columns appropriately (and appending a fifth row obtained by negating the sum of rows 2, 3, and 4, if desired) this is a canonical lift matrix particular to $(2C_4'', \Bb)$.
\end{proof}

For the almost-balanced case of Theorem \ref{mainthm2} we need the result analogous to the previous lemmas for one more biased graph.
Recall that we denote the graph obtained from $2C_3$ by deleting an edge by $2C_3 \bs e$.

\begin{lem} \label{P:ContractedTubeCanonical}
Let $(2C_3 \bs e,\Bb)$ be a biased graph with no balanced 2-cycle and let $A$ be an $\FF$-matrix representing $M(2C_3 \bs e, \Bb)$.
Then $A$ is projectively equivalent to a canonical lift matrix particular to $(2C_3\bs e,\mc B)$ and
$A$ is projectively equivalent to a canonical frame matrix particular to $(2C_3\bs e,\mc B)$ or to a roll-up of $(2C_3\bs e,\Bb)$.
\end{lem}

\begin{proof}
Assume $2C_3 \bs e$ is labelled as in Figure \ref{fig:2C3minus_e}.
\begin{figure}[tbp]
\begin{center}
\includegraphics[scale=0.9]{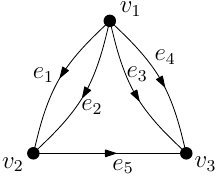}
\end{center}
\caption{$2C_3 \bs e$.}
\label{fig:2C3minus_e}
\end{figure}
Then $\{e_1, e_2, e_3\}$ is a basis, so we may assume the first three columns of $A$ are labelled $e_1$, $e_2$, $e_3$, and that these columns form an identity matrix.
Hence $A$ is projectively equivalent to the matrix
\begin{linenomath}\[A' = \bordermatrix{
	& e_1 & e_2 & e_3 & e_4 & e_5 \cr
	& 1 & 0 & 0 & 1 & 1 \cr
	& 0 & 1 & 0 & 1 & a \cr
	& 0 & 0 & 1 & 1 & b } \]\end{linenomath}
for some $a, b \in \FF$.
Since $\{e_1, e_2, e_5\}$ is not a circuit, $b \neq 0$, and since $\{e_3, e_4, e_5\}$ is not a circuit, $a \not= 1$.
Choose an element $t \neq -1$ and let
\begin{linenomath}
\[ T = \begin{pmatrix}
	t & 1 & -(a+t)/b \\
	1 & -1 & 0 \\
	0 & 0 & (a-1)/b
	\end{pmatrix} .
\]
\end{linenomath}
The determinant of $T$ is $(1-a)(t+1)/b$, so $T$ is nonsingular, and
\begin{linenomath}
\[ TA' =
\begin{pmatrix}
	t & 1 & -(a+t)/b & (b-a+t(b-1))/b & 0 \\
	1 & -1 & 0 & 0 & 1-a \\
	0 & 0 & (a-1)/b & (a-1)/b & a-1
	\end{pmatrix} .
\]
\end{linenomath}
After scaling columns so that every entry in rows 2 and 3 is $\pm 1$, this is a canonical lift matrix particular to $(2C_3 \bs e,\Bb)$, where the first row is the ``gains row'' indexed by $v_0$ and rows 2 and 3 are indexed by vertices $v_2$ and $v_3$, respectively.
Now re-index the first row as $v_1$.
Taking $t=0$ yields a canonical frame matrix particular to a roll-up of $(2C_3 \bs e,\Bb)$.
If $\FF - \{0,-1,-a,(b-a)/(1-b)\}$ is nonempty, then choosing an element from this set for $t$ yields a canonical frame matrix particular to $(2C_3 \bs e, \Bb)$.
\end{proof}

\section{Matrix representations correspond to gain graphs} \label{sec:MainThmProofs}

\subsection{Projectively equivalent canonical representations arise from switching equivalent gain graphs}

A biased graph representing a 3-connected matroid is 2-connected and has no 2-separation with one side inducing a balanced subgraph.
Thus Theorem \ref{mainthm1a} follows immediately from Theorems \ref{T:ProjectiveIsSwitching} and \ref{thm:proj_equiv_iff_switching_equiv_bal_vertex} below.

\subsubsection{Properly unbalanced representations}

%
%
%
\begin{thm}
\label{T:ProjectiveIsSwitching}
Let $(G,\mathcal B)$ be a loopless, 2-connected, properly unbalanced biased graph.
Let $\FF$ be a field.

\textup{(i)}
The canonical frame matrices given by two $\FF^\times$-gain functions $\vp$ and $\psi$ realizing $\GB$ are projectively equivalent if and only if $\vp$ and $\psi$ are switching equivalent.

\textup{(ii)}
The canonical lift matrices given by two $\FF^+$-gain functions $\vp$ and $\psi$ realizing $\GB$ are projectively equivalent if and only if $\vp$ and $\psi$ are switching-and-scaling equivalent.

\textup{(iii)}
Let $\vp$ be an $\FF^\times$\-/realization of $(G,\mathcal B)$ and let $\psi$ be an $\FF^+$\-/realization of $(G,\mathcal B)$.
Then $A_F(G,\vp)$ and $A_L(G,\psi)$ are not projectively equivalent.
\end{thm}

\begin{proof}
For statement (i) (respectively, statement (ii)), if $\vp$ and $\psi$ are switching (resp., switching-and-scaling) equivalent then
$A_F(G,\vp)$ and $A_F(G,\psi)$ (resp., $A_L(G,\vp)$ and $A_L(G,\psi)$) are projectively equivalent by Proposition \ref{P:switch_equiv_implies_proj_equiv}.

For the converse of statement (i) (resp., statement (ii)) assume that $\vp$ and $\psi$ are $\FF^\times$\-/realizations (resp., $\FF^+$\-/realizations) that are not switching (resp., switching-and-scaling) equivalent.
By Theorem \ref{MT:InequivalenceLocalized}, there is a minor $(H,\mathcal S)$ of $(G,\mathcal B)$ such that either $(H,\mathcal S) \in \Gg_0$
with $\vp|_H$ and $\psi|_H$ switching inequivalent
or $(H,\mathcal S)\cong\mathsf U_2$
(resp., $(H, \Ss) \cong \mathsf U_3$)
with $\vp|_H$ and $\psi|_H$ switching
(resp., switching-and-scaling)
inequivalent on the 2-cycle of $\mathsf U_2$
(resp., on the theta subgraph of $\mathsf U_3$).
By Lemmas
\ref{lem:2C3_equil_implies_switching_equiv},
\ref{lem:nobaltriangle},
\ref{lem:4tube_equil_implies_switching_equiv}, and
\ref{lem:U2Inequivalence},
$A_F(H,\vp|_H)$ and $A_F(H,\psi|_H)$ are not projectively equivalent
(resp., by Lemmas
\ref{lem:2C3_equil_implies_switching_equiv_lift_version},
\ref{lem:add_gains_K4_switch_eq_iff_equal},
\ref{lem:4tube_equil_implies_switching_equiv_LIFT}, and
\ref{lem:T3InequivalenceLift},
$A_L(H,\vp|_H)$ and $A_L(H,\psi|_H)$ are not projectively equivalent).
Thus $A_F(G,\vp)$ and $A_F(G,\psi)$ (resp., $A_L(G,\vp)$ and $A_L(G,\psi)$) are not projectively equivalent.

For statement (iii), first observe that if $\GB$ is not tangled, then $F\GB \neq L\GB$ so $A_F(G,\vp)$ and $A_L(G,\psi)$ are certainly not projectively equivalent.
So assume $\GB$ is tangled.
Then by
Corollary \ref{cor:ofInequivalenceLocalized}
$\GB$ has a minor $(H,\Ss) \in \Gg_0$ where $H$ is either $2C_3$ or $K_4$.
Thus by Lemmas
\ref{lem:C_3_frame_lift_not_ech_equiv} and \ref{lem:K4_frame_lift_not_ech_equiv} $A_F(G,\vp)$ and $A_L(G,\psi)$ are not projectively equivalent.
\end{proof}

\subsubsection{Almost balanced representations}

\paragraph{Derived gain functions.}

Recall that when $\GB$ has no two vertex-disjoint unbalanced cycles, $F\GB=L\GB$, and we write $M\GB$ to denote this matroid.
Recall also that for an $\FF^+$-gain function $\vp$, in the case that $G$ is connected, the matrix $A_L(G,\vp)$ is projectively equivalent to the full-rank matrix $A_L^{-v}(G,\vp)$, for any vertex $v \in V(G)$
(as described in Section \ref{sec:on_canonical_lift_representations}).
In particular, if $(G,\Bb_\vp)$ is connected and has a balancing vertex $u$ after deleting all joints, then $A_L(G,\vp)$ is projectively equivalent to $A_L^{-u}(G,\vp)$.
Recall also that for every almost-balanced biased graph $\GB$, there is a family of biased graphs $\Rr_{\GB}$, each member of which represents $M\GB$ as a frame matroid, and there is a uniquely chosen member $\GBu$ of $\Rr_{\GB}$ such that all other members of $\Rr_{\GB}$ are obtained from $\GBu$ as roll-ups (Section \ref{sec:rollups}).

Let $\GB$ be a biased graph with a joint.
No $\mathrm{GF}(2)^\times$-gain function can realize $\GB$, for the trivial reason that $\mathrm{GF}(2)^\times$ has no non-identity element.
Aside from those over $\mathrm{GF}(2)$,
there is a very close relationship between gain functions from the additive and multiplicative groups of a field realizing almost-balanced biased graphs.
Let $\GB$ be a connected almost-balanced biased graph, with balancing vertex $u$ after deleting its set $J$ of joints.
Every gain function realizing $\GB$ is switching equivalent to a gain function assigning the group identity element to each link not incident $u$, obtained by normalizing on a spanning tree of $G-u$.
So let $T$ be a spanning tree of $G-u$, and let $\vp$ be a $T$-normalized $\FF^\times$-gain function realizing $\GB$.
Then there is a $T$-normalized $\FF^+$-gain function $\vp^+$ realizing $\GBu$ obtained from $\vp$ up to loops by simply replacing the multiplicative identity with the additive identity.
That is, set $\vp^+(e)=0$ if $e$ is a link not incident to $u$ and $\vp^+(e)=\vp(e)$ if $e$ is a link incident to $u$.
To complete the definition of $\vp^+$, simply set $\vp^+(e)=0$ if $e$ is a joint not incident to $u$ or a balanced loop, and $\vp^+(e)=-1$ if $e$ is a joint incident to $u$.
Call $\vp^+$ the $\FF^+$-gain function \emph{derived from} $\vp$.

For each unbalancing class $U$ of $\Sigma(u)$, we denote by $(G_U,\Bb_U)$ the roll-up of $\GBu$ in which $U$ is a set of joints.
Let $T_U$ be a spanning tree of $\widehat{G}$ containing an edge in $U$, and let $T$ be the spanning tree of $G-u$ obtained from $T_U$ by deleting its edge incident to $u$.
Let $\psi$ be a $T_U$-normalized $\FF^+$-gain function realizing $\GBu$.
Assume $\FF$ is not $\mathrm{GF}(2)$, and choose an element $a \in \FF^\times$ that is not 1.
Then there is a $T$-normalized $\FF^\times$-gain function $\psi^\times$ realizing $(G_U,\Bb_U)$ obtained from $\psi$ up to loops by simply replacing the additive identity with the multiplicative identity.
That is, set $\psi^\times(e)=1$ if $e$ is a link not incident to $u$ and $\psi^\times(e)=\psi(e)$ if $e$ is a link incident to $u$.
Every link $e$ incident to $u$ in $G_U$ satisfies $\psi(e) \neq 0$, and so satisfies $\psi(e) \in \FF^\times$.
Complete the definition of $\psi^\times$ by simply setting $\psi^\times(e)=1$ if $e$ is a balanced loop and $\psi^\times(e) = a$ if $e$ is a joint.
Call $\psi^\times$ the $\FF^\times$-gain function \emph{derived from} $\psi$.

\begin{lem} \label{lem:bal_vertex_frame_and_lift_proj_equiv}
Let $\GB$ be a 2-connected almost-balanced biased graph and let $\FF$ be a field other than $\mathrm{GF}(2)$.
Let $J$ be the set of joints of $\GB$. Assume $\GB \bs J$ has a unique balancing vertex $u$, and let $T$ be a spanning tree of $G-u$.

\textup{(i)} Let $\vp$ be a $T$-normalized $\FF^\times$-gain function realizing a roll-up $(G_U,\Bb_U)$ of $\GBu$.
Then $A_L^{-u}(\widehat{G},\vp^+)$ is obtained from $A_F(G_U,\vp)$ by scaling columns.

\textup{(ii)} Let $\psi$ be a $T$-normalized $\FF^+$-gain function realizing $\GBu$.
Let $U$ be the (possibly empty) unbalancing class of $\Sigma(u)$ for which $\psi(U)=\{0\}$.
Then $A_F(G_U,\psi^\times)$ is obtained from $A_L^{-u}(\widehat{G},\psi)$ by scaling columns.
\end{lem}

\begin{proof}
(i)
Since $\vp$ is $T$-normalized, $\vp(e) = 1$ for each link $e$ that is not incident to $u$.
We may assume that all edges incident to $u$ are directed into $u$.
Let $A$ be the matrix obtained from $A_F(G_U,\vp)$ by re-indexing row $u$ as  $v_0$ (the ``gains" row) and scaling columns so that
each column that is nonzero in just one row has that nonzero entry equal to $-1$, and columns with a nonzero entry in row $v_0$ have either all other entries 0 or have other nonzero entry equal to $-1$.
Because $\vp$ assigns $1$ to every link not incident to $u$,
$A$ is a full-rank canonical lift matrix (with row $u$ removed) particular to $\GBu$, where each edge in $\Sigma(u)$ is a link directed out from $u$.
Moreover, for each edge $e \in E(G)$, the entry in row $v_0$, column $e$ of $A$ is equal to $\vp^+(e)$, so $A = A_L^{-u}(\widehat{G},\vp^+)$.

(ii)
Since $\psi$ is $T$-normalized, $\psi(e)=0$ for each edge $e$ that is not incident to $u$.
We may assume that all links incident to $u$ in $G$ are directed out from $u$.
Since the rows of $A_L^{-u}(\widehat{G},\psi)$ are indexed by $v_0$ (the ``gains row") and $V(\widehat{G})-u$, every column of $A_L^{-u}(\widehat{G},\psi)$ has at most two nonzero entries.
Let $A$ be the matrix obtained from $A_L^{-u}(\widehat{G},\psi)$ by re-indexing row $v_0$ as row $u$ and scaling columns so that
each column that has exactly two nonzero entries with a nonzero entry in row $u$ has its other nonzero entry equal to $1$,
and every column with just one nonzero entry has its nonzero entry equal to $1-a$, where $a$ is the chosen element of $\FF^\times$ different than 1 that $\psi^\times$ assigns to joints.
Since $\GBu$ has $u$ as a balancing vertex, each column $e$ of $A$ with 0 in row $u$ has either exactly two nonzero entries, which are $1$ and $-1$, and these appear in rows $v, w$ when $e$ has endpoints $v,w$ and neither $v$ nor $w$ is equal to $u$, or column $e$ at most one nonzero entry, which is $1-a$, appearing in row $v$ when $e \in U$ with endpoints $u,v$ in $\widehat{G}$.
Thus $A$ is a canonical frame matrix particular to $(G_U,\Bb_U)$ where each edge in $\Sigma(u)-U$ is a link directed into $u$ and each element in $U$ is a joint.
Moreover, by definition the gain function $\psi^\times$ realizes $(G_U,\Bb_U)$ and $A = A_F(G_U,\psi^\times)$.
\end{proof}

Let $\GB$ be a connected almost-balanced biased graph with a unique balancing vertex $u$ after deleting its joints.
Since we may always assume that all links incident to $u$ are directed either into or out from $u$, and every gain function realizing $\GB$ is switching equivalent to a gain function
assigning the identity element to each link not incident to $u$,
we may define a \emph{derived gain function} \emph{from} any $\FF^\times$- or $\FF^+$-gain function, by first switching appropriately.
Further, for each $\FF^+$-gain function realizing $\GBu$ and each unbalancing class $U \in \Sigma(u)$, we may always switch to obtain an $\FF^+$-gain function $\vp$ realizing $\GBu$ with the property that for each edge $e \in U$, $\vp(e)=0$.
Thus, with this extension of the notion of a derived gain function,  Proposition \ref{P:switch_equiv_implies_proj_equiv} and Lemma \ref{lem:bal_vertex_frame_and_lift_proj_equiv} immediately yield the following.

\begin{cor} \label{cor:bal_vertex_frame_and_lift_proj_equiv}
Let $\GB$ be a 2-connected almost-balanced biased graph with a unique balancing vertex $u$ after deleting its set of joints, and let $\FF$ be a field other than $\mathrm{GF}(2)$.

\textup{(i)} For every $\FF^+$-gain function $\vp$ realizing $\GBu$, and every unbalancing class $U \subseteq \Sigma(u)$, there is a derived $\FF^\times$-gain function $\vp^\times$ realizing the roll-up $(G_U,\Bb_U)$ of $\GBu$ such that $A_L(\widehat{G},\vp)$ and $A_F(G_U,\vp^\times)$ are projectively equivalent.

\textup{(ii)} For every $\FF^\times$-gain function $\psi$ realizing a roll-up $(G_U,\Bb_U)$ of $\GBu$, there is a derived $\FF^+$-gain function $\psi^+$ realizing $\GBu$ for which $A_F(G_U,\psi)$ and $A_L(\widehat{G},\psi^+)$ are projectively equivalent.

\textup{(iii)} There is an $\FF^\times$-gain function realizing $\GBu$ if and only if there is an $\FF^+$-gain function $\vp$ realizing $\GBu$ for which $\vp(e) \neq 0$ for each edge in $\Sigma(u)$.
\end{cor}

Equipped with the above tool, we can now state and prove a result analogous to Theorem \ref{T:ProjectiveIsSwitching} for almost-balanced biased graphs.

\begin{thm} \label{thm:proj_equiv_iff_switching_equiv_bal_vertex}
Let $\GB$ be a 2-connected, almost-balanced biased graph with a unique balancing vertex $u$ after deleting its joints,
with no joint incident to $u$,
and with no vertical 2-separation with one side balanced.
Let $\FF$ be a field.

\textup{(i)} The canonical lift matrices given by two $\FF^+$-gain functions $\vp$ and $\psi$ realizing $\GBu$ are projectively equivalent if and only if $\vp$ and $\psi$ are switching-and-scaling equivalent.

\textup{(ii)} Let $U$ and $W$ be unbalancing classes of $\Sigma(u)$.
Let $\vp$ and $\psi$ be $\FF^\times$-gain functions realizing $(G_U,\Bb_U)$ and $(G_W,\Bb_W)$, respectively.
The canonical frame matrices given by $\vp$ and $\psi$ are projectively equivalent if and only if their derived gain functions $\vp^+$ and $\psi^+$ are switching-and-scaling equivalent.

\textup{(iii)} Let $U$ be an unbalancing class of $\Sigma(u)$.
The canonical lift matrix given by an $\FF^+$-gain function $\vp$ realizing $\GBu$ and the canonical frame matrix given by an $\FF^\times$-gain function $\psi$ realizing $(G_U,\Bb_U)$ are projectively equivalent if and only if $\vp$ and the derived gain function $\psi^+$ are switching-and-scaling equivalent.

\textup{(iv)} Let $\vp$ and $\psi$ be $\FF^+$- and $\FF^\times$-gain functions, respectively, realizing $\GBu$.
The canonical lift matrix given by $\vp$ and the canonical frame matrix given by $\psi$ are projectively equivalent if and only if $\vp$ and the derived gain function $\psi^+$ are switching-and-scaling equivalent.
\end{thm}

\begin{proof}
(i)
Let $\vp$ and $\psi$ be $\FF^+$-gain functions realizing $\GBu$.
If $\vp$ and $\psi$ are switching-and-scaling equivalent then  by Proposition \ref{P:switch_equiv_implies_proj_equiv} $A_L(\widehat{G},\vp)$ and $A_L(\widehat{G},\psi)$ are projectively equivalent.

Conversely, suppose for a contradiction that $A_L(\widehat{G},\vp)$ and $A_L(\widehat{G},\psi)$
are projectively equivalent
but that $\vp$ and $\psi$ are not switching-and-scaling equivalent.
We may assume that all edges incident to $u$
have $u$ as their tail.
By switching we may assume that $\vp(e) = \psi(e) = 0$ for each edge $e$ not incident to $u$.
Consider the full-rank canonical lift matrices $A_L^{-u}(\widehat{G},\vp)$ and $A_L^{-u}(\widehat{G},\psi)$ with rows indexed by $v_0 \cup \br{V(G) - u}$, where $v_0 \notin V(G)$ is the ``gains row''.
Put $A=A_L^{-u}(\widehat{G},\vp)$ and $B=A_L^{-u}(\widehat{G},\psi)$, and suppose $T$ is a nonsingular matrix such that $TAS=B$, where $S$ is a nonsingular diagonal column-scaling matrix.
Let the rows and columns of $T$ be indexed by $v_0 \cup (V(G)-u)$ according to the rows of $A$.
By statement (ii) of Lemma \ref{P:switch_equiv_iff_diagonal} either $T$ has an entry on its main diagonal that is not 1 or $T$ has a nonzero entry off its main diagonal, in either case in a row other than $v_0$.
Suppose first $T$ has all entries off its main diagonal equal to 0, aside from those in row $v_0$.
Let $U$ be the $|V(G)| \times |V(G)|$ diagonal matrix with rows and columns indexed by $v_0 \cup V(G)-u$ in which entry $U_{v_0 v_0}=1$ and entry $U_{vv} = a\inv$ if entry $T_{vv} = a$.
Since $T$ is non-singular no such entry is 0.
Removing the row and column of $UT$ indexed by $v_0$ leaves an identity matrix and $(UT)AR=B$, where $R$ is an appropriate diagonal matrix scaling the columns of $(UT)A$.
Thus $\vp$ and $\psi$ are switching-and-scaling equivalent by statement (ii) of Lemma \ref{P:switch_equiv_iff_diagonal}, contrary to assumption.

So assume there is a nonzero element off the main diagonal of $T$ in a row other than $v_0$.
Suppose the entry in row $x$, column $y$ is nonzero, where $x \neq y$ and $x \neq v_0$.
Since $G$ is 2-connected and has no vertical 2-separation with one side balanced, there is an unbalanced cycle $C$ avoiding $x$ while containing $y$.
Let $f, f'$ be the edges of $C$ incident to $u$.
Then $\vp(f) \neq \vp(f')$.
Denote by $T_x$ row $x$ of $T$ and by $A_e$ column $e$ of $A$.
Consider the equations $T_x \cdot A_e = B_{xe}$ given by each of the dot products of row $x$ of $T$ with column $e$ of $A$, for each $e \in E(C)$.
Since $C$ avoids $x$, for each edge $e \in E(C)$ entry $B_{xe}$ is zero.
There are precisely two nonzero entries in each column $e$ of $A$ with $e \in E(C)$, and other than columns $f$ and $f'$ one of these two entries is 1 and the other is $-1$.
Thus the system of equations $\{T_x \cdot A_e = 0 : e \in E(C)\}$ imply $\vp(f) = \vp(f')$, a contradiction.

(ii)
If $\vp^+$ and $\psi^+$ are switching-and-scaling equivalent, then by Proposition \ref{P:switch_equiv_implies_proj_equiv}, $A_L(\widehat{G},\vp^+)$ and $A_L(\widehat{G},\psi^+)$ are projectively equivalent.
Hence by Corollary \ref{cor:bal_vertex_frame_and_lift_proj_equiv}, $A_F(G_U,\vp)$ and $A_F(G_W,\psi)$ are projectively equivalent.
Conversely, suppose $A_F(G_U,\vp)$ and $A_F(G_W,\psi)$ are projectively equivalent.
Then by Corollary \ref{cor:bal_vertex_frame_and_lift_proj_equiv}, $A_L(\widehat{G},\vp^+)$ and $A_L(\widehat{G},\psi^+)$ are projectively equivalent, and so by statement (i), $\vp^+$ and $\psi^+$ are switching-and-scaling equivalent.

The proofs of statements (iii) and (iv) are straightforward modifications of the proof of (ii).
\end{proof}

\subsection{Matrix representations arise from biased graph representations}

Theorem \ref{mainthm2} is an immediate consequence of Theorem \ref{thm:all_reps_eqiv_to_cannonical} below (together with a straightforward check for the case of rank 2).

\begin{thm} \label{thm:all_reps_eqiv_to_cannonical}
Let $M$ be a 3-connected matroid of rank greater than two, and let $\FF$ be a field.
Let $A$ be a matrix over $\FF$ representing $M$ and let $\GB$ be a biased graph representing $M$.
If $\GB$ is properly unbalanced then exactly one of the following holds.
\begin{enumerate}[label=\textup{(\roman*)}]
\item $A$ is projectively equivalent to a canonical lift matrix particular to $\GB$, or
\item $A$ is projectively equivalent to a canonical frame matrix particular to $\GB$.
\end{enumerate}
If $\GB$ is almost-balanced then each of the following hold, unless $\FF$ is $\mathrm{GF}(2)$, in which case  \emph{(i)} holds.
\begin{enumerate}[label=\textup{(\roman*)}]
\item $A$ is projectively equivalent to a canonical lift matrix particular to $\GBu$, and
\item $A$ is projectively equivalent to a canonical frame matrix particular to each roll-up of $\GBu$.
\end{enumerate}
\end{thm}


Theorem \ref{thm:all_reps_eqiv_to_cannonical} follows immediately from Theorems \ref{T:MainTheorem1} and \ref{thm:MainThm_bal_vertex_case} below.

\begin{thm} \label{T:MainTheorem1}
Let $M$ be a matroid represented by a 2-connected, properly unbalanced biased graph $(G,\Bb)$.
Let $\FF$ be a field and let $A$ be a matrix over $\FF$ representing $M$.
Exactly one of the following holds:
\begin{enumerate}[label=\textup{(\roman*)}]
\item $A$ is projectively equivalent to a canonical lift matrix particular to $\GB$, or
\item $A$ is projectively equivalent to a canonical frame matrix particular to $\GB$.
\end{enumerate}
\end{thm}

We will require the following fact on several occasions.
It follows immediately from the fact that a pair of edges incident to a vertex of degree two form a series pair in the matroid.

\begin{lem} \label{P:Subdivisions}
Let $(H,\mathcal S)$ be a subdivision of $(G,\mathcal B)$. Let $\FF$ be a field and let $\Gamma \in \{\FF^\times, \FF^+\}$.

\textup{(i)} The $\FF$\-/matrix representations of $F(H,\mathcal S)$ are in one-to-one correspondence with the $\FF$\-/matrix representations of $F(G,\mathcal B)$ up to projective equivalence.

\textup{(ii)} The $\FF$\-/matrix representations of $L(H,\mathcal S)$ are in one-to-one correspondence with the $\FF$\-/matrix representations of $L(G,\mathcal B)$ up to projective equivalence.

\textup{(iii)} The $\Gamma$\-/realizations of $(H,\mc S)$ are in one-to-one correspondence with the $\Gamma$\-/realizations of $(G,\mc B)$ up to switching (resp., switching-and-scaling).
\end{lem}

We also need the following more technical fact to prove Theorem \ref{T:MainTheorem1}.

\begin{lem} \label{L:TangledDoesntExtend}
Let $\GB$ be a connected biased graph with a joint $e$ such that $\GB \bs e$ is a biased $2C_3$ with no balanced 2-cycle or a biased $K_4$ with no balanced triangle.
Let $\FF$ be a field, and let $\vp$ be an $\FF^\times$- or $\FF^+$-gain function on $G$.

\textup{(i)} If $A_F(G \bs e, \vp)$ represents $M(\GB \bs e)$ then $A_F(G\bs e,\vp)$ does not extend to an $\bb F$\-/representation of $L\GB$.

\textup{(ii)} If $A_L(G \bs e,\vp)$ represents $M(\GB \bs e)$ then $A_L(G \bs e, \vp)$ does not extend to an $\bb F$\-/representation of $F\GB$.
\end{lem}

\begin{proof}
We give a detailed proof for the case in which $\GB \bs e$ is a biased $2C_3$.
The case for which $\GB \bs e$ is a biased $K_4$ follows from $\Delta$-$Y$ and $Y$-$\Delta$ exchanges, by Propositions \ref{prop:DeltaYoperation} and \ref{P:DeltasAndCanonical}.

(i) Suppose for a contradiction that there is a matrix $A$ over $\FF$ representing $L\GB$ such that removing column $e$ from $A$ yields the matrix $A_F(G \bs e, \vp)$.
We may assume that $A$ has full rank, and so has three rows.
Since $e$ is not a loop of $L\GB$, column $e$ of $A$ is nonzero.
If column $e$ has just one nonzero entry, then $A$ is an $\bb F$\-/representation of a matroid $F(\Omega)$ where $\Omega$ is obtained from $\GB \bs e$ by adding a joint to a vertex.
But comparing circuits we see that $F(\Omega) \neq L\GB$, a contradiction.
If column $e$ has exactly two nonzero entries, then $A$ is an $\bb F$\-/representation of a matroid $F(\Omega)$, where $\Omega$ is obtained from $\GB \bs e$ by adding a link.
But again comparing circuits we see that $F(\Omega) \neq L\GB$, a contradiction.
So finally suppose column $e$ has three nonzero entries.
Let $X$ be an unbalanced 2-cycle of $\Om$.
Then $X \cup e$ is a circuit of $L\GB$ but the columns of $A$ corresponding to $X \cup e$ are linearly independent, a contradiction.

(ii) Suppose for a contradiction that there is a matrix $A$ representing $F\GB$ such that removing column $e$ from $A$ yields $A_L(G \bs e, \vp)$.
Let $V(G) = \{v_1, v_2, v_3\}$ where $e$ is incident to $v_1$.
Then $A \bs e = A_L(G \bs e, \vp)$ has rows indexed by $v_0 \cup V(G)$ where $v_0 \notin V(G)$ is the ``gains row" and removing row $v_0$ from $A_L(G \bs e, \vp)$ leaves the oriented incidence matrix of $G \bs e$.
Since $F\GB$ has rank three while $A$ has four rows, $A$ is not of full rank.
Since in $A \bs e$ row $v_0$ is not in the span of $\{v_1, v_2, v_3\}$, neither is row $v_0$ in the span of rows $\{v_1, v_2, v_3\}$ in $A$.
Thus in $A$ row $v_3$ is in the span of rows $v_1$ and $v_2$.
Since any linear combination of rows $v_1$, $v_2$, and $v_3$ in $A$ yields a corresponding linear combination in $A \bs e$, this implies that the entries of $A$ in rows $v_1$, $v_2$, and $v_3$ of column $e$ sum to zero.
Put $A_{v_1 e} = a$, $A_{v_2 e} = b$; then $A_{v_3 e} = -(a+b)$.

First suppose that $a=b=0$.
As $e$ is not a loop of $F\GB$, column $e$ is nonzero. Thus $A_{v_0 e} \neq 0$.
Let $X$ be the unbalanced 2-cycle consisting of the pair of edges linking $v_2$ and $v_3$.
Then $X \cup e$ is independent in $F\GB$ but the columns of $A$ representing $X \cup e$ are linearly dependent, a contradiction.
Next suppose $a=-b$.
Let $X$ be the unbalanced 2-cycle consisting of the edges linking $v_1$ and $v_3$.
The set $X \cup e$ is a circuit in $F\GB$ but has columns linearly independent in $A$, a contradiction.
Finally, suppose none of $a$, $b$, nor $-(a+b)$ are zero.
Let $X$ be the unbalanced 2-cycle consisting of the edges linking $v_1$ and $v_2$.
Then $X \cup e$ is a circuit of $F\GB$ but has columns linearly independent in $A$, a contradiction.
\end{proof}

\begin{proof}[Proof of Theorem \ref{T:MainTheorem1}]
We show that if $M=F\GB$ then there is an $\FF^\times$-gain function $\vp$ such that $A$ is projectively equivalent to $A_F(G,\vp)$,
that if $M=L\GB$ then there is an $\FF^+$-gain function $\psi$ such that $A$ is projectively equivalent to $A_L(G,\psi)$, and
that if $\GB$ is tangled, so $F\GB=L\GB$, then $A$ is
not projectively equivalent to both $A_F(G,\vp)$ and $A_L(G,\psi)$.

By Theorem \ref{MT:Unavoidsubdivisions}, $(G,\Bb)$ contains a subgraph $\Om_0$ that is a subdivision of a biased graph in $\Tt_0$.
Let $(G_0,\Bb_0)$ be the biased subgraph of $\GB$ induced by $E(\Om_0)$, with $V(G_0)=V(G)$.
Let $A_0$ be the submatrix of $A$ consisting of the columns whose elements are in $E(\Om_0)$.
By Lemmas
\ref{thm:All_reps_2C3s_are_canonical}, \ref{L:SplitsOf2C3}, \ref{thm:All_reps_K4s_are_canonical},
\ref{thm:All_reps_Tube_are_canonical}, or \ref{thm:All_reps_Tube_Lift_are_canonical}, and
Lemma \ref{P:Subdivisions},
there exists an $\FF^\times$-gain function $\vp_0$ such that $A_0$ is projectively equivalent to $A_F(G_0,\varphi_0)$,
or there exists an $\FF^+$-gain function $\psi_0$ such that $A_0$ is projectively equivalent to $A_L(G_0, \psi_0)$, but not both.

Let $J$ be the set of joints of $\GB$.
Since both $\Om_0$ and $G$ are 2-connected, there is a sequence of 2-connected biased subgraphs
$\Om_0 \subset \Om_1 \subset \cdots \subset \Om_n$
where $\Om_n = \GB \bs J$ such that for each $i \in \{0,\ldots,n-1\}$ there is a path $P_i$ in $G$ internally disjoint from $\Om_i$
so that $\Om_i \cup P_i = \Om_{i+1}$.
For each $i \in \{0,\ldots,n\}$, let $(G_i,\Bb_i)$ be the biased subgraph of $\GB$ induced by $E(\Om_i)$ with $V(G_i)=V(G)$.
Let $A_i$ be the submatrix of $A$ consisting of all rows of $A$ and precisely those columns representing $E(\Om_i)$.
Thus for each $i$, in the case that $M=F\GB$, $A_i$ represents $F(\Om_i)$; in the case that $M=L\GB$, $A_i$ represents $L(\Om_i)$.
Inductively assume that for some $i \geq 0$ there exist nonsingular matrices $T_0, \ldots, T_i$, nonsingular diagonal matrices $S_0, \ldots, S_i$, and gain functions $\vp_0, \ldots, \vp_i$, such that either
\begin{enumerate}
\item[(1)] $T_j A_j S_j = A_F(G_j,\varphi_j)$ for each $j \in \{0, \ldots, i\}$, or
\item[(2)] $T_j A_j S_j = A_L(G_j,\psi_j)$ for each $j \in \{0, \ldots, i\}$.
\end{enumerate}
We will show the same projective equivalence for $A_{i+1}$.
We first obtain this conclusion in the case that $P_i$ consists of a single edge.
Then if $P_i$ has length greater than one, the conclusion follows from Lemma \ref{P:Subdivisions}.
So suppose $P_i$ consists of a single edge $e_i$ linking vertices $u_i, v_i \in V(\Om_i)$.
We consider cases (1) and (2) above separately.

(1)
Consider the matrix $T_i A_{i+1}$.
Matrix $A_{i+1}$ has rows indexed by $V(G)$, according to the indexing of the corresponding rows of $T_i A_i S_i$, and columns indexed by $E(G_{i+1})$.
We first show that column $e_i$ of $T_iA_{i+1}$ is zero in every row aside from $u_i$ and $v_i$.
Suppose for a contradiction that $e_i$ is nonzero in a row $x$ of $T_i A_{i+1}$ differing from $u_i$ and $v_i$.
Since $\Om_0$ does not have a balancing vertex, neither does $\Om_i$.
Thus $\Om_i-x$ is unbalanced and connected.
Hence there is a subset $U_i \subseteq E(\Om_i-x)$ that induces a subgraph of $\Om_i$ that is a spanning tree of $\Om_i-x$ along with one additional edge whose fundamental cycle with respect to this tree is unbalanced.
Contained in $U_i \cup e_i$ is a biased subgraph $C$ whose edge set is a circuit of $F(\Om_{i+1})$.
The subgraph $C$ contains $e_i$ and does not contain $x$.
Since $T_i A_i S_i = A_F(G_i,\vp_i)$, the columns of $T_i A_{i+1}$ representing $E(C)-e_i$ are all zero in row $x$.
Hence while $C$ is a circuit of $F(\Om_{i+1})$ the columns of $T_i A_{i+1}$ representing $C$ are linearly independent, a contradiction.

We now show that both rows $u_i$ and $v_i$ in column $e_i$ of $T_i A_{i+1}$ are nonzero.
Since $e_i$ is not a loop of $M$, at least one entry of column $e$ is nonzero; without loss of generality assume its entry in row $u_i$ is not zero.
For a contradiction, suppose its entry in row $v_i$ is zero.
Let $Q$ be a subgraph of $\Om_i-v_i$
consisting of an unbalanced cycle and a path (possibly trivial) connecting this cycle to $u_i$.
In $F(\Om_{i+1})$, $E(Q) \cup e_i$ is independent, but the columns of $T_i A_{i+1}$ representing $E(Q) \cup e_i$ are linearly dependent, a contradiction.
Thus column $e_i$ of $T_i A_{i+1}$ is nonzero in precisely its rows $u_i, v_i$ corresponding to the endpoints of edge $e_i$ in $\Om_i$.
Let $T_{i+1} = T_i$ and let $S_{i+1}$ be the diagonal matrix obtained from $S_i$ by adding a column to scale column $e_i$ of $T_{i+1} A_{i+1}$ so that its entry in row $u_i$ is 1.
Extend the $\FF^\times$-gain function $\vp_i$ by defining $\vp_{i+1}(e_i)$ to be $-(T_{i+1} A_{i+1} S_{i+1})_{v_i e}$.
Now $T_{i+1} A_{i+1} S_{i+1}$ is the canonical frame matrix $A_F(G_{i+1}, \vp_{i+1})$.
By induction, there is a nonsingular matrix $T_n$, a diagonal matrix $S_n$, and a gain function $\vp_n$ such that  $T_n A_n S_n = A_F(G_n,\vp_n)$.

If $\GB$ has no joints we are done.
So assume $J$ is nonempty.
We now claim that $M \neq L\GB$.
For suppose contrarily that $M=L\GB$.
Since $T_nA_nS_n$ is
a canonical frame representation of $M \bs J$, it must be the case that $\Om_n$ is tangled.
By Theorem \ref{T:TangledMinor} $\Om_n$ contains a link minor $(H,\Ss)$ that is either a biased $2C_3$ with no balanced 2-cycle or a biased $K_4$ with no balanced triangle.
Since $J$ is nonempty $\GB$ has a link minor $(H',\Ss')$ where $(H',\Ss')$ is obtained by adding a joint $e$ incident to a vertex of $(H,\Ss)$.
Since $T_n A S_n$ is a representation over $\FF$ for $L\GB$ that agrees with $T_n A_n S_n$ on all elements aside from possibly those in $J$, and since $(H',\Ss')$ is a link minor of $\GB$, by Lemma \ref{lem:gain_graph_and_canonical_matrix_minors}
there is a matrix $B$ over $\FF$ representing $L(H',\Ss')$ with the property that $B \bs e$ is a canonical frame matrix particular to $(H,\Ss)$.
But this is impossible by Lemma \ref{L:TangledDoesntExtend}.
Thus $M \neq L\GB$.

Finally, we show that each column $e$ of $T_n A$ for which $e \in J$ has exactly one nonzero entry.
Suppose $e \in J$ has endpoint $u$ and that column $e$ of $T_n A$ is nonzero in row $v \neq u$.
Let $C$ be the edge set of an unbalanced cycle in $\Om_n - v$ together with a path linking this cycle and $u$.
Then $C$ is a circuit of $F\GB$ but its corresponding columns in $T_n A$ are linearly independent, a contradiction.
Thus there is a diagonal matrix $S$ scaling the columns of $T_n A$ such that $T_n A S$ is a canonical frame matrix particular to $\GB$.

(2)
We proceed as in case (1), considering the matrix $T_i A_{i+1}$.
Matrix $A_{i+1}$ has rows indexed by $V(G) \cup v_0$, according to the indexing of the corresponding rows of $T_i A_i S_i$ where $v_0 \notin V(G)$ corresponds to the ``gains row'' of $T_i A_{i} S_{i}$, and columns indexed by $E(G_{i+1})$.
We first show that column $e_i$ of $T_iA_{i+1}$ is zero in every row aside from $u_i$, $v_i$, and $v_0$.
Suppose for a contradiction that $e_i$ is nonzero in a row $x \notin \{u_i, v_i, v_0\}$ of $T_i A_{i+1}$.
As in case (1), let
$U_i \subseteq E(\Om_i-x)$ be a set of edges inducing a subgraph of $\Om_i$ that is a spanning tree of $\Om_i-x$ along with one additional edge whose fundamental cycle with respect to this tree is unbalanced.
Contained in $U_i \cup e_i$ is a biased subgraph $C$ whose edge set is a circuit of $L(\Om_{i+1})$.
The subgraph $C$ contains $e_i$ and does not contain $x$.
Since $T_i A_i S_i = A_L(G_i,\psi_i)$, the columns of $T_i A_{i+1}$ representing $E(C)-e_i$ are all zero in row $x$.
Hence while $C$ is a circuit of $L(\Om_{i+1})$ the columns of $T_i A_{i+1}$ representing $C$ are linearly independent, a contradiction.

Since $\Om_i$ is unbalanced and 2-connected and $\Om_{i+1}$ is obtained from $\Om_i$ by adding a single edge,  $r(L(\Om_{i+1})) = r(L(\Om_i))$.
Thus $A_i$ and $A_{i+1}$ have the same rank.
Since row $v_0$ of $T_i A_i S_i$ is not in the span of the rows $V(G)$ neither is row $v_0$ of $T_i A_{i+1}$ in the span of the rows in $V(G)$.
But the sum of the rows in $V(G)$ of $T_i A_i$ is zero, so likewise the sum of the rows in $V(G)$ of $T_i A_{i+1}$ must be zero:
otherwise the rank of $T_i A_{i+1}$ would be greater than that of $T_i A_i$, a contradiction.

Suppose first that both entries of column $e$ in rows $u_i$ and $v_i$ are zero.
Element $e_i$ is not a loop of $M$, so then its entry in row $v_0$ is nonzero.
Let $C$ be an unbalanced cycle in $\Om_i - u_i$.
Then $C \cup e$ is independent in $L(\Om_{i+1})$ but the columns of $T_i A_{i+1}$ representing $C \cup e$ are linearly dependent, a contradiction.
Thus column $e_i$ has entries $a$ and $-a$ in rows $u_i$ and $v_i$, where $a \neq 0$.
Take $T_{i+1} = T_i$ and let $S_{i+1}$ be the diagonal matrix obtained by adding a column to $S_i$ to scale column $e_i$ of $T_{i+1} A_{i+1}$ by $a\inv$.
Extend the $\FF^+$-gain function $\psi_i$ to $E(G_{i+1})$ by defining $\psi(e_i)$ to be the entry in row $v_0$ of column $e_i$.
Now $T_{i+1} A_{i+1} S_{i+1}$ is the canonical lift matrix $A_L(G_{i+1}, \psi_{i+1})$.
By induction, there is a nonsingular matrix $T_n$, a diagonal matrix $S_n$, and a gain function $\psi_n$ such that  $T_n A_n S_n = A_L(G_n,\psi_n)$.

If $\GB$ has no joints we are done.
So assume $J$ is nonempty.
Analogous to the situation in case (1), we now claim that $M \neq F\GB$.
Suppose to the contrary that $M = F\GB$.
Since $T_nA_nS_n$ is a canonical lift representation of $M \bs J$, it must be the case that $\Om_n$ is tangled.
By Theorem \ref{T:TangledMinor}, $\Om_n$ contains a link minor $(H,\Ss)$ that is either a biased $2C_3$ with no balanced 2-cycle or a biased $K_4$ with no balanced triangle.
Since $J$ is nonempty $\GB$ has a link minor $(H',\Ss')$ where $(H',\Ss')$ is obtained by adding a joint $e$ incident to a vertex of $(H,\Ss)$.
Since $T_n A S_n$ is a representation over $\FF$ for $F\GB$ that agrees with $T_n A_n S_n$ on all elements aside from possibly those in $J$, and since $(H',\Ss')$ is a link minor of $\GB$, by Lemma \ref{lem:gain_graph_and_canonical_matrix_minors}
there is a matrix $B$ over $\FF$ representing $F(H',\Ss')$ with the property that $B \bs e$ is a canonical lift matrix particular to $(H,\Ss)$.
This violates Lemma \ref{L:TangledDoesntExtend}, so $M \neq F\GB$.

Finally, we show that each column $e$ of $T_n A$ for which $e \in J$ has a nonzero entry only in row $v_0$.
Suppose for a contradiction that $e \in J$ with column $e$ of $T_n A$ nonzero in row $v \neq v_0$.
Let $C$ be the edge set of an unbalanced cycle in $\Om_n - v$.
Then $C \cup e$ is a circuit of $L\GB$ but its corresponding columns in $T_n A$ are linearly independent, a contradiction.
Thus $T_n A S_n$ is a canonical lift matrix particular to $\GB$.

This completes the proof that at least one of statements (i) or (ii) of the theorem hold.
But $A_0$ is not projectively equivalent to both a canonical frame matrix and a canonical lift matrix particular to $(G_0,\Bb_0)$.
Thus neither is $A$ projectively equivalent to both a canonical frame matrix and a canonical lift matrix particular to $\GB$.
\end{proof}

The first statement of Theorem \ref{thm:all_reps_eqiv_to_cannonical} follows immediately from Theorem \ref{T:MainTheorem1}.
For the second statement of Theorem \ref{thm:all_reps_eqiv_to_cannonical}, we would like to show that for each 3-connected matroid $M$ represented by an almost-balanced biased graph $\GB$, given any matrix $A$ representing $M$, $A$ is projectively equivalent to a canonical lift matrix particular to $\GBu$.
Unfortunately, this can fail in the case that $M$ has rank 2.
Let $M$ be a 3-connected rank-2 matroid represented by the biased graph $(G,\emptyset)$ consisting of $E(M)-1$ links between a pair of vertices and a single joint.
Let $\FF$ be a field and let $A$ be a matrix over $\FF$ representing $M$.
Then $A$ is projectively equivalent to the matrix
\begin{linenomath}
\begin{equation*} \label{eqn:a_rank_two_matrix}
A' =
\begin{pmatrix}
1 & 0 & 1 & 1 & \cdots & 1 \\
0 & 1 & a_1 & a_2 & \cdots & a_{n-2}
\end{pmatrix}
\end{equation*}
\end{linenomath}
where $n = |E(M)|$, for each $i$, $a_i \notin \{0,1\}$, and the $a_i$ are distinct.
The matrix $A'$ is a canonical frame representation particular to a roll-up of $(G,\emptyset)$, and $A'$ is a full-rank canonical lift matrix particular to $(G,\emptyset)$ (with its second row as the ``gains row'' indexed by $v_0$, and "missing" its row indexed by one of the vertices of $G$).
It is straightforward to see that by elementary row operations and column scaling we may obtain from $A'$ a canonical frame matrix particular to $(G,\emptyset)$.
However, $(\widehat{G},\widehat{\emptyset})$ is the loopless contrabalanced biased graph obtained from $(G,\emptyset)$ by unrolling its joint, and there is no guarantee that $A$ need be projectively equivalent to a canonical matrix representation particular to $(\widehat{G},\widehat{\emptyset})$.
Indeed, there is no guarantee that such a representation exists.
For instance, if $n=6$ and $\FF$ is $\mathrm{GF}(5)$, then the above argument shows that $A$ is projectively equivalent to a canonical frame matrix particular to a roll-up of $(G,\emptyset)$, and to a canonical lift matrix particular to $(G,\emptyset)$, but a canonical lift representation particular to $(\widehat{G},\widehat{\emptyset})$ requires six distinct elements of $\FF^+$, so no such matrix exists.
The problem is that the field is just too small by one element to permit a canonical lift representation of a such a rank-2 matroid,
Fortunately, this is the only problem that may occur:

\begin{thm}
\label{thm:MainThm_bal_vertex_case}
Let $M$ be a matroid of rank greater than two represented by a 2-connected almost-balanced biased graph $\GB$ having no 2-separation with one side balanced.
Let $\FF$ be a field and let $A$ be an $\bb F$-matrix representing $M$.
Then
\begin{enumerate}[label=\textup{(\roman*)}]
\item $A$ is projectively equivalent to a canonical lift matrix particular to $\GBu$, and
\item $A$ is projectively equivalent to a canonical frame matrix particular to each roll-up of $\GBu$, unless $\FF$ is $\mathrm{GF}(2)$.
\end{enumerate}
Furthermore, $A$ is projectively equivalent to a canonical frame matrix particular to $\GBu$ if and only if whenever $\vp$ is an $\FF^+$-gain function for which $A$ and $A_L(\widehat{G},\vp)$ are projectively equivalent, $\vp$ is switching equivalent to a gain function assigning $0$ to $e$ if and only if $e$ is a link not incident to $u$ or $e$ is a loop of $M$.
\end{thm}

We will need the following straightforward fact.

\begin{lem} \label{lem:D10deltaYT2prime}
Let $C$ be the balanced triangle of $D_{1,0}$ and let $Y$ be a $K_{1,3}$-subgraph of $D_{1,0}$ meeting $C$ in exactly two edges.
Then
the simplification of $\nabla_Y \mathsf D_{1,0}$ is isomorphic to $\mathsf B_1'$.
\end{lem}

\begin{proof}[Proof of Theorem \ref{thm:MainThm_bal_vertex_case}]
Suppose first that $\GBu$ does not contain a contrabalanced theta subgraph.
The only  biased graph representing $U_{2,4}$ without a contrabalanced theta is shown at right in Figure \ref{fig:U24unlabelled}.
Moreover, $F_7$ is neither frame nor lifted-graphic, and the only biased graphs representing $F_7^*$, $M^*(K_5)$, and $M^*(K_{3,3})$ are all properly unbalanced \cite{MR1058551}.\footnote{The mistake in \cite{MR1058551}, corrected at \url{http://people.math.binghamton.edu/zaslav/Tpapers/index.html} does not affect this claim.}
Since none of these biased graphs can occur as a minor of $\GBu$ and $\GBu$ represents $M$, $M$ contains none of $U_{2,4}$, $F_7$, $F_7^*$, $M^*(K_5)$, nor $M^*(K_{3,3})$ as a minor.
Thus $M$ is graphic.
The first two statements follow. The third statement follows from
Corollary \ref{cor:bal_vertex_frame_and_lift_proj_equiv}.

So assume now that $\GBu$ contains a contrabalanced theta subgraph.
Let $v$ be a balancing vertex of $\GBu$ and let $J$ be the set of joints of $\GBu$.
Since the set of joints not incident to $v$ form an unbalancing class of $\Sigma(v)$ in $\GB$, and all of the unbalancing classes of $\Sigma(v)$ are unrolled in $\GBu$, every joint of $\GBu$ is incident to $v$.
We claim that $\GBu \bs J$ does not have another balancing vertex $u \neq v$.
For suppose to the contrary that $\GBu \bs J$ has a balancing vertex $u \neq v$.
Then $\GBu \bs J$ has the structure
described in Proposition \ref{prop:Zas_fattheta}: there are graphs $G_1, \ldots, G_m$ such that $\widehat{G} = G_1 \cup \dots \cup G_m$ and $G_j \cap G_k = \{u,v\}$ for each pair $j \neq k$ and a cycle is in $\widehat{\Bb}$ if and only if it is contained in a single graph $G_j$.
Since $\GBu$ contains a contrabalanced theta, $m \geq 3$.
If there is a subgraph $G_i$ with $E(G_i) \geq 2$, then $(E(G_i), E(G) \bs E(G_i))$ is a 2-separation of $\GB$ with one side balanced, contrary to assumption.
Thus for each $i$, $|E(G_i)| =1$.
Thus $G = mK_2$ and $\Bb$ is empty.
But $M(mK_2,\emptyset)$ is isomorphic to the $m$-point line $U_{2,m}$, so $r(M)=2$, contrary to assumption.

Thus $v$ is the unique balancing vertex of $\GBu \bs J$.
Since $G$ is 2-connected, so is $\widehat{G}$.
Thus by Proposition \ref{P:UniqueBalancingVertex} $\GBu$ contains a biased subgraph $\Om_0$ that is a subdivision of $\mathsf D_{1,0}$, $\mathsf B_0'$, $\mathsf B_1'$, or $\mathsf B_2'$.
Since both $\Om_0$ and $\widehat{G}$ are 2-connected, there is a sequence of 2-connected biased subgraphs
$\Om_0 \subset \cdots \subset \Om_n$ where $\Om_n = \GBu \bs J$,
and for each $i \in \{0, \ldots, n-1\}$ there is a path $P_i$ in $G$ internally disjoint from $\Om_i$ so that $\Om_{i+1} = \Om_i \cup P_i$.
For each $i \in \{0,\ldots,n\}$, let $(G_i,\Bb_i)$ be the biased subgraph of $\GBu$ induced by $E(\Om_i)$ with $V(G_i)=V(\widehat{G})$ and let $A_i$ be the submatrix of $A$ consisting of all rows of $A$ and precisely those columns representing $E(\Om_i)$.
Thus for each $i$, $A_i$ represents $M(\Om_i)$.
By Lemmas \ref{P:Subdivisions} and \ref{P:ContractedTubeCanonical}, Proposition \ref{P:Whittle}, and Lemma \ref{lem:D10deltaYT2prime}, $A_0$ is projectively equivalent to a canonical lift matrix particular to $\Om_0$.
Inductively assume that there are sequences $T_0, \ldots, T_i$ and $S_0, \ldots, S_i$ of $\FF$-matrices such that for each $k \in \{0, \ldots, i\}$ each matrix $T_kA_kS_k$ is a
canonical lift matrix particular to $\Om_k$.
We may assume that $A$ is of full rank, and so that
for each $k \in \{0,\ldots,i\}$ the rows of $T_k A_k S_k$ are indexed by $v_0 \cup \br{V(G_k)-v}$ (as described in Section \ref{sec:on_canonical_lift_representations}).
Since for each $k$, $V(G_k)=V(\widehat{G})$, for each $k$ the rows of $T_k A_k S_k$ are indexed by $v_0 \cup (V(\widehat{G})-v)$.

Consider $\Om_{i+1} = \Om_i \cup P_i$ and the matrix $T_i A_{i+1}$ representing $M(\Om_{i+1})$.
Let us assume $P_i$ consists of a single edge $e_i$ whose endpoints are $x_i, y_i \in V(\Om_i)$.
Let $x \in V(G_i) - \{x_i, y_i\}$.
Since $\Om_i$ is 2-connected, there is a spanning tree $T$ of $\Om_i - x$.
If $x$ is not $v$, then there is an edge $f$ such that the fundamental cycle in $T \cup f$ is unbalanced in $\Om_i$: set $W = E(T) \cup f$; otherwise set $W = E(T)$.
The subgraph $W \cup e_i \subseteq \Om_{i+1}$ contains a subgraph $C$ that is either a balanced cycle, a pair of unbalanced cycles sharing just vertex $v$, or a contrabalanced theta.
Since $E(C)$ is a circuit of $M(\Om_{i+1})$, this implies that in column $e_i$ of $T_i A_{i+1}$, the entry in row $x$ is 0.
As long as neither endpoint of $e_i$ is $v$, this also implies that the entries in rows $x_i$ and $y_i$ are nonzero.
If $e_i$ has endpoints $v, y_i$, for some vertex $y_i \neq v$,
then the form of $C$ in $\Om_{i+1}$ implies that the entry in matrix $T_i A_{i+1}$ in column $e_i$, row $y_i$ must be nonzero.
Thus $T_i A_{i+1} S_{i+1}$, with rows indexed by $v_0 \cup \br{V(\widehat{G})-v}$, is a canonical lift matrix for some appropriate column scaling matrix $S_{i+1}$.
Hence by induction, there are matrices $T_n$ and $S_n$ such that $T_n A_n S_n$ is
a canonical lift matrix particular to $\Om_n = \GBu \bs J$.

Finally, consider the set of joints $J$.
Let $e$ be a joint.
By assumption $e$ is incident to $v$ and every other joint is in parallel with $e$.
Since $v$ is the unique balancing vertex of $\GBu$, for every vertex $x \neq v$, there is an unbalanced cycle $C_x$ in $G-x$ of length $>1$.
Since $C_x \cup e$ is a circuit of $M$, row $x$ of column $e$ of $T_n A$ is zero.
Thus every row of column $e$ aside from row $v_0$ is zero.
Since $e$ is not a loop of $M$, the entry in row $v_0$ of column $e$ must be nonzero.
Since all joints of $\GBu$ are in parallel with $e$, every column of $A$ representing a joint is zero in all rows but $v_0$.
Thus there is a diagonal matrix $S$ scaling the columns of $T_n A$ so that $T_n A S$ is a canonical lift matrix particular to $\GB$.
This completes the proof of statement (i).

(ii) Let $(H,\Ss)$ be a roll-up of $\GBu$, and let $U \subseteq \Sigma(v)$ be the unique unbalancing class of edges in $\Sigma(v)$ that are joints in $(H,\Ss)$.
By statement (i) there is a $\FF^+$-gain function $\vp$ on $\widehat{G}$ realizing $\widehat{\Bb}$ for which $A=A_L(\widehat{G},\vp)$.
Let $T$ be a spanning tree of $\widehat{G}$ containing exactly one edge in $U$.
Then the $T$-normalized gain function $\vp'$ obtained by switching on $\vp$ satisfies $\vp'(e)=0$ for all edges $e$ not incident to $v$ and $\vp'(e)=0$ for each edge in $U$.
By Corollary \ref{cor:bal_vertex_frame_and_lift_proj_equiv},
the derived $\FF^\times$-gain function $\vp'^\times$ realizes $(H,\Ss)$.
By Lemma \ref{lem:bal_vertex_frame_and_lift_proj_equiv}, $A$ and $A_F(H,\vp'^\times)$ are projectively equivalent.

The final statement follows immediately from Corollary  \ref{cor:bal_vertex_frame_and_lift_proj_equiv}.
\end{proof}

\subsection{Projective equivalence classes are in 1-1 correspondence with switching classes}

Finally, we can prove Theorem \ref{mainthm1to1correspondence}.

\begin{proof}[Proof of Theorem \ref{mainthm1to1correspondence}]
By Proposition \ref{P:switch_equiv_implies_proj_equiv}, every switching class of gain functions is contained in a projective equivalence class of matrix representations.
Thus we just need show that if $A$ and $B$ are projectively equivalent representations of $M$, then $A$ and $B$ are each projectively equivalent to canonical representations whose gain functions are contained in the same switching class.
Because $M$ is 3-connected, $\GB$ is 2-connected and has no 2-separation with one side balanced.

Assume first that $\GB$ is properly unbalanced and not tangled.
Then either $M=F\GB$ or $M=L\GB$, but not both.
Assume that $M=F\GB$.
Let $A$ and $B$ be projectively equivalent $\FF$-matrices representing $M$.
By Theorem \ref{thm:all_reps_eqiv_to_cannonical}(1), each of $A$ and $B$ are projectively equivalent to a canonical frame
matrix particular to $\GB$.
Let $\vp$ and $\psi$ be $\FF^\times$-gain functions
such that $A$ is projectively equivalent to $A_F(G,\vp)$ and $B$ is projectively equivalent to $A_F(G,\psi)$.
By Theorem \ref{T:ProjectiveIsSwitching}(1), $\vp$ and $\psi$ are switching
equivalent.
Similarly, if $M=L\GB$ and $A$ and $B$ are projectively equivalent $\FF$-matrices representing $M$, then
by Theorem \ref{thm:all_reps_eqiv_to_cannonical}(1), each of $A$ and $B$ are projectively equivalent to a canonical lift matrix particular to $\GB$.
Let $\vp$ and $\psi$ be $\FF^+$-gain functions such that $A$ is projectively equivalent to $A_L(G,\vp)$ and $B$ is projectively equivalent to $A_L(G,\psi)$.
By Theorem \ref{T:ProjectiveIsSwitching}(ii), $\vp$ and $\psi$ are switching-and-scaling equivalent.

Now assume that $\GB$ is properly unbalanced but tangled.
Then $L\GB$ and $F\GB$ coincide: $M=L\GB=F\GB$.
Let $A$ and $B$ be projectively equivalent $\FF$-matrices representing $M$.
By Theorem \ref{thm:all_reps_eqiv_to_cannonical}, each of $A$ and $B$ is projectively equivalent to a canonical lift matrix particular to $\GB$, or to a canonical frame matrix particular to $\GB$, but not both.
By Theorem \ref{T:ProjectiveIsSwitching}(iii), either $A$ and $B$ are both projectively equivalent to canonical lift matrices or both are projectively equivalent to canonical frame matrices.
In either case, by statement (i) or (ii) of Theorem \ref{T:ProjectiveIsSwitching}, the gain functions from which these canonical representations arise belong to the same switching class.

Finally, assume that $\GB$ is almost-balanced, and let $A$ and $B$ be projectively equivalent $\FF$-matrices representing $M$.
By Proposition \ref{prop:Zas_fattheta}, $\GB$ has a unique balancing vertex.
By Theorem \ref{thm:all_reps_eqiv_to_cannonical}, $A$ and $B$ are each projectively equivalent to a canonical lift matrix particular to $\GBu$, say, given by $\FF^+$-gain functions $\vp$ and $\psi$ respectively.
By Theorem \ref{thm:proj_equiv_iff_switching_equiv_bal_vertex}(i), $\vp$ and $\psi$ belong to the same switching class.
\end{proof}

\subsubsection*{Acknowledgement}
We thank the referees for their careful reading and valuable suggestions. These vastly improved the paper.

\bibliographystyle{amsplain} \bibliography{Matroids1_bibfile}
\end{document}